\title{When Gr\"{u}nbaum meets Poncelet\\
{\Large -- Infinite Classes of Movable $(n_4)$ Configurations --}\\
}
\author{ \it Leah Wrenn Berman, G\'{a}bor G\'{e}vay, \\ \it J\"{u}rgen Richter-Gebert, Serge Tabachnikov}

\documentclass{article}
\usepackage{graphicx}
\usepackage{amsthm}
\newtheorem*{theorem*}{Theorem}
\newtheorem*{theoremA*}{Theorem A}
\newtheorem*{theoremB*}{Theorem B}
\newtheorem*{theoremC*}{Theorem C}

\newtheorem{construction}{Construction}
\newtheorem{theorem}{Theorem}
\newtheorem{notheorem}{Not--a--Theorem}

\newtheorem{lemma}{Lemma}[section]

\newtheorem{definition}{Definition}

\usepackage{amssymb}
\usepackage[font=small,labelfont=bf]{caption} 
\usepackage{color}
\usepackage{mathtools}
\usepackage{amsmath}
\usepackage[dvipsnames]{xcolor}
\usepackage[all]{nowidow}
\usepackage{float}

\usepackage{soul} 
\usepackage[normalem]{ulem} 

\usepackage[color links=true, backref=page]{hyperref} 

\hypersetup{
  colorlinks,
  citecolor=blue,
  linkcolor=Red,
  urlcolor=Blue
  }

\date{\vspace{-2ex}}

\begin{document}
\newpage
\maketitle





\def\upto{, \ldots ,}
\def\lines{\square}
\def\points{\bigcirc}
\def\lines{\vee} 
\def\points{\wedge}
\parskip=1mm
\def\meet{\wedge}
\def\join{\vee}

\def\pts[#1]{\wedge_{#1}} 
\def\lns[#1]{\vee_{#1}} 

\setlength{\parindent}{.5cm}

\newcommand{\del}[1]{\textcolor{red}{\sout{#1}}}
\newcommand{\new}[1]{\textcolor{ForestGreen}{#1}}

\begin{abstract}
We study relations between $(n_4)$ incidence configurations and the classical Poncelet Porism. Poncelet's result studies two conics and a sequence of points and lines 
that inscribes one conic and circumscribes the other. Poncelet's Porism states that whether  this sequence  closes up  after $m$ steps only depends on the conics 
and not on the initial point of the sequence. In other words: Poncelet polygons are movable. 

We transfer this motion into a flexibility statement about a large class of 
$(n_4)$ configurations, which are configurations where 4 (straight) lines pass through each point and four points lie on each line. A first 
instance of such configurations in real geometry 
had been given by Gr\"{u}nbaum and Rigby in their classical 1990 paper where they 
constructed the first known real geometric realisation
of a well known
combinatorial $(21_4)$ configuration
(which had been studied by Felix Klein),
now called  the Gr\"{u}nbaum--Rigby configuration.

Since then, there has been an intensive search 
for movable $(n_4)$ configurations, 
but it is very surprising that the Gr\"{u}nbaum--Rigby
$(21_4)$ configuration  admits nontrivial motions. It is well-known that the Gr\"unbaum--Rigby configuration is the smallest example of an infinite class of 
 $(n_4)$ configurations, the \emph{trivial celestial configurations}. A major result of this paper is that we show that all {\it trivial celestial configurations} are movable via Poncelet's Porism and results about properties of Poncelet grids. 
 
Alternative approaches via geometry of billiards, in-circle nets, and pentagram maps that relate the subject to 
discrete integrable systems are given as well. 
\end{abstract}

\tableofcontents

\parskip=2mm
\newpage
\section{Introduction}

\rightline{\small \it It is a good rule of thumb that if in the course of}
\rightline{\small \it  mathematical research an ellipse appears,}
\rightline{\small \it there is likely to be an interesting result nearby.}
\vskip1mm
\rightline{\small \it From {\it Finding Ellipses, \cite{DGSV18}}}

\vskip6mm

Some connections in mathematics come as a surprise.
This happened to us while exploring the possibility of nontrivial motions
of the classical  Gr\"{u}nbaum--Rigby $(21_4)$ configuration
\cite{GrRi90}. 
We asked if it is possible to fix four points of the configuration and move a fifth point in such a way that all the other incidences of the configuration are retained.
The existence of such a motion was conjectured by one of us in \cite{GePo23}. 
The $(n_4)$ configurations are collections of points and lines such that exactly 4 points lie on each line and, dually, four lines pass through each point. The proof of the construction of the Gr\"unbaum--Rigby configuration relies on the symmetries of a regular heptagon, so it seemed unlikely that it would exhibit nontrivial motions. Nevertheless, after finding a first construction for the Gr\"{u}nbaum--Rigby $(21_4)$ configuration that  provably 
exhibits such a nontrivial motion (for details of this construction see the companion paper~\cite{BGRGT24b}), a closer inspection of the construction exhibited structures that looked similar to the geometric situation 
in the classical Poncelet Porism \cite{Cen16a,Cen16b,Fl09,GriHa78,Pon1865}. 

Poncelet's Porism states that polygons that are simultaneously inscribed in a conic $\mathcal{A}$ and circumscribed around another conic $\mathcal{B}$ admit a one-dimensional 
space of motions while keeping their incidence properties with respect to $\mathcal{A}$ and~$\mathcal{B}$. In a sense, it looked like exactly this motion is responsible for the 
movability of the Gr\"unbaum--Rigby $(21_4)$ configuration and that the same technique could be applied to other configurations of similar type as well. It turned out that the Gr\"unbaum--Rigby $(21_4)$ configuration is only the tip 
of the iceberg and underneath lies an infinite class of movable $(n_4)$ configurations owing their non-trivial motion to Poncelet's Porism. This article exhibits the most general class 
that we were able to find that has this property: the {\it trivial celestial configurations} (see \cite[\S3.7 - 8]{Gr09}, under the name trivial \emph{$k$-astral} configurations). The term \emph{trivial} may be slightly misleading here. They are trivial in the sense that they 
arise in a kind of natural way that makes it easy to prove the underlying incidence property. But it is in particular this kind of ``arising in a natural way'' that makes them related to a Poncelet polygon.

In this article we show how the two concepts are related. We try to be constructive and explicit here. Whenever possible we give approaches that are related to constructions 
that (at least in principle) can generate drawings of the underlying configurations in a {\it generic} position. Generic in our context means that the constructions exploit the full 
range of flexibility that is available through Poncelet's Porism. Usually, this will be the choice of five independent points. Four of them fix a projective transformation and the remaining 
point controls two additional degrees of freedom.

Both concepts `$(n_4)$ configurations' as well as `Poncelet's Porism' are intrinsically 
concepts of projective geometry \cite{Cox92,Cox94,RG11}, since they only make statements about incidence and tangency relations of points, lines and conics and do not require 
measurement of genuine Euclidean relations like angles, distances or radii. 
Nevertheless, they reside in a rich area of related fields of classical and modern mathematics.  
Our approaches are also related to areas like 
discrete integrable systems \cite{BaFa91,SchTa10,IzTa17,Ta16},
geometry of billiards  \cite{IzTa17,LaTa07,Ta05,DraRa11},
pentagram maps \cite{Sch92,OST10},
incircle nets \cite{AkBo18},
elementary geometry,
and many more. In particular, we  demonstrate the relation to the dynamics of elliptical billiards and show how those considerations can lead to a totally different 
approach to the main results here.
Geometric constructions for Poncelet polygons that lead to explicit constructions for movable $(n_4)$ configurations
 are discussed on our companion paper~\cite{BGRGT24b}, where we also discuss
algebraic characterisations in terms of projectively invariant bracket expressions in the sense of
\cite{RG95,RG06,RG11}.

Since our work relates fields that are usually not considered in the same context, we decided to write this article to be as self-contained as possible. Also, we  explicitly introduce concepts 
in all relevant areas.  The reader may excuse the relatively long exposition for the benefit of (hopefully) better accessibility. All drawings in the article are made with the interactive 
geometry software Cinderella
\cite{RiKo99}, and experimenting within this software was crucial in the process of discovering the results in the paper. A collection of interactive animations related to this article can be found at:  \href{https://mathvisuals.org/Poncelet/}{\tt https://mathvisuals.org/Poncelet/}.

\subsection{Gr\"{u}nbaum--Rigby and other $(n_4)$ configurations}

To the best of our knowledge, the first mention of an $(n_4)$ configuration goes back to Felix Klein and his groundbreaking article  {\it Ueber die Transformationen siebenter 
Ordnung der elliptischen Funktionen} from 1878 \cite{Kl1878}.
Besides many other influential concepts (like, for instance, containing the first published image of a tiling in the Poincar\'{e} disk $\ldots$ before Poincar\'{e} invented the Poincar\'{e} disk) it contained 
the paragraph shown in Figure~\ref{fig:klein}. 

\begin{figure}[H]
\noindent
\begin{center}
\includegraphics[width=0.85\textwidth]{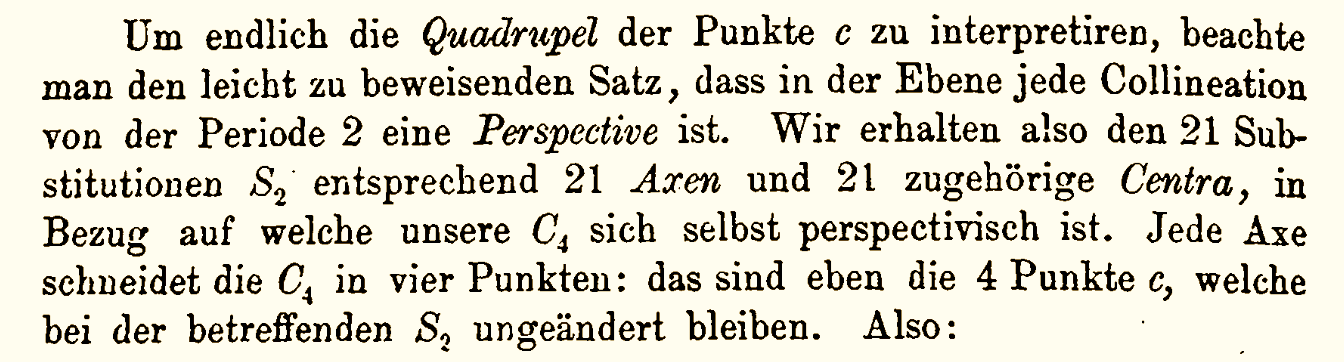}
\end{center}
\captionof{figure}{The first $(21_4)$ configuration mentioned in an article of Felix Klein.}~\label{fig:klein}
\end{figure}

There, Felix Klein describes a configuration that arises in the context of the algebraic curve $x^3y+y^3z+z^3x=0$, that has 21 points 
({\it Centra}) and 21 lines ({\it Axen}) such that on each line there are 4 points and through each point there are 4 lines. The configuration in Felix Klein's work is embedded in complex 
space, and there is no projective transformation that makes  all elements real, simultaneously.

It took over 110 years until an entirely real incidence configuration of the same combinatorial type was presented by Gr\"{u}nbaum and Rigby in their now classical paper from 
1990 \cite{GrRi90}. 
Although geometric $(n_{3})$ configurations had been studied since the mid-1800s, this was the first publication of a real realization of an $(n_{k})$ configuration with $k \geq 4$ that got serious attention.\footnote 
{One historic comment is appropriate here. As early as 1898 there was another mention of an  $(n_4)$ configuration, in fact over the real numbers, by the mostly unknown Hungarian mathematician   Leopold Klug. He described a
$(35_4)$ configuration 
that can be constructed as a union of two smaller configurations which are
duals of each 
other~\cite{Klu1898,  BoGePi15}. To the best of our knowledge this is the only place where real $(n_4)$ configurations have been discussed prior to the Gr\"{u}nbaum and Rigby $(21_4)$.}
It was the starting point for the search for more complex (real) incidence configurations that exhibit high degrees of symmetry 
and specific incidence properties and of the modern study of configurations. Figure~\ref{fig:GrRi} shows their original drawing. Since then, many articles have appeared exploring the 
existence of $(n_k)$ configurations for many parameters of $n$ (the number of points, which by counting equals the number of lines) and $k$ (the number of incidences per object). Since the number of publications on 
this topic is huge, here we only refer to Gr\"{u}nbaum's book \cite{Gr09}, which gives an overview of the topic until 2009.

\begin{figure}[h]
\smallskip
\noindent
\begin{center}
\includegraphics[width=0.65\textwidth]{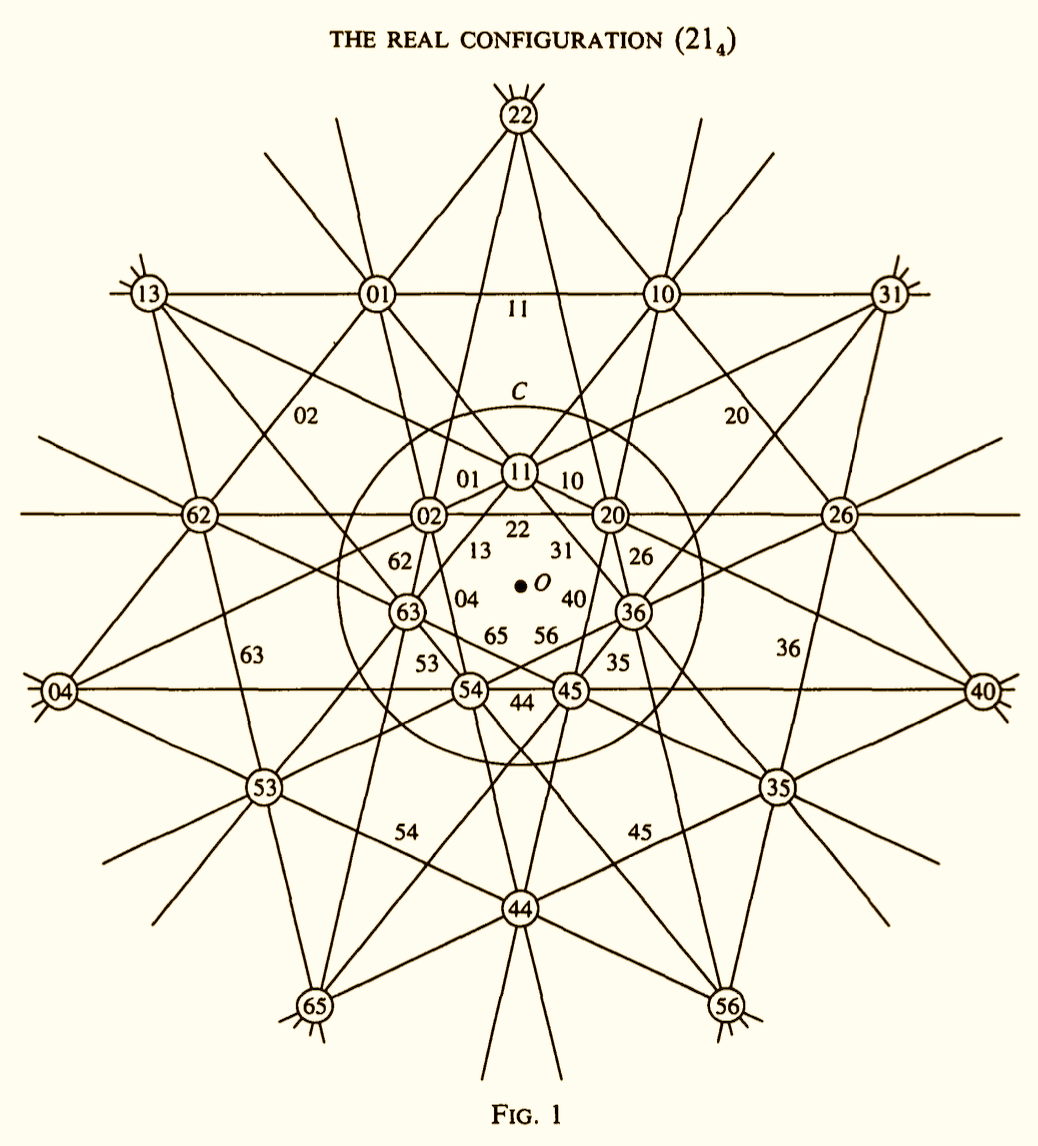}
\end{center}
\captionof{figure}{Gr\"{u}nbaum and Rigby's rendering of the real $(21_4)$  configuration \cite{GrRi90}, also exhibiting a circle of self-polarity.}~\label{fig:GrRi}
\end{figure}

Of special interest were those configurations that could be realised with rotational symmetry,  
and a class of configurations
was developed based on starting with the vertices
of a regular polygon and iteratively constructing various
diagonals and intersections of diagonals. These are called
celestial\footnote{Configurations constructed in this way have had various names, including \emph{$k$-astral, stellar} and \emph{celestial}. A \emph{$k$-astral} $(n_{4})$ configuration is defined \cite[p. 34]{Gr09} to have $k$ symmetry classes of points and lines.  At one point, most if not all of the known $(n_{4})$ configurations with $k$ symmetry classes were constructed as described below (also see $\cite[\S 3.7-8]{Gr09}$), but since there are also lots of other $(n_{4})$ configurations with $k$ symmetry classes that are \emph{not} constructed in this way, the literature has shifted to calling this sort of configuration \emph{celestial} rather than $k$-astral.} configurations~\cite{BerBer14}.

We will describe this process in detail in Section~\ref{sect:fromPoncelet}. In the notation 
describing such constructions, the  $(21_4)$ configuration above is denoted $7\#(3,1;2,3;1,2)$: Start with a regular 7-gon; connect pairs of points that are three steps 
apart by a line, of those lines intersect those that are one step apart; of those points  connect pairs that are 2 steps apart, etc.

In  \cite{GR00b,Gr00a, Gr09} criteria were provided determining under which conditions such a construction sequence leads to an $(n_4)$ configuration. Due to the high number of incidences and symmetry
in such a configuration, it was generally expected by researchers in the field that such configurations only admit trivial motions (those arising just from projective transformations). The main theorem of this 
article will be the fact that the contrary is the case for a huge class of naturally arising $(n_4)$ configurations (among them the Gr\"{u}nbaum--Rigby configuration). We will show 
that all configurations of type
\[
m\#(a_1,b_1;a_2,b_2\upto a_k,b_k)
\]
for which some mild non-degeneracy conditions hold and for which the multisets $a_i$ and $b_i$ are identical  admit 2 non-trivial degrees of movability. 

\subsection{Poncelet's Porism}
 Polygons also play a crucial role in Poncelet's Porism. This theorem is one of the earliest and at the same time deepest facts in projective geometry. It was discovered in  1813 while Poncelet was a prisoner of war and published in 1822~\cite{Pon1865}, at a time where projective geometry was just being developed.
Although the  statement of Poncelet's Porism is easy to make, its proof is by no means obvious, and it has far-reaching connections to many areas of mathematics like elliptic functions, integrable systems, algebraic geometry, dynamical systems, topology and many more.

\noindent
{\bf Poncelet's Porism:\ }{\it
Let $\mathcal{A}$ and $\mathcal{B}$ be two conics in the projective plane.
If $(p_1\upto p_{m})$ is a polygon whose vertices lie on $\mathcal{A}$  and whose edges are tangent to $\mathcal{B}$, then there exists such a  polygon starting with an arbitrary point on $\mathcal{A}$.
}

A polygon that is inscribed in $\mathcal{A}$ and circumscribed around $\mathcal{B}$ will be called a 
{\it Poncelet polygon} in the rest of this article. One could also think of this process constructively.
From a point  $p_1$  on $\mathcal{A}$
draw a tangent to $ \mathcal{B}$ and intersect this tangent with   $\mathcal{A}$. Take the intersection that is different from $p_1$ and call it $p_2$. From there, take the other tangent to $\mathcal{B}$ and iterate.
If such a sequence closes up after $m$ steps, it will do so for any starting point on $\mathcal{A}$. Viewing a Poncelet polygon as such a construction sequence sheds a light on both the algebraic and projective nature of Poncelet's Porism.
The conics may be located so that the intersection points or tangents become complex or lie at infinity. And indeed, Poncelet's theorem unfolds its full beauty and power when the situation is interpreted in the complex projective plane. Nevertheless, here we will usually consider situations where all elements of a Poncelet polygon are real. One is on the safe side for staying real if, for instance,  $\mathcal{A}$ and  $\mathcal{B}$
are a pair of nested ellipses with  $\mathcal{B}$ inside  $\mathcal{A}$.

\begin{figure}[h]
\smallskip
\noindent
\begin{center}
\includegraphics[width=0.45\textwidth]{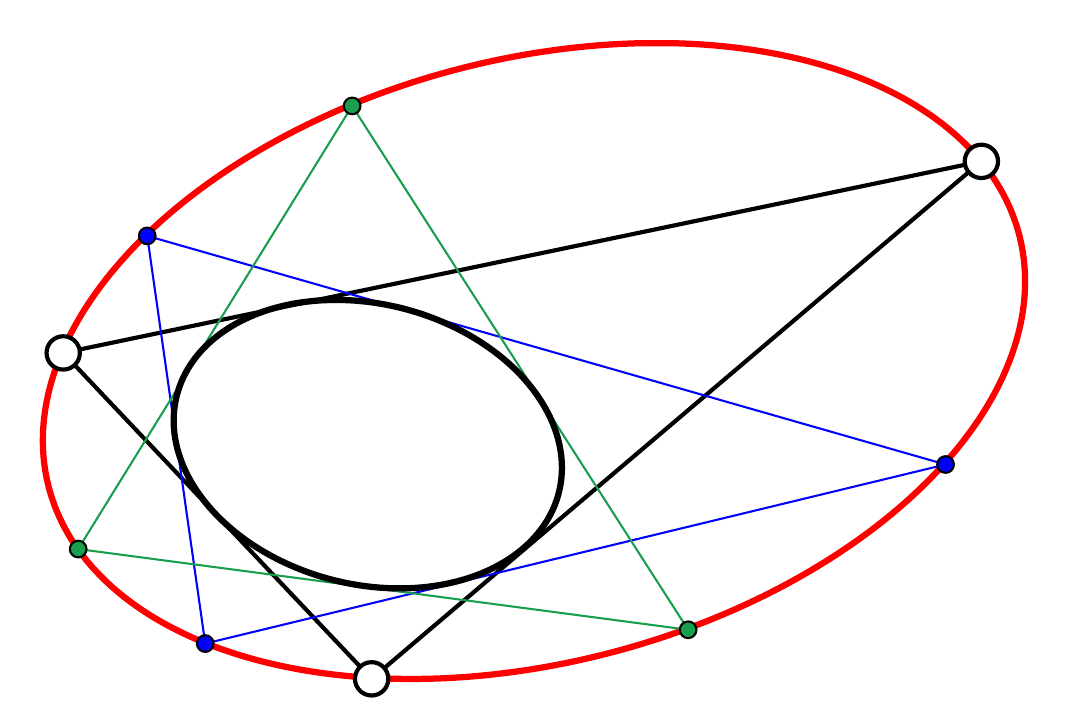}\hfill
\includegraphics[width=0.45\textwidth]{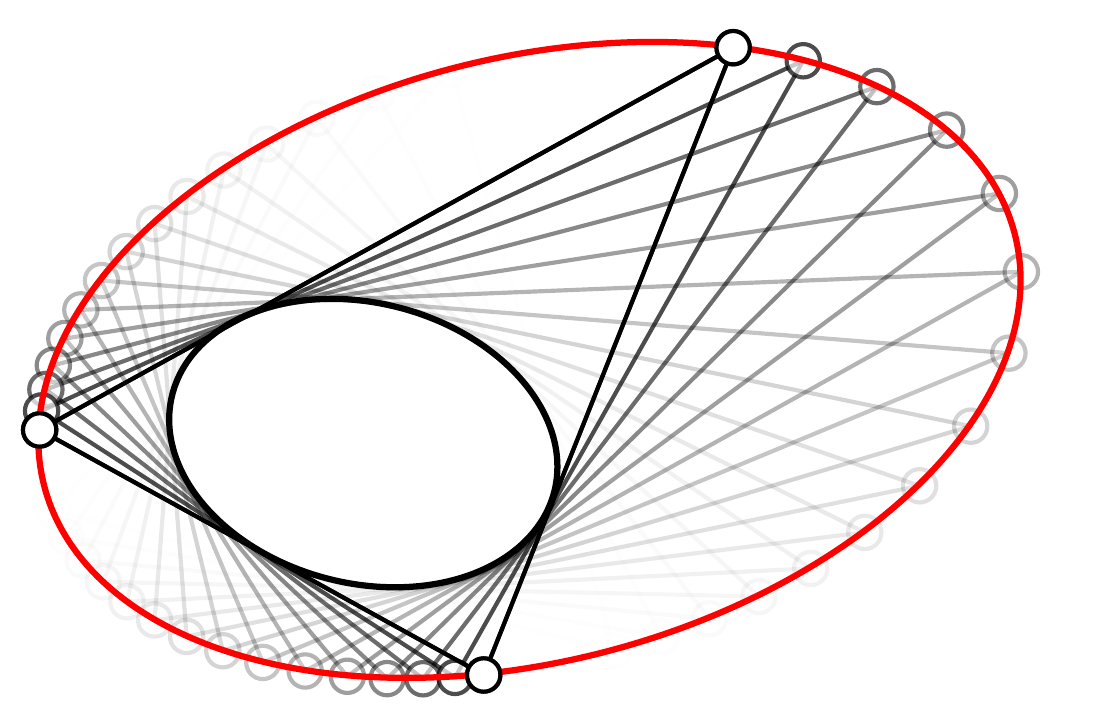}
\end{center}
\captionof{figure}{Illustrations of Poncelet's Porism in the case of a triangle}~\label{fig:Ponc}
\end{figure}

Figure~\ref{fig:Ponc} illustrates the situation for a Poncelet triangle. The picture on the left shows that the existence of one closing triangle (say the black one) implies the existence of others (blue and green). The picture on the right emphasises the dynamic nature of Poncelet's Porism: we may move the starting point along conic $\mathcal{A}$ and get a continuous family of Poncelet $m$-gons. Two things should be emphasised here:
The way the term $m$-gon is interpreted in the context of the Poncelet Porism does not imply convexity. Self-intersections and star-polygons are fully admissible. The second observation is that slightly 
changing the position of the conics will in general immediately lead to the Poncelet chain `breaking up' and no longer closing up to form an $m$-gon. For huge $m$, the exact position of the conics is numerically very sensitive. It turns out that  for each situation of two conics that support a Poncelet $m$-gon (up to projective transformations) there still is a one-parameter family of continuous movements that does not break  the Poncelet polygon. This one-parameter family, together with the movement of the point $p_1$ along the conic, are the two degrees of freedom that correspond to the two degrees of freedom for non-trivial motions of certain $(n_4)$ configurations. 

Although Poncelet's Porism is projective in nature, many of the classical proofs rely on an application of 
elliptic functions. Since there is enough good literature on that subject 
\cite{Ber87,Cen16a,Cen16b,DraRa11,GriHa78,HaHu15}, we will not present any proofs of Poncelet's Porism here.
However, we want to explicitly point to a relatively recent article by Halbeisen and Hungerb\"{u}hler \cite{HaHu15}
who give a proof of the Theorem that is entirely based on projective arguments and reduces Poncelet's Porism
to iterated applications of Pascal's and Brianchon's theorem. Such an approach is very much in the spirit of our  article. We also try to be as projective as possible when doing our constructions.

\subsection{Movable $(n_4)$ configurations}

The incidence situation of $(n_4)$ configurations shows that we have to expect relatively rigid objects.
A rough degree of freedom count shows the following. In an $(n_4)$  configuration there are $n$ points and $n$ lines. Each of them has two degrees of freedom. Subtracting 8 degrees of freedom for trivial actions by a planar projective transformation leaves us with $4n-8$ degrees of freedom if there were no incidences.
However, every point is incident to four lines, which eliminates $4n$ degrees of freedom. This leaves us with a degree of freedom count of $4n-4n-8=-8$. Such negative degrees of freedom must be compensated by the presence of geometric incidence theorems that create some of the incidences {\it for free}. So in general, one would expect that
$(n_4)$ configurations are relatively rigid objects, and indeed, many of them presumably are.

Nevertheless, there are several classes of $(n_4)$ configurations that are known to be movable. Some of them are movable even with the additional requirement of keeping rotational symmetry. 
Such classes were first constructed by one of us in~\cite{Ber06} providing examples of provably movable $(n_4)$ configurations. The smallest $n$ achieved in this first publication is $n=44$. 
Later \cite{Ber08} that bound was improved to $n=30$. One interesting aspect of the current work is that we improve this bound to $n=21$:
We show that  the classical Gr\"{u}nbaum--Rigby $(21_4)$ admits a non-trivial movement with two degrees of freedom. We also show that this is only the smallest representative of a huge class: 
the trivial $k$-celestial configurations. All of them turn out to have non-trivial movement. Formerly, the only constructions of trivial celestial configurations that were known were based on nesting
regular $n\over k$-gons and proving the requirements that lead to the $(n_4)$ properties by trigonometric calculations
on the angles involved. Our constructions show that $n\over k$-gons can be replaced by Poncelet $n\over k$-gons,
and the circles supporting the $n\over k$-gons become conics. We show that the trigonometric calculations can be replaced by 
arguments based on 
relationships between points and lines in a Poncelet grid.
At the end of Section~\ref{sect:examples} we even present a movable $(110_6)$ configuration,
whose construction can be generalised to show the existence of moveble $(n_k)$ configurations for
infinitely many values of $n$, depending on $k$.

\subsection{Overview of the paper}

In Section~\ref{sect:prelim} we present some important background about the geometry of conics and their relations to
Poncelet polygons and Poncelet grids
\cite{IzTa17,Sch07,Ta05,Ta19}. 
Section~\ref{sect:fromPoncelet} will prove our main theorem in a constructive way. 
We first show how the main result can be reduced  from $(n_4)$ configurations with $k$ rings of points to the situation where we only have $3$ rings of points. Then we will give an explicit procedure that starts with a Poncelet polygon and from that produces an $(n_4)$ configuration (with three rings). 
The impatient reader may right away jump to Figures~\ref{fig:star1} and \ref{fig:star2} to see this construction. The construction will produce $n$ lines. All incidences required to be an $(n_4)$ configuration 
are then demonstrated. After that we present  alternative proving approach  as using in-circle nets and in the language of elliptical billiards and discrete integrable systems, in the case where the Poncelet conics are confocal ellipses. 

\section{Preliminaries on Conics} \label{sect:prelim}

In this section we collect a few important facts about geometric objects and relations that are used throughout this article. In particular we focus on conics and their relation to Poncelet's Porism. 
We also discuss the role of real vs. complex realisations in the context of our work.

\subsection{Real vs.\ complex}

A note on the matter of real vs.\ complex is appropriate here.
We will deal with elements of projective geometry, which are points, lines, conics and projective transformations. 
We will often draw pictures in the real projective plane $\mathbb{RP}^2$. However, often it is algebraically more appropriate to consider the complex projective plane 
$\mathbb{CP}^2$ as the ambient space.
Conics will usually be represented as quadratic forms given by a symmetric matrix. It may easily happen that a real quadratic form only has complex solutions (when the matrix is positive definite). Although this conic has no points in $\mathbb{RP}^2$, algebraically this is still a valid geometric object in $\mathbb{CP}^2$.

It may also happen that intersections between conics may become complex. We even may get mixed situations in which two intersections are real and the other two are complex.

The main goal of this paper will be to derive statements about {\it real} configurations. Nevertheless, many of the incidence-theoretic statements along the way may be of an entirely algebraic nature and are best and most generally formulated over the complex numbers. In what follows we will draw real pictures whenever possible. Still, the statements should be considered over the complex numbers. Whenever we will specifically deal with a {\it complex}  configuration we will explicitly mention this.

\subsection{Pencils and co-pencils}

In what follows, we will consider  dependent and co-dependent sets of conics.
\begin{definition}
Three conics given by matrices $\mathcal{A}$, $\mathcal{B}$, $\mathcal{C}$ are called {\em dependent} if the corresponding matrices are linearly dependent. They are called 
{\it co-dependent} if their inverses $\mathcal{A}^{-1}$, $\mathcal{B}^{-1}$, $\mathcal{C}^{-1}$ (if they exist) are linearly dependent.
\end{definition}

The set of all matrices that can be derived as linear combinations of $\mathcal{A}$ and $\mathcal{B}$ is called the {\it pencil} through $\mathcal{A}$ and $\mathcal{B}$. These are all matrices of the form 
$\lambda \mathcal{A}+\mu \mathcal{B}$. If $\mathcal{A}$ and $\mathcal{B}$ represent different conics, all conics in the pencil $\lambda \mathcal{A}+\mu \mathcal{B}$ have four points in common
(counted with multiplicity and including complex solutions). Dually, we may define the \emph{co-pencil} of $\mathcal{A}$ and $\mathcal{B}$ as the set of all co-dependent matrices. Counted with multiplicity 
and including complex cases, conics in a co-pencil  have four tangents in common.
\medskip

\begin{figure}[H]
\smallskip
\noindent
\includegraphics[width=0.45\textwidth]{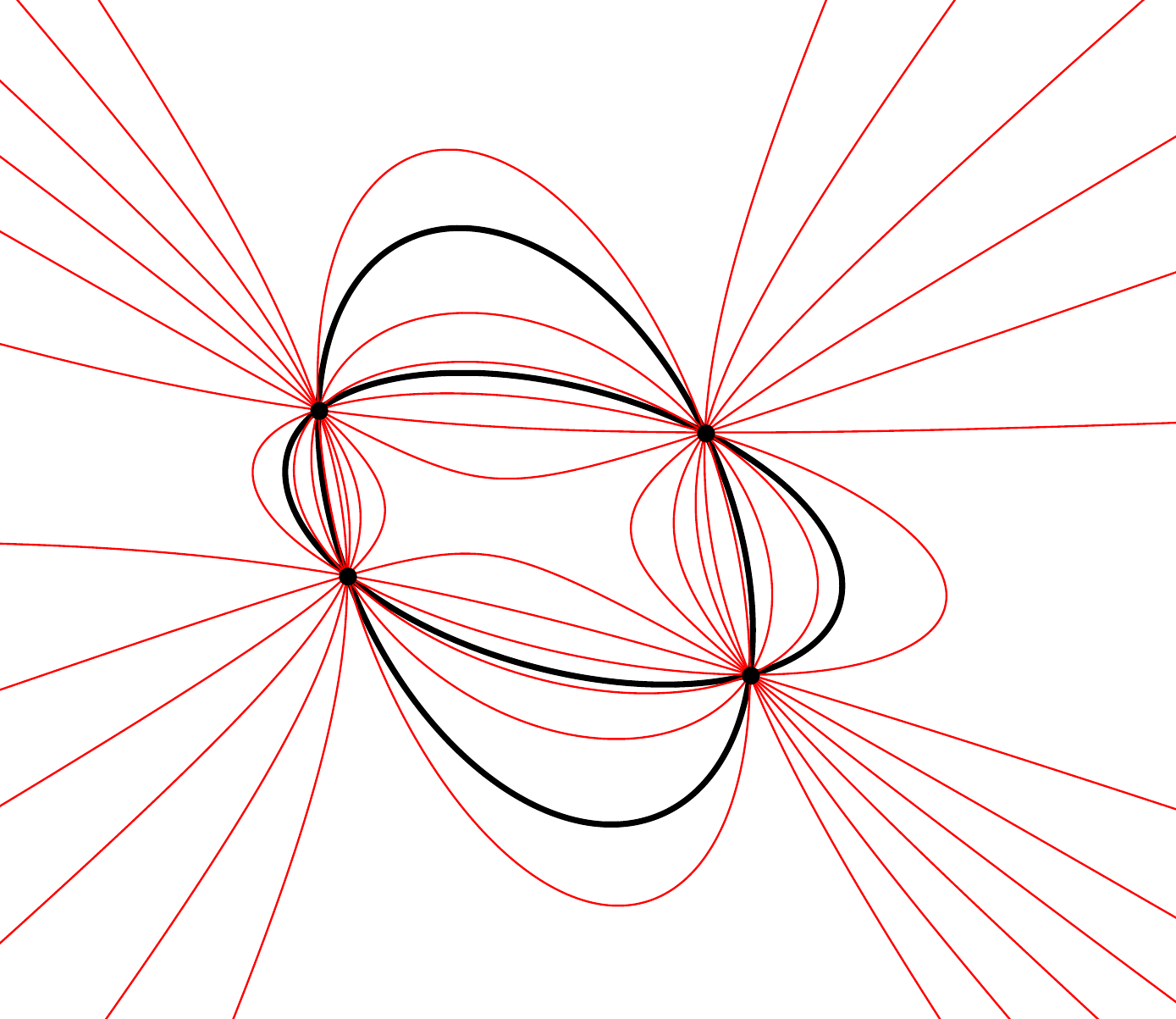}\hfill
\includegraphics[width=0.45\textwidth]{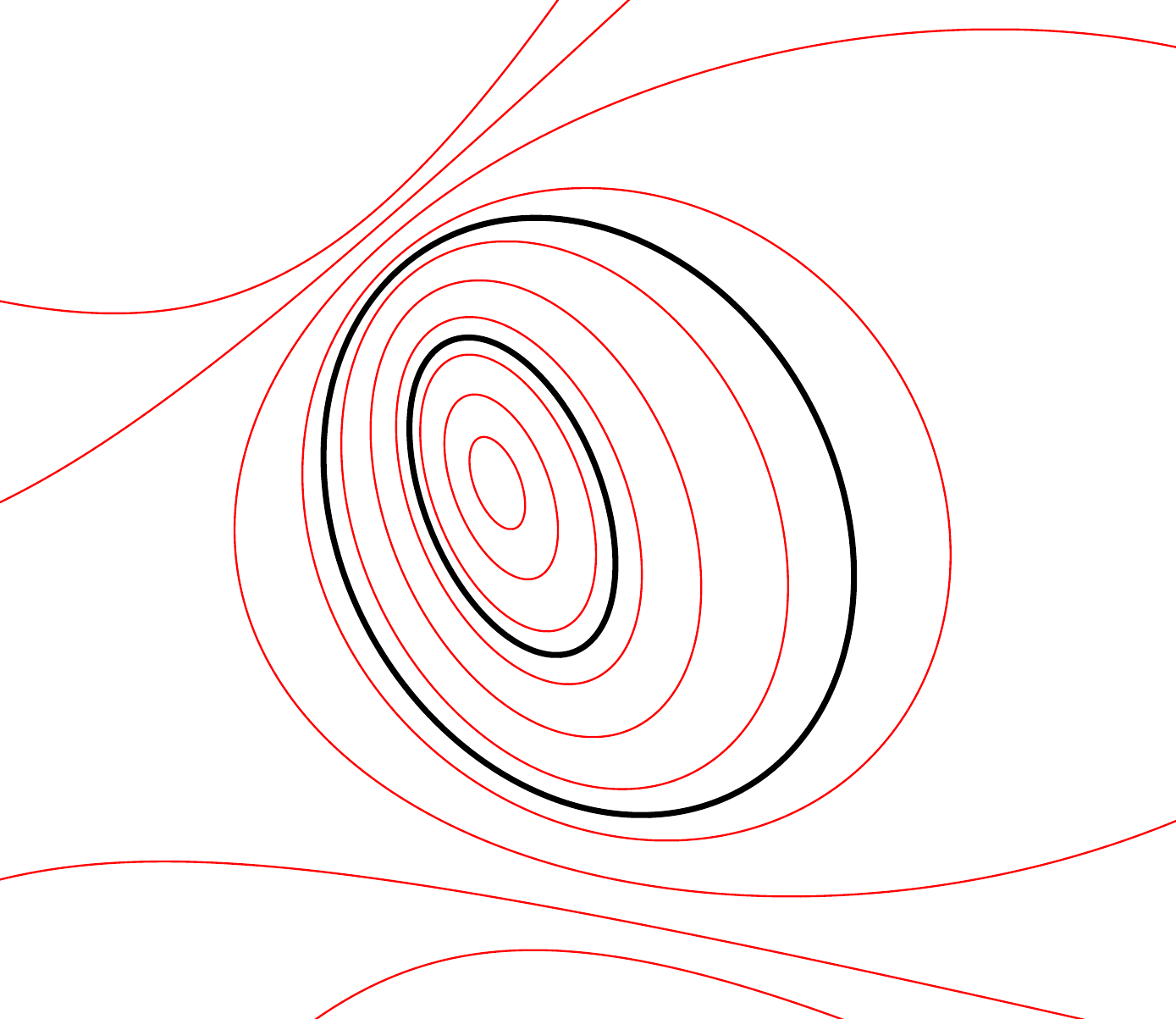}
\captionof{figure}{Some conics (red) from the pencil of dependent conics spanned by two other conics (black). 
                          On the left is the case of four real intersection points, on the right is the case of four complex intersection points.}
\end{figure}

\begin{figure}[h]
\noindent
\includegraphics[width=0.45\textwidth]{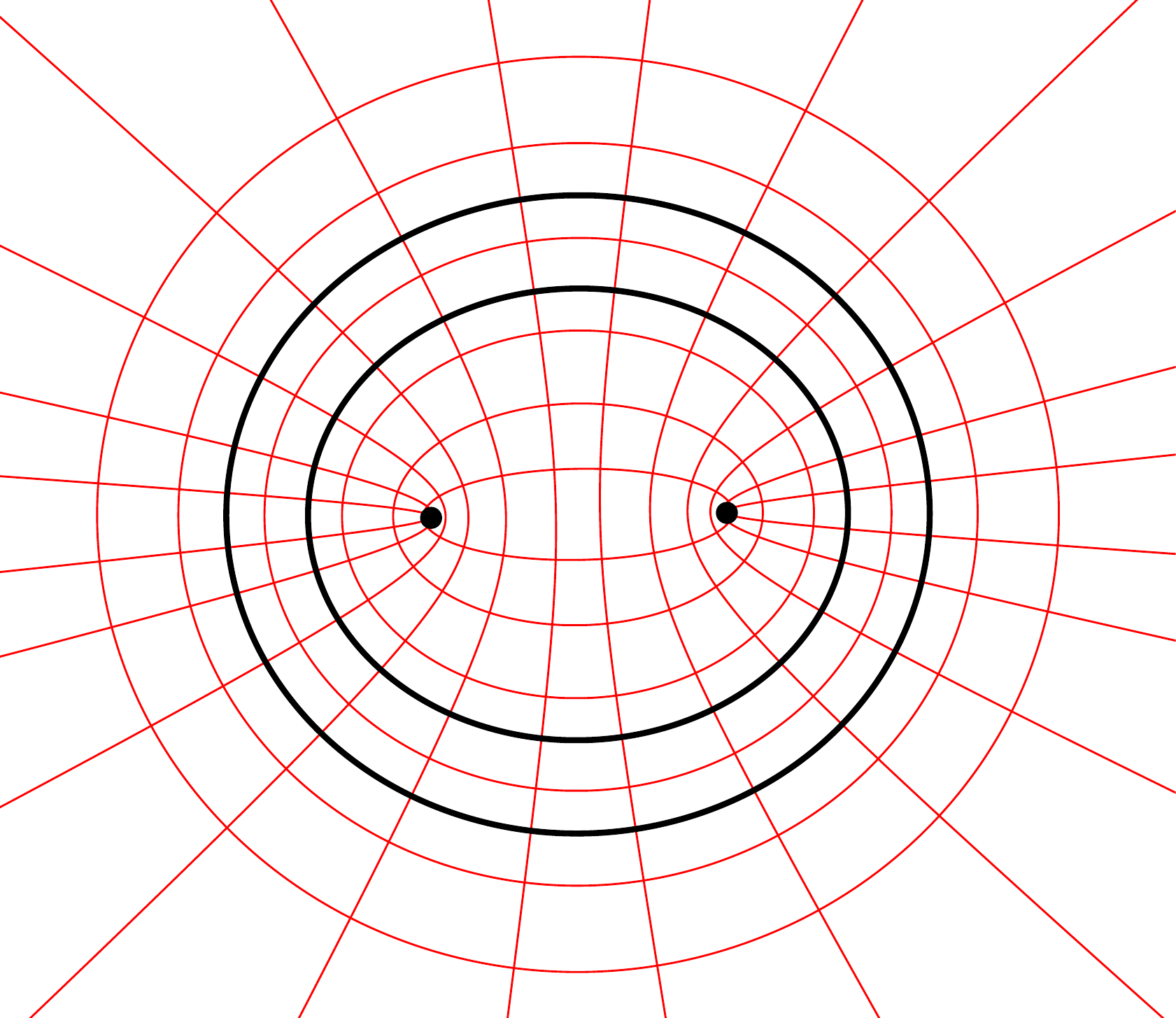}\hfill
\includegraphics[width=0.45\textwidth]{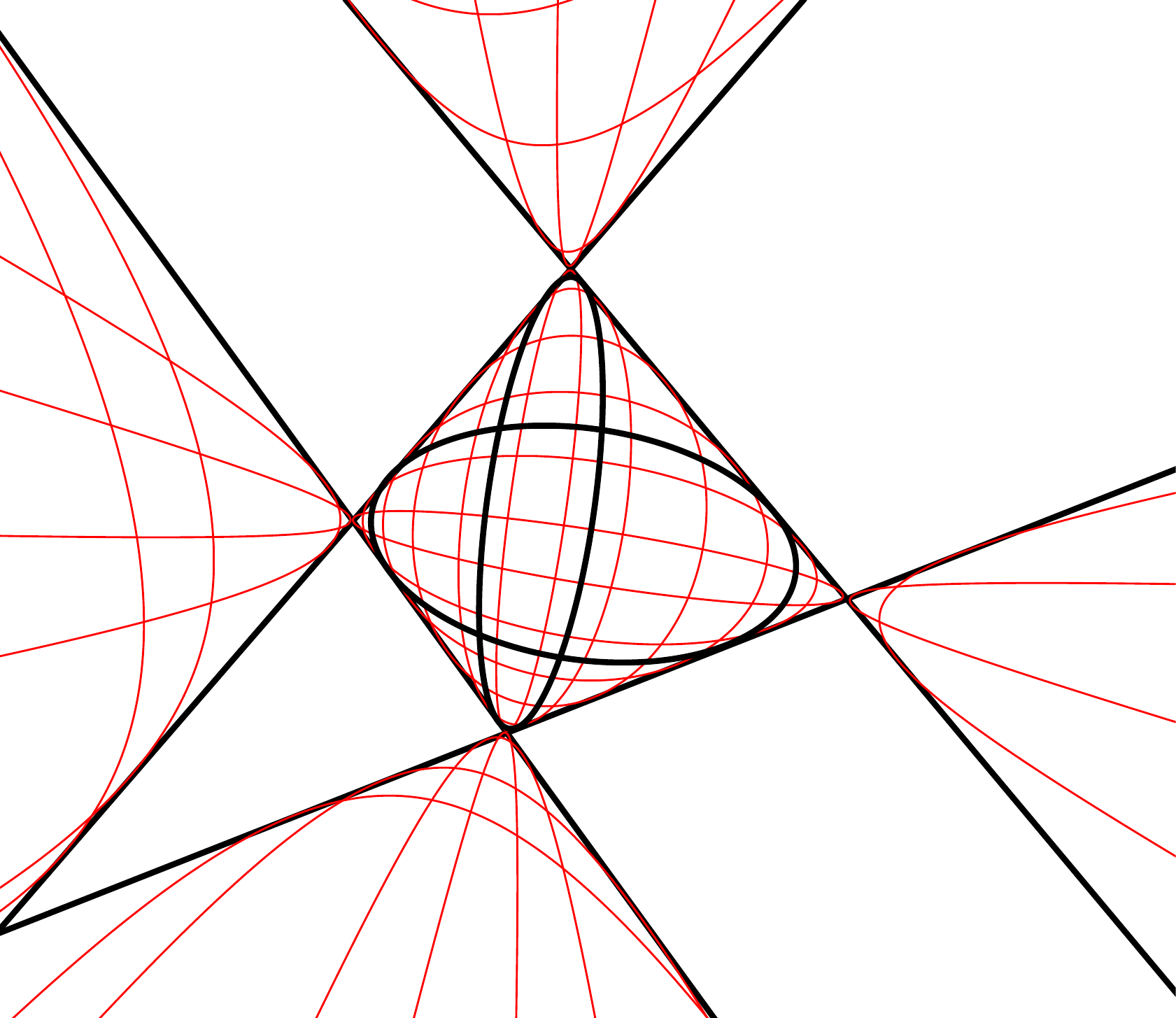}
\captionof{figure}{Some conics (red) from the co-pencil of co-dependent conics spanned by two other conics (black). 
                          On the right is the case of four real tangents, on the left is the case of four complex tangents. 
                          Up to projective transformation, this case corresponds to a set of confocal conics. }
\end{figure}

\bigskip

\subsection{Operations on rings of points and lines}

We fix some notation for operating  on polygons here.
Let $P=(p_1,p_2\upto p_m)$ be a sequence of points on a conic.
We consider indices modulo $m$ and define a sequence of lines
\[
\lns[i](P):=(p_1\vee p_{1+i}\upto p_m\vee p_{m+i}).
\]
where $p_{j} \vee p_{k}$ indicates the line passing through points $p_{j}$ and $p_{k}$.
Thus $\lns[i](P)$ is a sequence of $m$ lines that arise from connecting 
points of $P$ to points of $P$ by shifting the index  $i$ steps.
Since the points of $P$ were assumed to lie on a conic, we cover the limit situation
by defining $\lns[0](P)$ to be the tangents to the conics at the points $p_i$.

Dually we define a similar operation for a list $L=(l_1,l_2\upto l_m)$ of lines tangent to a conic: The notation
\[
\pts[i](L):=(l_1\wedge l_{1-i}\upto l_m\wedge l_{m-i})
\]
represents a sequence of points formed by the intersection of specific tangents; here, $l_{j} \wedge l_{k}$ represents the point of intersection of the lines $l_{j}$ and $l_{k}$. 
Notice that the index shift goes in exactly the opposite direction. Similarly to the point case we define  $\pts[0](L)$ to be the touching points of the lines $l_i$ to the conic.
By this convention we obtain that 
\[\pts[i](\lns[i](P))=P\]
so that 
\[\pts[i]\cdot \lns[i]={\mathrm{id}}.\]

\subsection{Poncelet grids}

An important ingredient in the proof of our main theorem is {\it Poncelet grids}. They unveil a  whole class of conics that underly a single Poncelet $m$-gon. These conics  exhibit characteristic dependencies and co-dependencies. The existence of these conics was already known to Darboux \cite{Dar17}. A nice geometric treatment  and proofs can be found in \cite{Ber87}. 
The results can further be sharpened by observing that for odd $m$ the points in the different Poncelet grid rings 
are related by projective transformations (see \cite{LaTa07,Sch07}). For even $m$ they fall into 2 different classes by parity and are projectively equivalent within the classes.

We will use these results about the conics and dependencies to ensure the existence of the related $(n_4)$ configurations.
In order to create real configurations, whenever we draw pictures we will restrict ourselves to Poncelet polygons supported by a set of nested ellipses. 

\medskip
For our main result we need to deal with Poncelet grids and their duals.
For this let $P=(p_1,p_2\upto p_m)$ be a Poncelet $m$-gon. Throughout this section we count indices modulo $m$. 
We assume that all vertices of $P$ are on a conic $\mathcal{C}$ and all lines $L=\lns[1](P)$ are tangent to a conic $\mathcal{X}$.

We now form intersections of these lines.
We set $P_i=\pts[i](L)$. In particular,  the limit case $P_0$ is the set of  touching points of $L$ to the inner conic.

\begin{theorem}\label{ponceletGrid1}
Let $P=(p_1\upto p_m)$ be a Poncelet polygon with points on a conic $\mathcal{C}_0$ and with lines $L=\lns[1](P)$ tangent to a conic $\mathcal{X}$. The points in each ring $P_i=\pts[i](L)$ 
all lie on a conic $\mathcal{C}_i$.
All conics $(\mathcal{C}_0\upto \mathcal{C}_m,\mathcal{X})$ are co-dependent.
\end{theorem}

\begin{figure}[h]
\begin{center}
\includegraphics[width=0.75\textwidth]{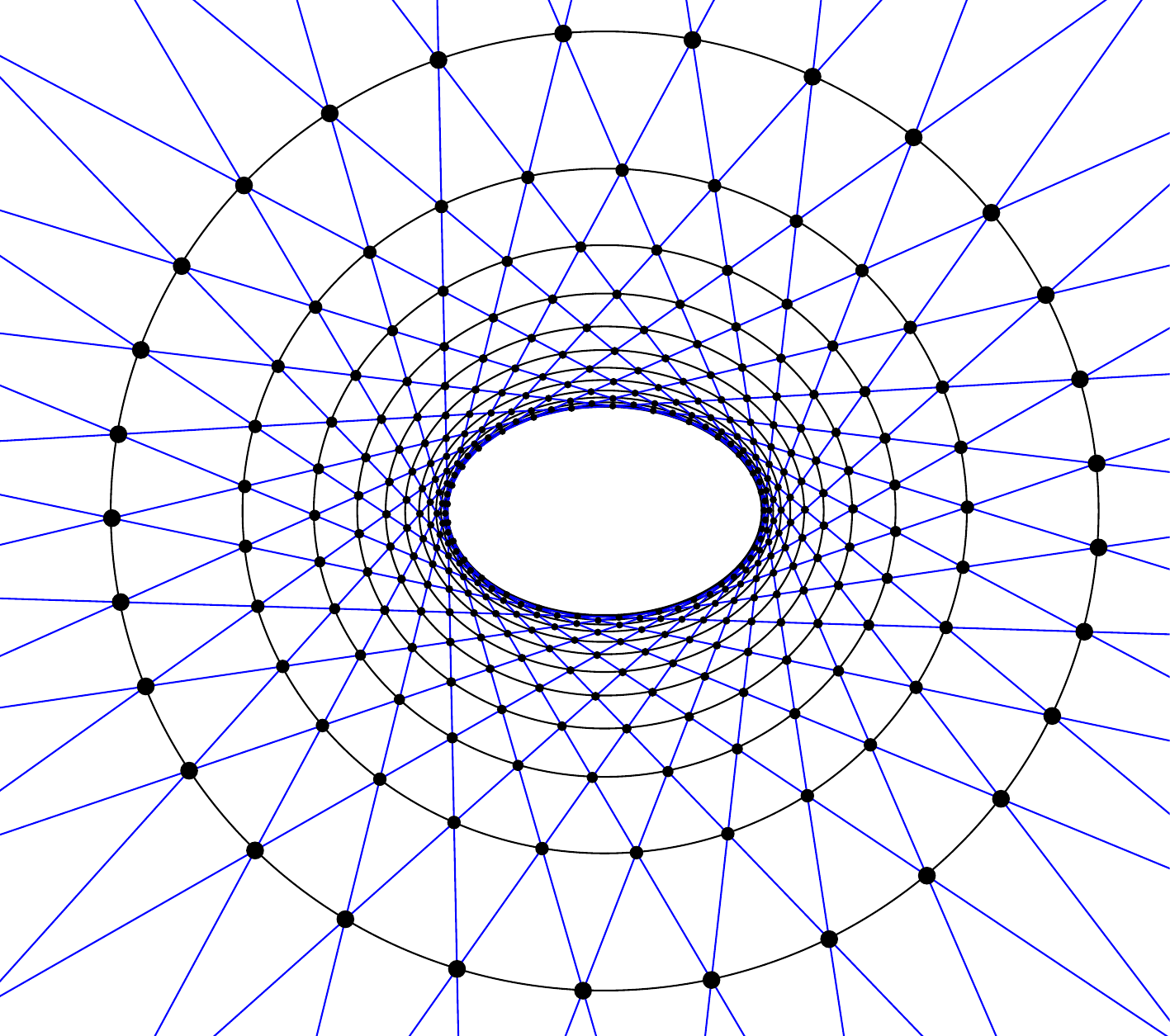}
\captionof{figure}{Different rings of points and conics in a Poncelet grid} \label{fig:grid1}
\end{center}
\end{figure}

We call the set of points $P_i$, lines $L$, and conics $(\mathcal{C}_0\upto \mathcal{C}_m,\mathcal{X})$ a 
{\it Poncelet grid}. Figure \ref{fig:grid1} shows a Poncelet $29$-gon and the corresponding Poncelet grid conics. The Poncelet polygon consists of the tangents to the innermost conic $\mathcal{B}$. From inside out one sees the point rings $\pts[0](L),\pts[1](L),\pts[2](L),\ldots$. 
Each ring of points lies on a conic $\mathcal{C}_i$, and may be considered as a Poncelet (star-)polygon supported by $\mathcal{C}_i$ and $\mathcal{X}$. If $\gcd(i,m)>1$,
the points in a ring $\pts[i](L)$ decompose into several smaller Poncelet polygons. When the index $i$ is above $m/2$, the conics start to repeat. We get $\mathcal{C}_{j} = \mathcal{C}_{i}$ where $j = i - m/2$. We exclude index $m/2$ when $m$ is even.

There is an obvious dual version of the above theorem.
For this we consider different 
rings of lines $L_i=\lns[i](P)$ of a Poncelet $m$-gon. We set $L_0=L$. As a dual statement we get:

\begin{theorem} \label{thm:DualPonceletPolygon}
Let $P$ be a Poncelet $m$-gon with points on conic $\mathcal{C}$ and lines $L=\lns[1](P)$ tangent to $\mathcal{X}=\mathcal{X}_0$. 
Each ring of lines $L_i=\lns[i](P)$ is tangent to a single conic $\mathcal{X}_i$.
All conics $(\mathcal{C},\mathcal{X}_0\upto \mathcal{X}_n)$ are dependent.
\end{theorem}

Figure \ref{fig:grid29} shows some of the rings of the dual Poncelet grid of a Poncelet 29-gon.

\begin{figure}
\includegraphics[width=0.3\textwidth]{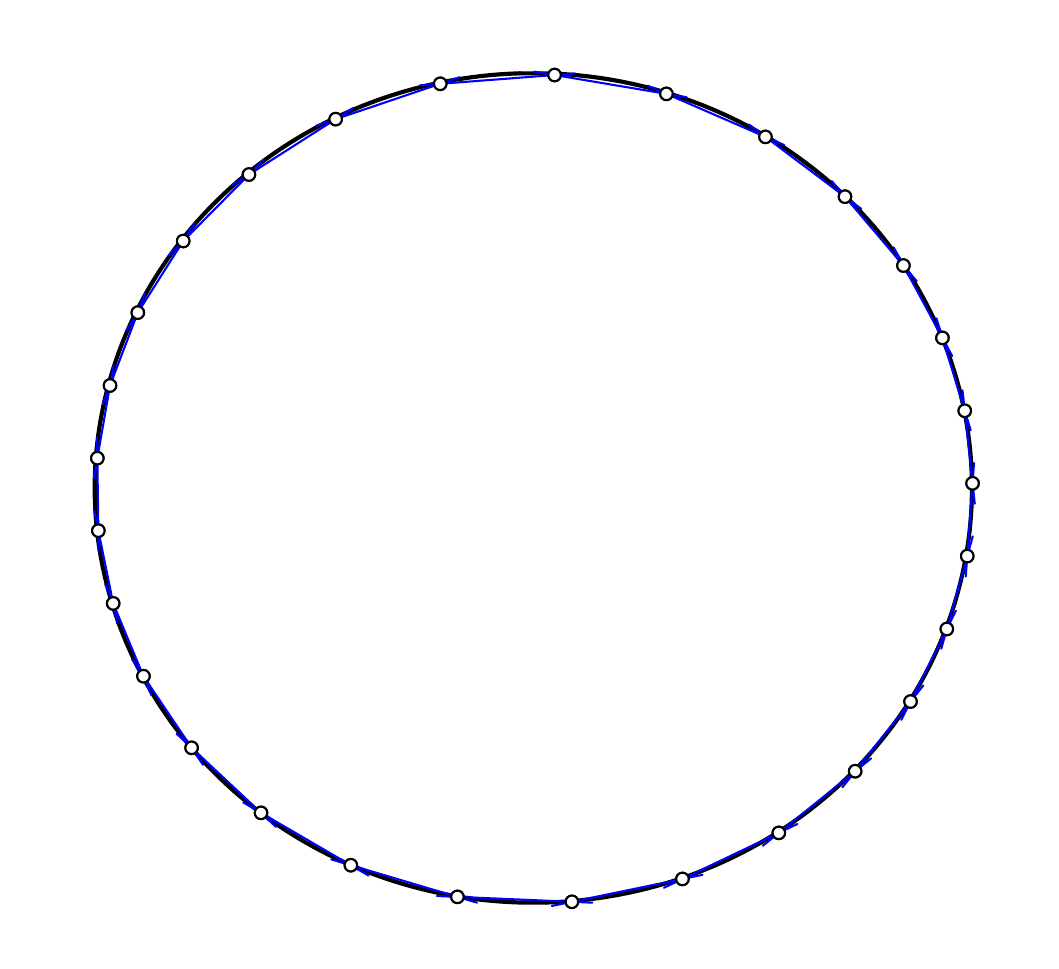}\hfill
\includegraphics[width=0.3\textwidth]{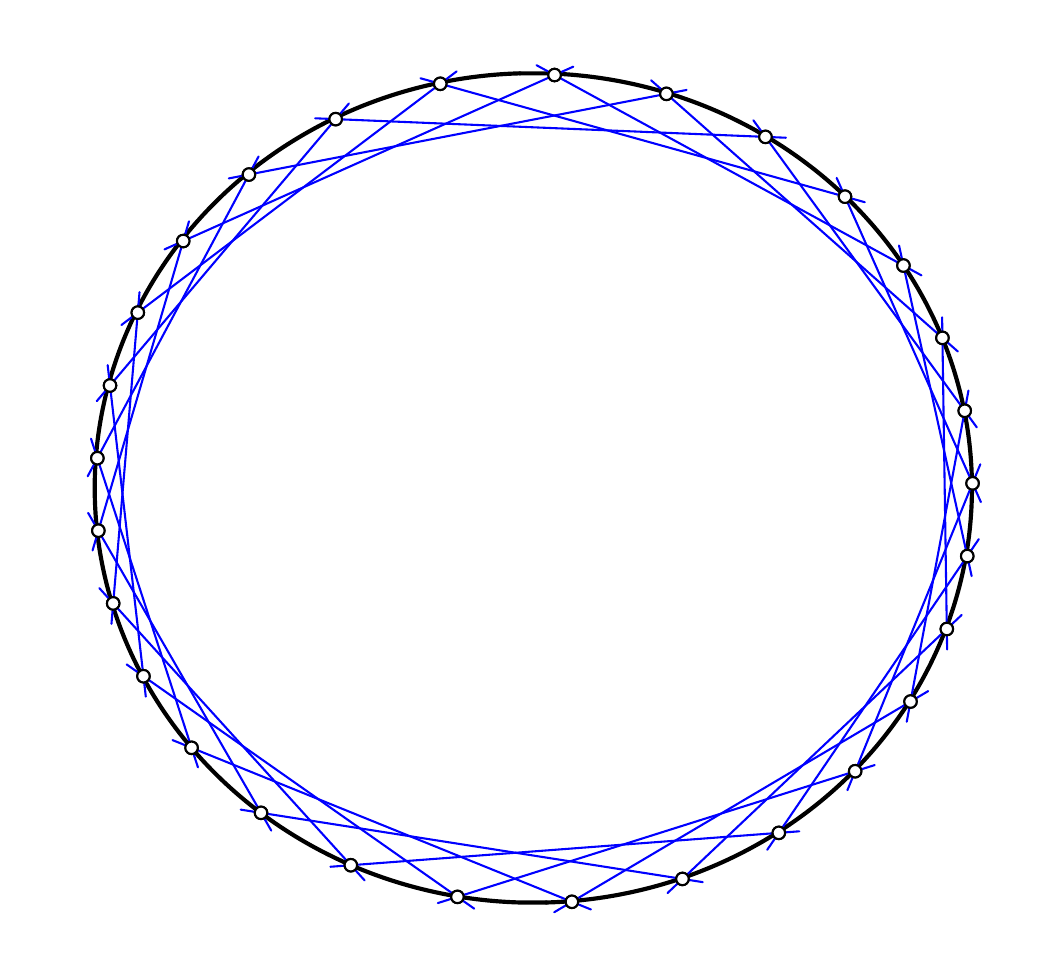}\hfill
\includegraphics[width=0.3\textwidth]{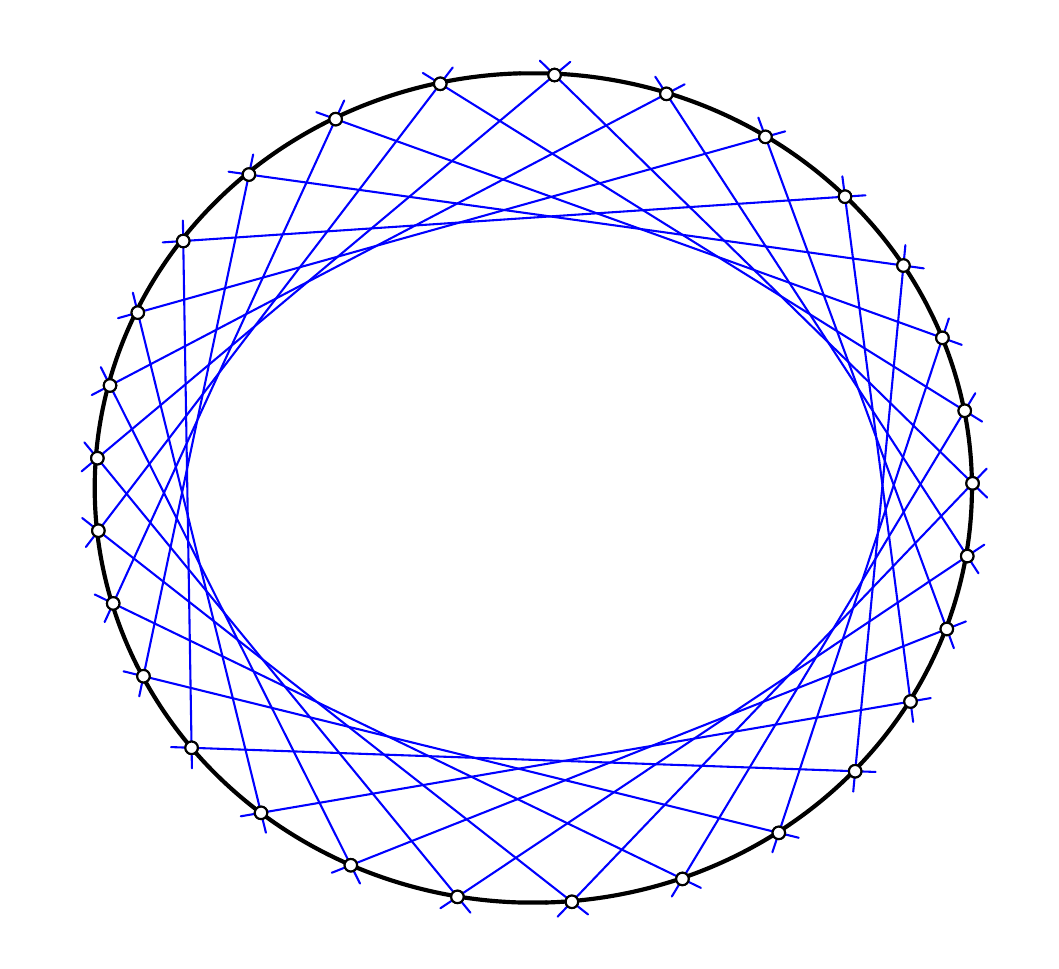}

\noindent
\includegraphics[width=0.3\textwidth]{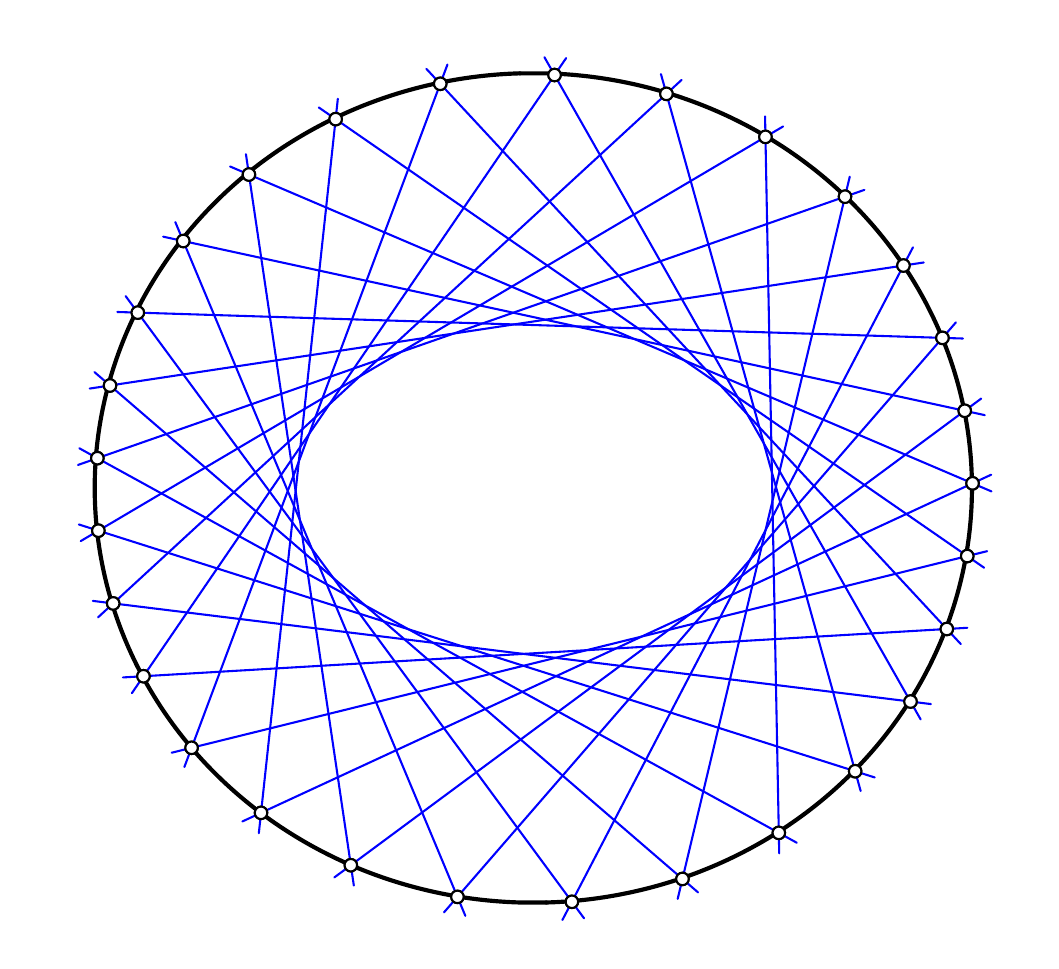}\hfill
\includegraphics[width=0.3\textwidth]{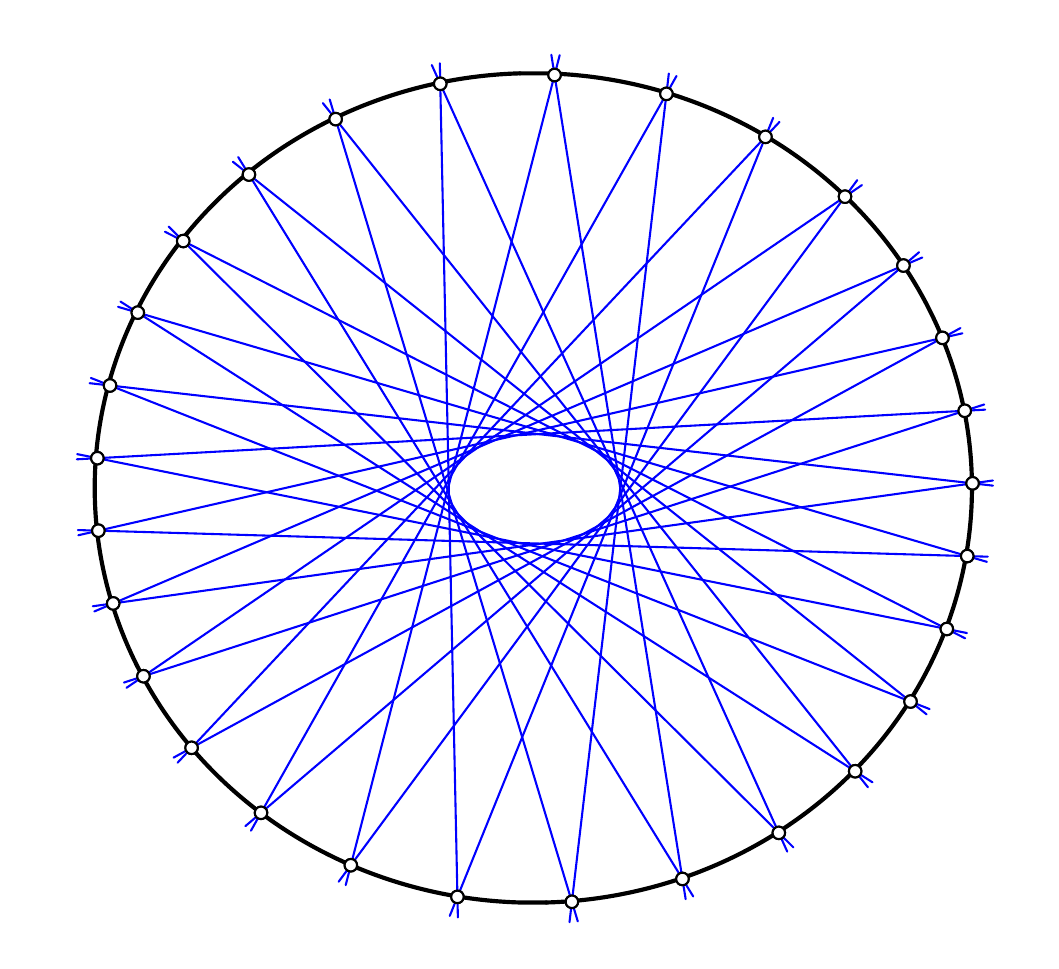}\hfill
\includegraphics[width=0.3\textwidth]{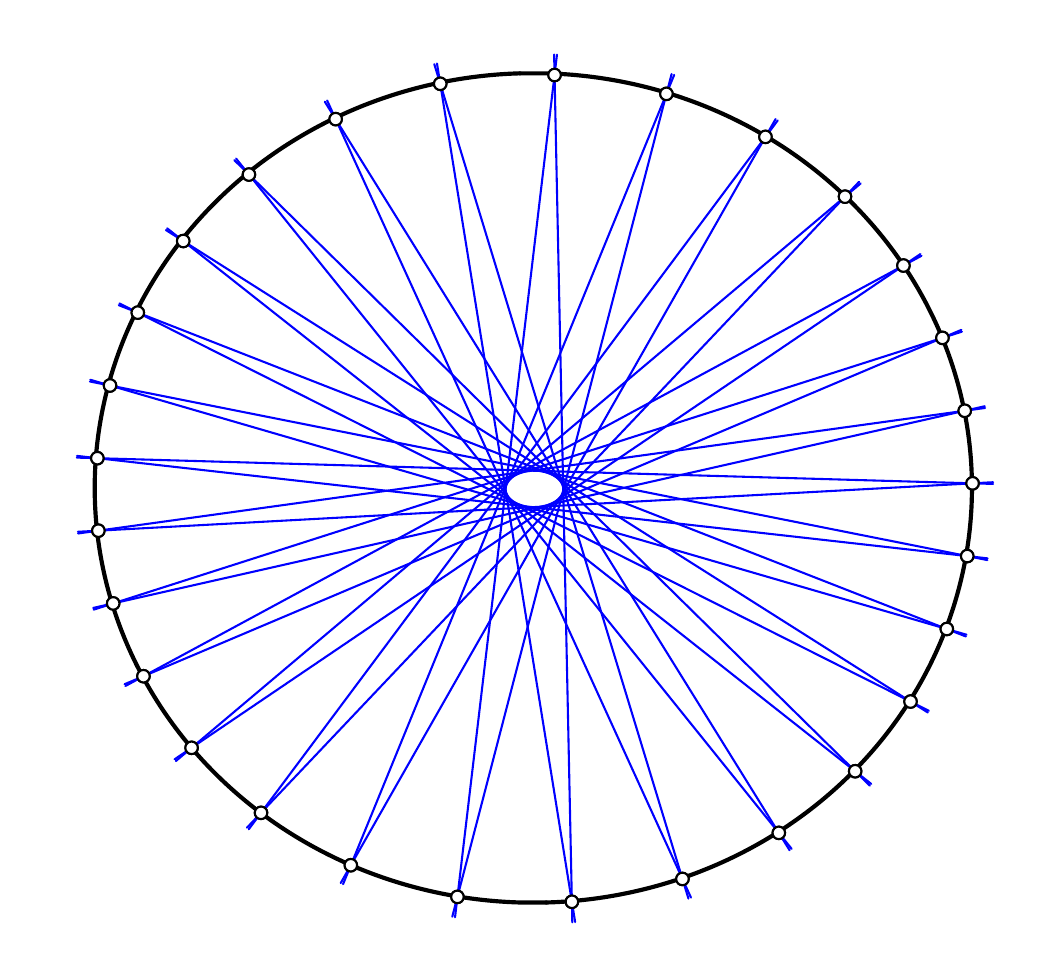}

\captionof{figure}{Rings of lines arising from dual Poncelet grids of a Poncelet 29-gon.}\label{fig:grid29}
\end{figure}

\section{From Poncelet to $(n_4)$ configurations}
 \label{sect:fromPoncelet}

\subsection{Movable $(n_4)$ configurations}

We now show how one can construct trivial celestial $(n_4)$ configurations from 
a given Poncelet polygon and its related Poncelet grids. 
We  describe different constructions that emphasise various structural aspects of the configuration.
We first describe a  construction that is close to the original setup of Gr\"{u}nbaum.
Starting from a Poncelet polygon we construct additional points and lines that end up forming an $(n_4)$ configuration. We will prove the correctness of the construction by a direct calculation in a self-contained way.
After that we will relate the  core incidence statement of the previous argument to a local
incircle argument that arises after a special coordinate transformation.  Then, we  describe related ways to obtain weaker versions of our main result in the situation where the Poncelet conics are confocal ellipses,  including
one that relates the construction to  {\it incircle nets} \cite{AkBo18} locating the precise position of possible points in the configuration. We will also present another approach  giving an argument based on local perturbed coordinate systems that arise in the theory of the {\it geometry of billiards} \cite{Ta05}. Furthermore, a significantly simpler proof will be given in the special case where we start with a Poncelet polygon with an {\it odd} number of points.
\medskip

Let $m\geq 7$ and $P=(p_1\upto p_m)$ be a Poncelet polygon. 
Beginning with a $P$ we construct new points and lines from the points of $P$. We adopt the concept of {\it celestial configurations} of Gr\"{u}nbaum and Berman to describe a construction 
process that leads to $(n_4)$ configurations. In Gr\"{u}nbaum's book on configurations this notion was exclusively applied to point sets $P$ that are the vertices of regular 
$m$-gons \cite{Ber01,Ber06,Ber08,Gr00a,GR00b,Gr09}. The fact that the same procedure can (in specific cases) also be applied to Poncelet polygons instead of regular 
$m$-gons, replacing trigonometric constraints by projective and algebraic ones, is one of the main results of this article.

For $m\geq 7$ consider the formal symbol $m\#(a_1,b_1;a_2,b_2;\ldots;a_k,b_k)$ where each of the letters $a_i$ and  $b_i$ is a positive integer less than $m/2$.
To this symbol and an initial point set $P=(p_1\upto p_m)$ we assign additional points and lines by the following construction process.

\begin{construction}\label{con:mainconstr}
Given symbol $m\#(a_1,b_1;a_2,b_2;\ldots;a_k,b_k)$
we construct the following sequences of point $P_i$ and lines $L_i$:   
\[
\begin{array}{lllll}
P_0=P,\\   L_1=\lns[{a_1}](P_0),&P_1=\pts[{b_1}](L_1),&\\L_2=\lns[{a_2}](P_1),&P_2=\pts[{b_2}](L_2),\\ 
\ldots,\\
L_k=\lns[{a_k}](P_{k-1}),&P_k=\pts[{b_k}](L_k).
\end{array}
\]
\end{construction}
The situation is illustrated in Figure~\ref{fig:celest1}. The red points are the points of a Poncelet octagon (the construction above does not explicitly require this, but this is our main use-case). 
The initial points are assumed to be indexed counterclockwise in the order they appear on the conic. The left picture illustrates the construction of $8\#(3,1)$. The number 3 indicates that we connect $p_i$ and $p_{i+3}$ to get the black lines $\lns[3](P)$. The lines $i$ and $i-1$ are emphasized in the picture. The number 1 that follows the 3 in 
the sequence indicates that we 
intersect lines $i$ and $i-1$ for $i\in{1\upto m}$ to get $\pts[1](\lns[3](P))$: the blue points. The blue point with index $i$ is marked in the picture. 

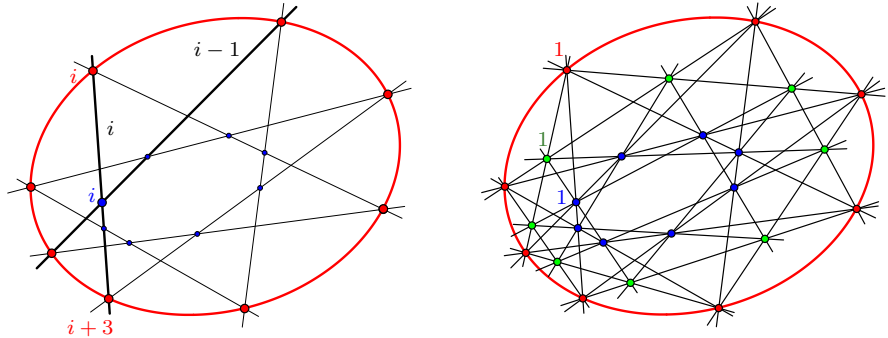
\begin{figure}[h]
\begin{center}
\includegraphics[width=0.47\textwidth]{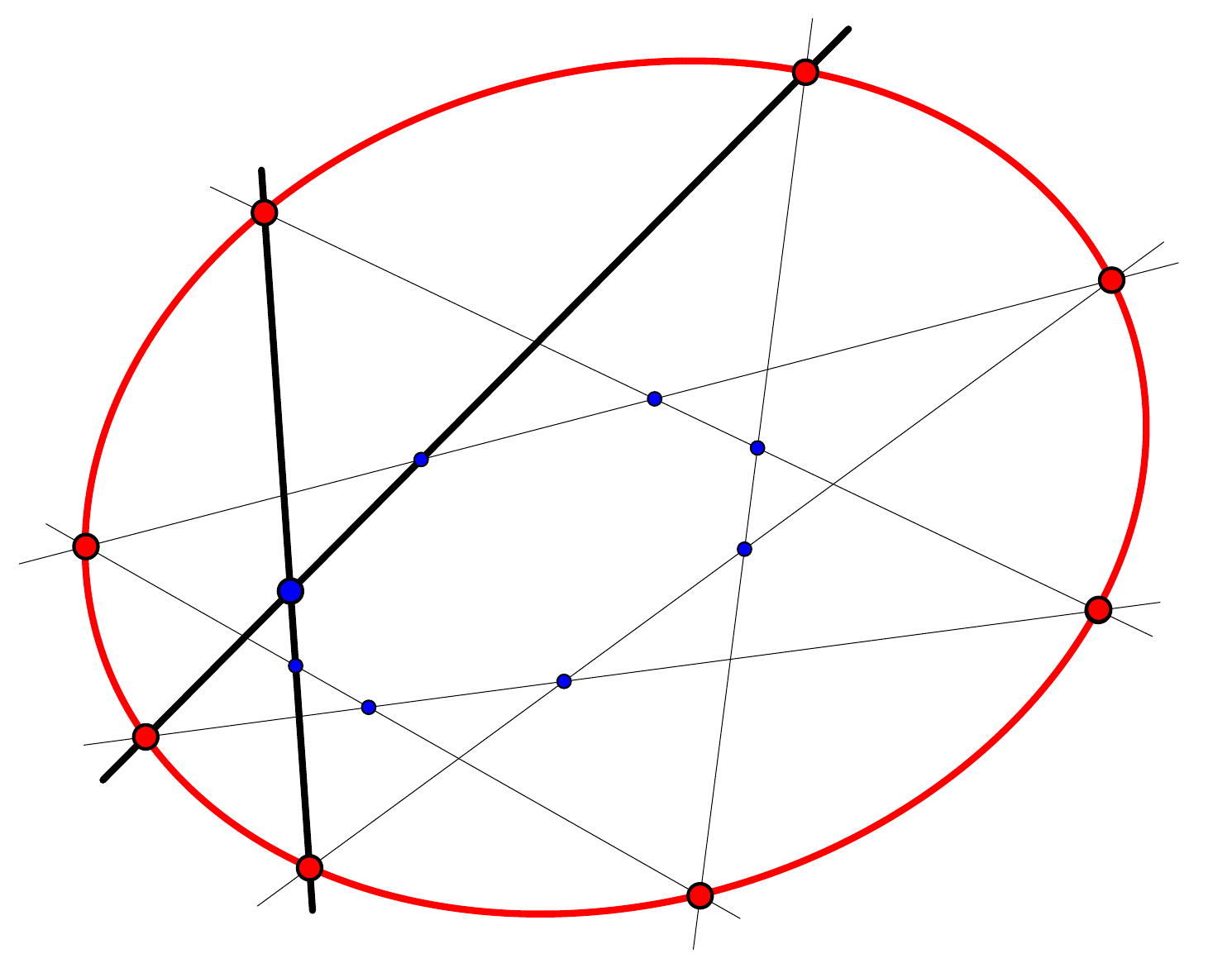}\hfill
\includegraphics[width=0.47\textwidth]{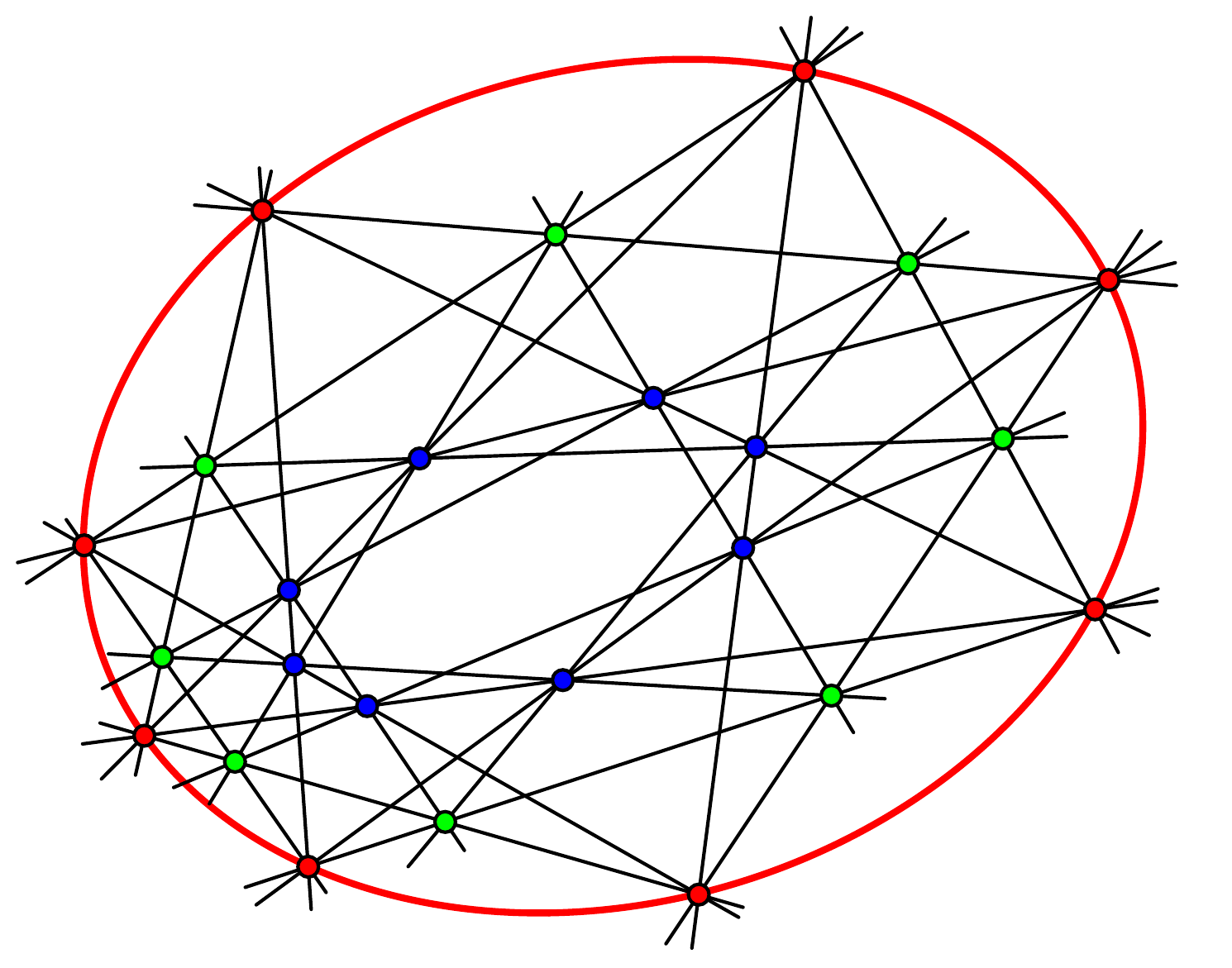}
\begin{picture}(0,0)
\put(-319,95){\footnotesize {\color{red}$i$}}
\put(-320,0){\footnotesize {\color{red}$i+3$}}
\put(-305,75){\footnotesize {\color{black}$i$}}
\put(-272,105){\footnotesize {\color{black}$i-1$}}
\put(-312,50){\footnotesize {\color{blue}$i$}}
\put(-136,105){\footnotesize {\color{red}$1$}}
\put(-135,49){\footnotesize {\color{blue}$1$}}
\put(-142,71){\footnotesize {\color{OliveGreen}$1$}}
\end{picture}
\captionof{figure}{Construction of $8\#(3,1)$ and of  $8\#(3,1;2,3;1,2)$ beginning with an initial set of (red) points $P$.}\label{fig:celest1}
\end{center}
\end{figure}

In the right part of Figure~\ref{fig:celest1} the construction is taken further and all elements of $8\#(3,1;2,3;1,2)$ are shown. We marked an initial red point with a label 1 (this is $(P_0)_1$). 
The blue and green points ($(P_1)_1$ and  $(P_2)_1$, respectively) that get a label 1 by the construction are marked as well. 

One might wonder, where the last ring of points $P_3$ in the picture is. As a matter of fact, it coincides 
with the initial set of red points $P_0$. 
This is no accident. 
The fact that constructions like this may close up under certain conditions is the main theorem of this article.
In this specific case we get 
\[\pts[2]\cdot\lns[1]\cdot\pts[3]\cdot\lns[2]\cdot\pts[1]\cdot\lns[3](P)=P.\]
After applying a sequence of six operations, each point of $P$ gets mapped back to its original position. 
 As a result, in the situation of Figure~\ref{fig:celest1} 
this construction leads to a $(24_4)$ configuration: every point lies on 4 lines and every line passes through four points. Before we come to this theorem, we specify the exact relation to $(n_4)$ configurations. 
To 
ensure that the sequential sets of points  and lines $P_i$, $P_{i+1}$ (resp.,  $L_i$, $L_{i+1}$) are distinct
we require that no adjacent letters (taken cyclically) in 
$m\#(a_1,b_1;a_2,b_2;\ldots;a_k,b_k)$ 
are identical. Except for the first and the final ring of points, each point is by construction involved  in 4 lines––two from which it is constructed and two others used to construct the next ring of lines. 
For the same reason every line is incident to four points. If, in addition, the configuration closes up and the initial ring of points
coincides with the last ring (as sets), these points are incident with four lines.
Thus we obtain a configuration where on each line there are (at least) four points and  through each point there are (at least) four lines.
Additional incidences might happen. For that reason we call such a configuration a {\it pre-$(n_4)$} configuration. 
Gr\"{u}nbaum's book characterises those symbols where we get proper $(n_4)$ configurations that do not have additional coincidences (see \cite[\S 3.5--3.8]{Gr09}). {In particular, symbols of the form $m\#(a_1,b_1;a_2,b_2;\ldots;a_k,b_k)$ where adjacent entries are distinct and the multisets $[a_{1}, \ldots, a_{k}]$ and $[b_{1}, \ldots, b_{k}]$ are equal (as multisets) satisfy the required properties for not having additional unintended incidences. 

The important fact for us is the following:
If we have an instance of Construction \ref{con:mainconstr} that closes 
up and which satisfies Grünbaum's properties or the parameters of the symbol, then we have an $(n_4)$ configuration. The fact that in a Poncelet polygon we can move the points along the conic translates to the fact that 
beyond projective transformations, the $(n_4)$ configuration admits additional degrees of freedom: in other words, {\it it is movable}!

We proceed with three theorems of increasing structural complexity. 
They ensure the existence of configurations that close up, based on Poncelet polygons and the $\lines$-, $\points$-constructions. We start with a situation that characterises the situations 
for constructions which have exactly three rings of points.

\begin{theoremA*} 
Let $P=(P_1,P_2\upto P_m)$ be a Poncelet polygon and let $a,b,c$ each be distinct positive integers strictly smaller than $m/2$. Then
\[
\pts[c]\cdot\lns[b]\cdot
\pts[a]\cdot\lns[c]\cdot
\pts[b]\cdot\lns[a]\big
(P)=P.
\]

\end{theoremA*}
 { This theorem shows that starting with a Poncelet polygon with initial ring of vertices $P$, the incidence structure given by the symbol $m\#(a,b;c,a;b,c)$ closes up and produces a $(3m_{4})$ configuration.} 
We postpone the proof of this crucial fact to the next section. First we show how this local fact implies much more general situations.

Assume that Theorem A is already proven. Based on this we show that operations of the form  $(\pts[{b_i}]\cdot\lns[{a_i}])$ commute with each other when applied to a Poncelet polygon.

\begin{theoremB*} 
Let $P$ be a Poncelet polygon and let $a,b,c,d<m/2$. Then:
\[
 (\pts[{d}]\cdot\lns[{c}]) (\pts[{b}]\cdot\lns[{a}])(P)
 =
 (\pts[{b}]\cdot\lns[{a}]) (\pts[{d}]\cdot\lns[{c}])(P).
 \]

\end{theoremB*}

\begin{proof}
Since $AB=BA$ if and only if $ABA^{-1}B^{-1}=\mathrm{id}$ and 
\[(\pts[{x}]\cdot\lns[{y}])^{-1}=\pts[{y}]\cdot\lns[{x}],\]
to prove Theorem B, it suffices to show that

\begin{equation}
\pts[{c}]\cdot\lns[{d}]\cdot \pts[{a}]\cdot\lns[{b}] 
\cdot
\pts[{d}]\cdot\lns[{c}]\cdot \pts[{b}]\cdot\lns[{a}] 
(P)=P. \tag{1}
\end{equation}

\noindent
By multiplying  both sides of the conclusion of Theorem A by $\lns[a]\cdot\pts[b]\cdot
\lns[c]$ and canceling, we get 
\[
\lns[{c}] \cdot \pts[{b}] \cdot \lns[{a}]
(P)
=
\lns[a]\cdot\pts[b]\cdot
\lns[c]\,
(P), \tag{2}
\]
 if $P$ is a Poncelet polygon. 
In addition, Theorem A holds for any three distinct positive integers less than $m/2$. Thus 
\[
\pts[{c}]\cdot\lns[{d}]\cdot \pts[{a}]\cdot
\lns[{c}] \cdot\pts[{d}]
\cdot
\lns[{a}] 
(P)=P.
\]
Since $\lns[{b}]\cdot \pts[{b}] = \text{id}$, we can insert it as we like, so
\[
\pts[{c}]\cdot\lns[{d}]\cdot \pts[{a}]\cdot
\lns[{c}] \cdot\pts[{d}]
\cdot
\underbrace{\lns[{b}]
\cdot \pts[{b}]}_{\mathrm{id}}\cdot\lns[{a}] 
(P)=P.
\]
Applying replacement (2) (using letters $b,c,d$), 
\[
\pts[{c}]\cdot\lns[{d}]\cdot \pts[{a}]\cdot
\underbrace{\lns[{c}] \cdot\pts[{d}]
\cdot
\lns[{b}]}_{\lns[{b}]\cdot\pts[{d}]\cdot\lns[{c}]}
\cdot \pts[{b}]\cdot\lns[{a}] 
(P)=P,
\]
so (1) 
holds. Thus, Theorem B follows, given Theorem A. 
\end{proof}

\bigskip
\noindent
Finally, we use Theorem B to show that the construction denoted by 
\[m\#(a_{1}, b_{1}; \cdots; a_{k}, b_{k})\] produces a trivial celestial configuration, with four points on each line and four lines through each point. By $[x,y,\ldots]$ we denote the multiset, containing elements with their multiplicity.
We show:
\begin{theoremC*} 
Let $P$ be a Poncelet polygon, and let $a_1,b_1,a_2,b_2\upto a_k,b_k$ all be positive integers strictly smaller than $m/2$ where adjacent entries (taken cyclically) have distinct values.
Furthermore let $[a_1\upto a_k]=[b_1\upto b_k]$ as multisets. Then we have
\[
\pts[{b_k}]\cdot\lns[{a_k}]\cdot\ldots\cdot \pts[{b_2}]\cdot
\lns[{a_2}] \cdot\pts[{b_1}]
\cdot
\lns[{a_1}] 
(P)=P.
\]
\end{theoremC*}
\begin{proof}
We group the expression above in pairs of operations:
\[
(\pts[{b_k}]\cdot\lns[{a_k}])\cdot \ldots \cdot
(\pts[{b_2}]\cdot
\lns[{a_2}]) \cdot(\pts[{b_1}]
\cdot
\lns[{a_1}]) 
(P).
\]
Since all intermediate point rings are Poncelet polygons, Theorem B implies that adjacent pairs 
commute with each other. Hence we can rearrange the order of the pairs such that we have
\[
(\pts[b_k]\cdot\lns[a_k]) \cdot \ldots 
\cdot
(\pts[b_1]\cdot\lns[a_1])\cdot
(\pts[b_i]\cdot\lns[a_i]) 
(P)
\]
with $b_i=a_1$ (such an index $b_i$ must exist because of the equal multiset property).
We can cancel $\vee_{a_{1}}\cdot \wedge_{b_i}$ since the two operations are inverse to each other, so the expression
 simplifies to
\[
(\pts[{b_k}]\cdot\lns[{a_k}]) \cdot \ldots 
\cdot
(\pts[{b_1}]
\cdot
\lns[{a_i}]) 
(P).
\]
By sequentially rearranging pairs
we can eliminate the $a_1,a_2,\ldots$, one after the other,
until we end up with an equation
\[
\pts[{b_j}]\cdot\lns[{a_k}]
(P)
\]
for identical $b_j$, $a_k$ which obviously holds.
\end{proof}

Theorem C, therefore, shows that if Theorem A holds, then all valid trivial celestial configurations $m\#(a_{1}, b_{1}; \cdots; a_{k}, b_{k})$ can be constructed beginning with $m$ points forming a Poncelet polygon.

\subsection{Proof of Theorem A}\label{proofofA}

The remaining work of this section is to prove Theorem A:
Every configuration of the type $m\#(a,b;c,a;b,c)$ closes up properly.
We will give an explicit construction that creates a pre-$(n_4)$ configuration from a given Poncelet polygon. 
After that we will prove the correctness of the construction and show how it implies Theorem A.

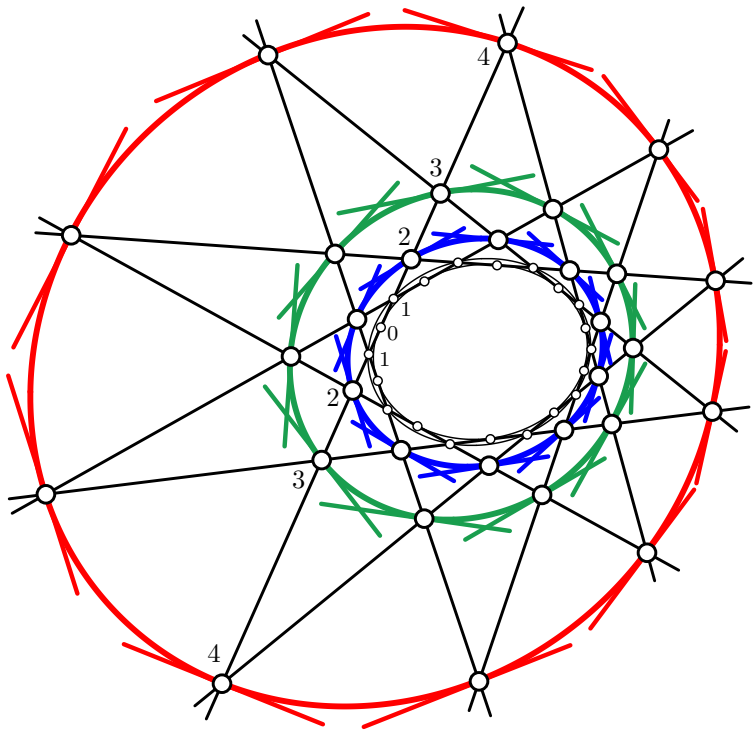
\begin{figure}[H]
\begin{center}
\includegraphics[width=0.95\textwidth]{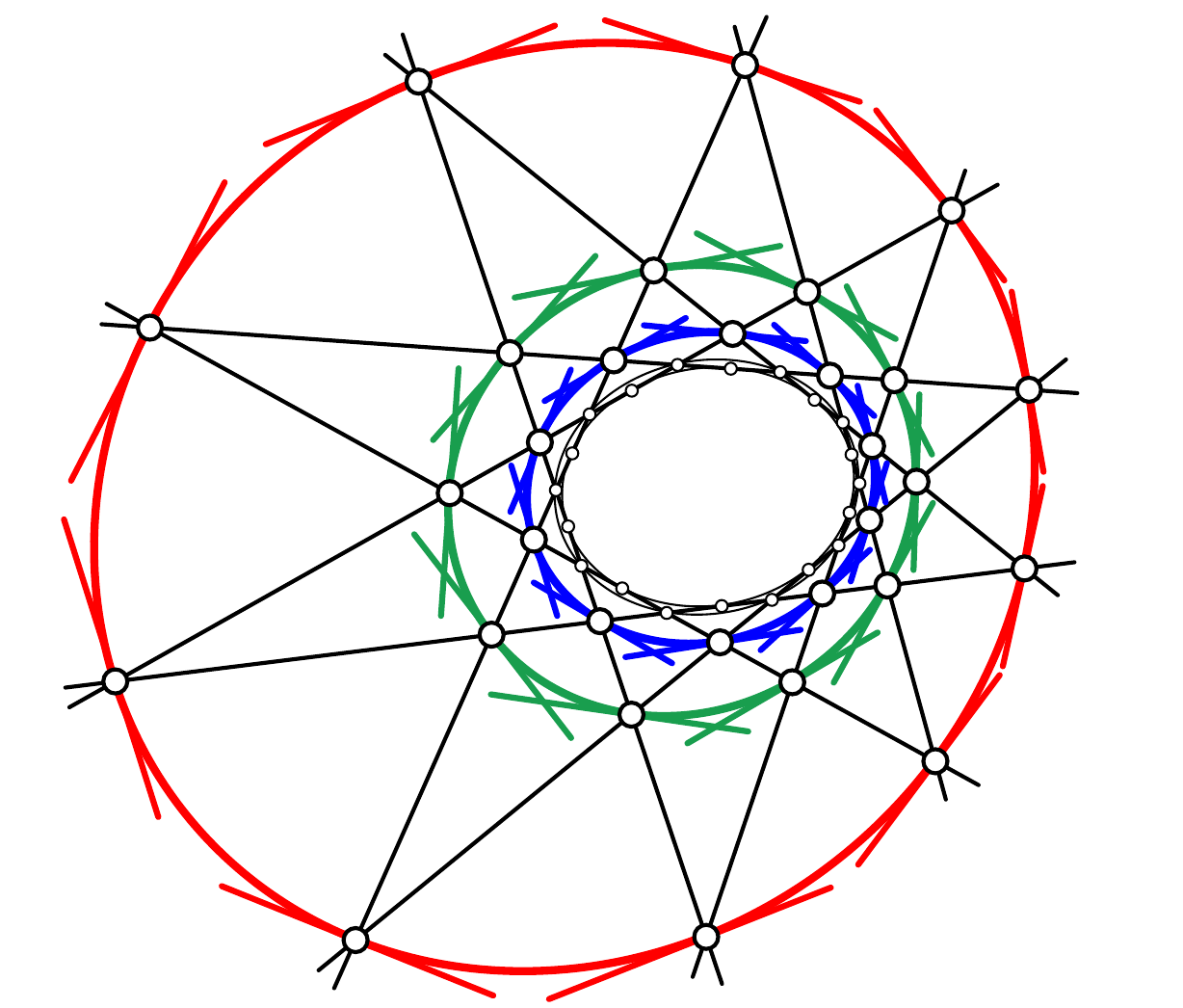}
\begin{picture}(0,0)
\put(-165,158){\footnotesize $1$}
\put(-170,149){\footnotesize $0$}
\put(-173,139){\footnotesize $1$}
\put(-193,124){$2$}
\put(-206,93){$3$}
\put(-238,27){$4$}
\put(-166,185){$2$}
\put(-154,211){$3$}
\put(-136,253){$4$}

\end{picture}
\end{center}
\vskip-5mm
\captionof{figure}{Tangents to points in a Poncelet grid at rings with labels $a$, $b$ and $c$; in this case $a=2, b=3, c=4$. }\label{fig:star1}
\end{figure}

\subsubsection{The construction}\label{mainconstruction}

We start our construction with a Poncelet $m$-gon ($m\geq 7$),
and consider the lines $L=(l_1\upto l_m)$ that support its edges. 
Let $s$  be the greatest integer strictly below $m/2$.
Consider the following intersections between those lines, organised into $s$ rings of $m$ points
\[
P_1=\pts[1](L)\upto \ P_s=\pts[s](L).
\]
along with the points of tangency $P_0=\pts[0](L)$. Note these points $P$ are Poncelet grid points; they are not  the same points $P$ that we are using to construct the configuration in the construction 1 above.

From the Poncelet grid theorem (Theorem \ref{ponceletGrid1}) we know that the points of each ring $P_i$  lie on a conic $\mathcal{C}_i$. All conics are co-dependent with respect to the others. 
({\it Remark:}  up to projective transformation they could be represented by a collection of confocal conics).

Now, we pick three distinct natural numbers  $a,b,c$ between $1$ and $s$ (inclusive). We focus on the rings $P_a$, $P_b$ and $P_c$, and draw tangents to the corresponding conics 
$\mathcal{C}_a$, $\mathcal{C}_b$, $\mathcal{C}_c$. 
In our notation those three rings of tangents are
\[
L_a=
\lns[0]\cdot\pts[a]
(L),\ \quad
L_b=
\lns[0]\cdot\pts[b]
(L),\ \quad
L_c=
\lns[0]\cdot\pts[c]
(L).\ 
\]

\begin{figure}[H]
\begin{center}
\includegraphics[width=0.85\textwidth]{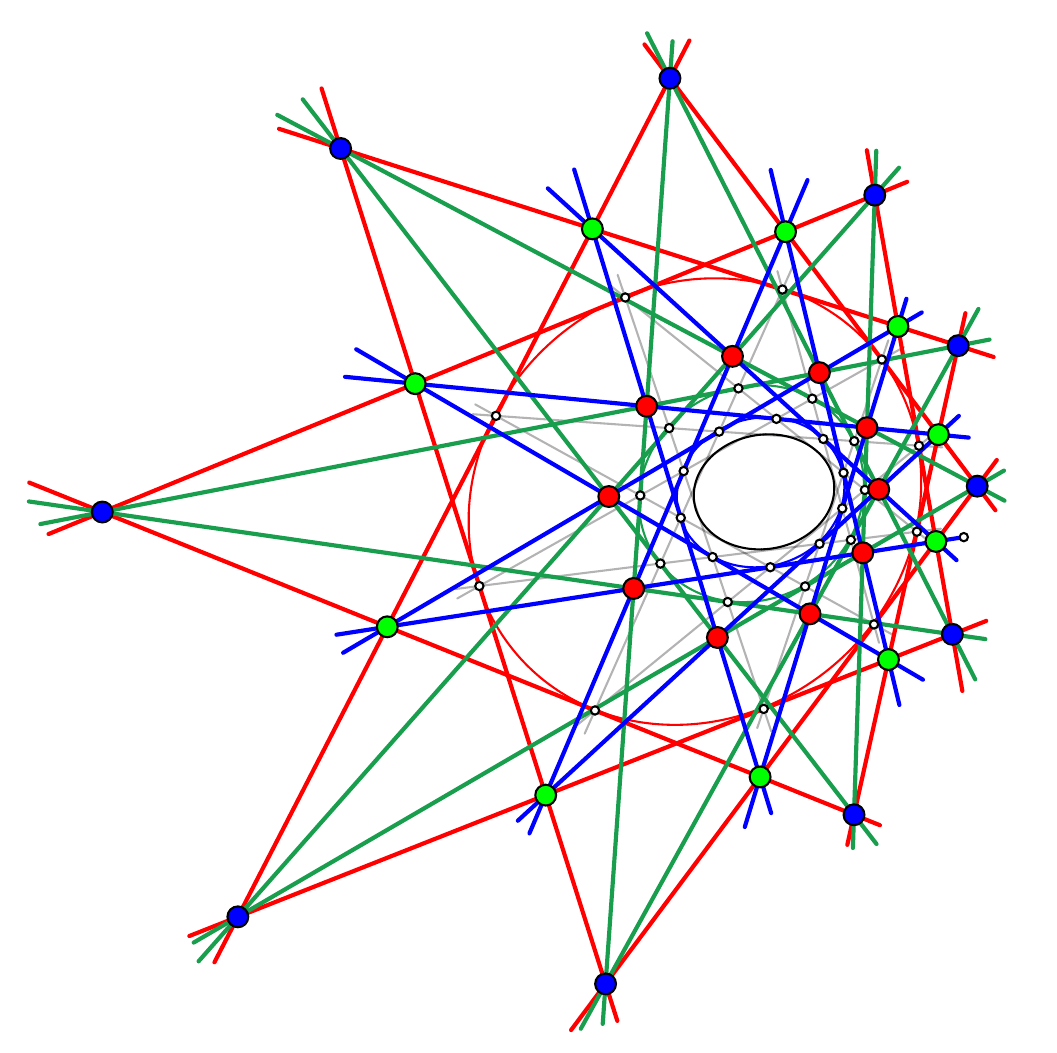}
\begin{picture}(0,0)

\end{picture}
\end{center}
\vskip-5mm
\captionof{figure}{The tangents are the lines of an $m\#(a,b;c,a;b,c)$ configuration, in this case, $10\#(2,3;4,2;3,4)$ beginning with the green points and red lines. Note that the green points are the points $
\pts[{a}]\cdot\lns[{0}]\cdot\pts[{c}]
(L)$, where $L$ is the set of thin black lines.} \label{fig:star2}
\end{figure}

The situation is illustrated in Figure \ref{fig:star1}. There $n=10$. We labelled the points on one line with the indices of the corresponding rings.
The three rings selected are $P_2$, $P_3$ and $P_4$, colored in 
{\it blue},
{\it green} and
{\it red}, respectively.

\noindent 
Amazingly, the construction is essentially already finished at that point.
We have:
\begin{theorem}\label{constrthm}
The $3m$ lines in $L_a,L_b,L_c$ 
support the following intersection pattern: for each pair of rings of tangent lines there are $m$ points in which two lines of each ring meet.
\end{theorem}

\noindent 
The situation is illustrated Figure \ref{fig:star2}.
There the lines of Figure~\ref{fig:star1} are extended, and we ``zoom out" to show all intersection points. In the present situation we obtain a 
$(30_4)$ configuration that corresponds to the construction
$10\#(2,3;4,2;3,4)$ beginning with the green points. Observe that our choice of $a,b,c$ again occurs as parameters here.
\begin{figure}[h]
\begin{center}
\includegraphics[width=0.95\textwidth]{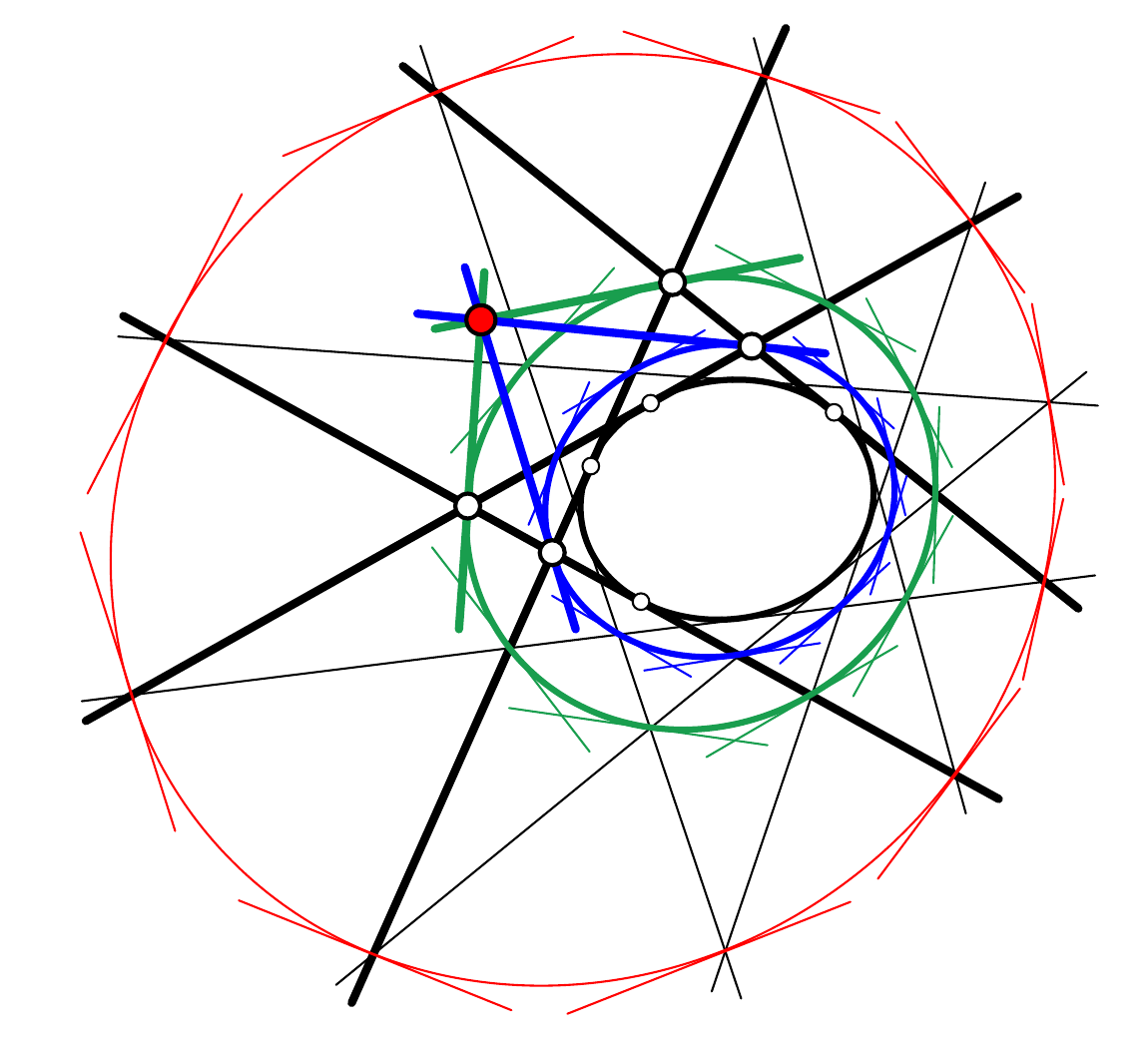}
\begin{picture}(0,0)

\end{picture}
\end{center}
\vskip-5mm
\captionof{figure}{Situation that plays a role for one of the $3m$ four-fold incidences.} \label{fig:star3}
\end{figure}

\subsubsection{A local incidence lemma}

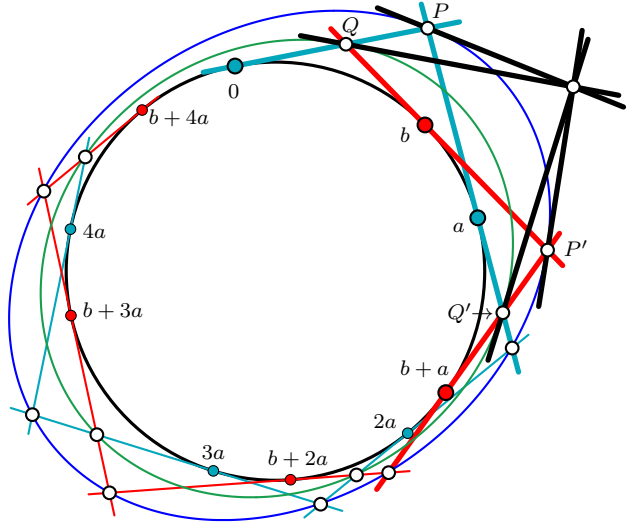
\begin{figure}[h]
\begin{center}
\includegraphics[width=0.75\textwidth]{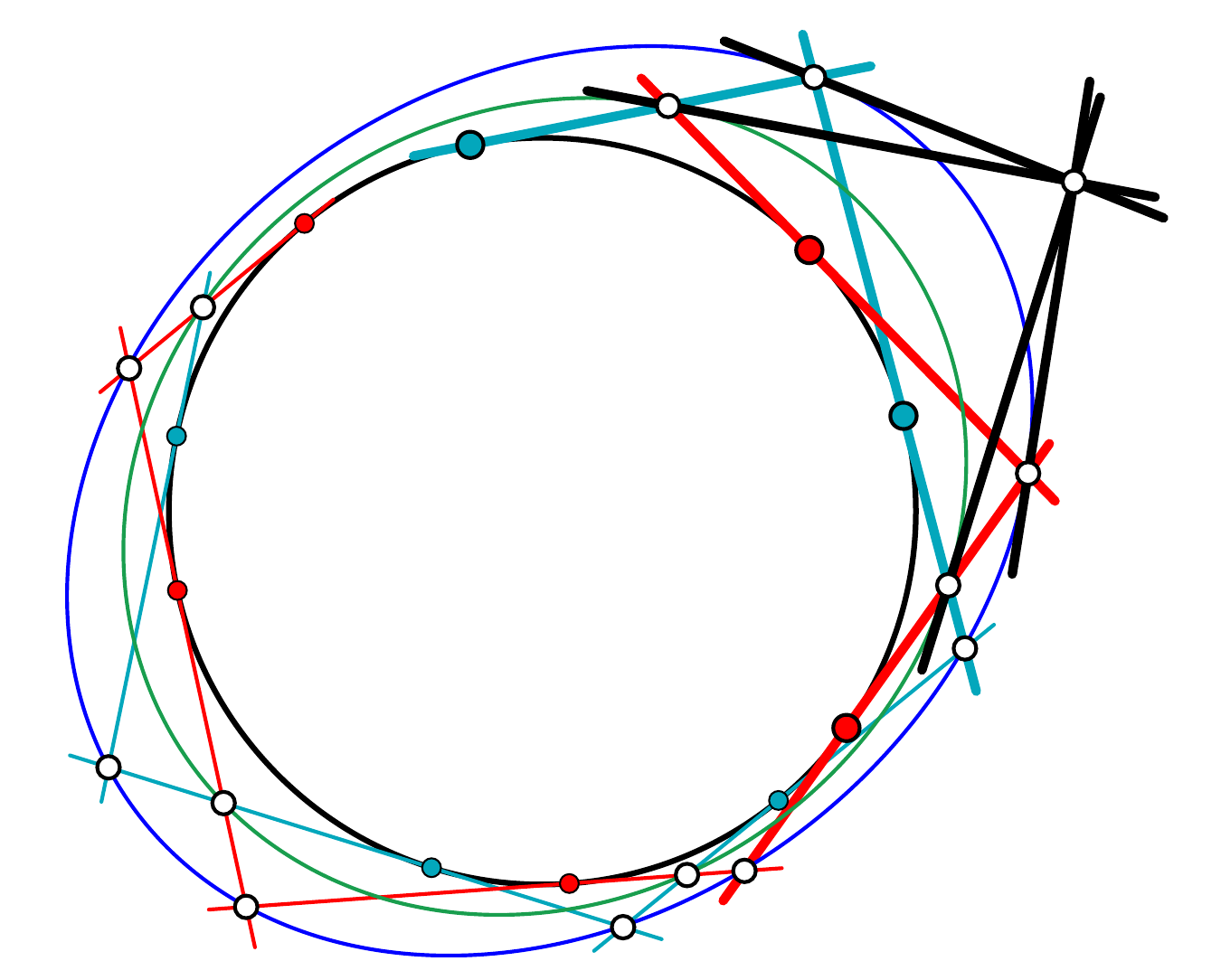}
\begin{picture}(0,00)
\put(-165,165){\footnotesize $0$}
\put(-100,148){\footnotesize $b$}
\put(-80,114){\footnotesize $a$}
\put(-110,40){\footnotesize $2a$}
\put(-175,28){\footnotesize $3a$}
\put(-220,82){\footnotesize $b+3a$ }
\put(-195,155){\footnotesize $b+4a$}
\put(-100,60){\footnotesize $b+a$}
\put(-150,26){\footnotesize $b+2a$}
\put(-220,112){\footnotesize $4a$}
\put(-88,195){\footnotesize $P$}
\put(-38,105){\footnotesize $P'$}
\put(-122,190){\footnotesize $Q$}
\put(-82,81){\footnotesize $Q'\!\!\rightarrow$}

\end{picture}
\end{center}
\vskip-2mm
\captionof{figure}{Two Poncelet chains (red and cyan) on the same supporting conics and the core incidence statement. 
}\label{fig:CGTchain}
\end{figure}

We now 
consider one specific quadruple concurrence at one specific point of the $(3m_4)$ configuration. 
Since the situation is totally symmetric it is sufficient to prove the occurrence of one such concurrence to show the existence of 
all $3m$ such concurrences.
Figure \ref{fig:star3}
highlights the core of the situation, around one of the quadruple intersections of the green and blue lines.

\medskip

The crucial fact we need to prove for our construction to work is (loosely speaking)
{\it ``Whenever the local situation in Figure \ref{fig:star3} comes from a Poncelet grid then the two green and two blue tangents meet in a point"}.
In essence, this is a statement about a local incidence configuration.
For this we consider four lines $l_0,\ l_a,\ l_b, \ l_{a+b}$ from a Poncelet grid $L$ and show that they locally generate this incidence pattern.
From Theorem~\ref{ponceletGrid1} we know that the points in
$\wedge_i(L)$ 
lie on conics $\mathcal{C}_i$. The following Lemma isolates the core incidence pattern.

\begin{lemma}\label{corelemma}
Let $l_0,l_1,l_2,\ldots$ be the lines of a Poncelet chain  tangent to a  
conic~$\mathcal{X}$, and let $a$, $b$ be such that
$l_0,\ l_a,\ l_b, \ l_{a+b}$ are pairwise distinct lines.
Let $\mathcal{B} = \mathcal{C}_a$ and 
$\mathcal{G} = \mathcal{C}_b$  be the conics passing through the Poncelet grid points for rings $a$ and $b$ respectively. Consider four points 
$
P=l_0\wedge l_a,\ 
P'=l_b\wedge l_{a+b},\ 
Q=l_0\wedge l_b,\ 
Q'=l_a\wedge l_{a+b}.\ 
$
Then the tangents 
$
\mathcal{B}\cdot P,\ 
\mathcal{B}\cdot P',\ 
\mathcal{G}\cdot Q,\ 
\mathcal{G}\cdot Q'
$ 
meet in a point.
\end{lemma}

\noindent
In fact we may consider the Poncelet subchains
\[
l_0,l_a,l_{2a},l_{3a},\ldots  \quad\text{and}\quad l_b,l_{a+b},l_{2a+b},l_{3a+b},\ldots \]
as two  different Poncelet chains supported by the same conics.

 This highlights the fact that the essence of the last lemma is more of a continuous nature in which $b$ might even vary smoothly.
The situation is illustrated in Figure~\ref{fig:CGTchain}. The two Poncelet chains are colored red and cyan. 
A self-contained proof of this lemma by direct calculation is presented  in Appendix~A of this article. 
The spirit of that proof is very much based on a coordinate-level approach and uses invariant theoretic arguments. In that sense it is similar
to the approaches we take in the companion paper \cite{BGRGT24b}.
 Theorem~\ref{constrthm} can be directly derived  from Lemma~\ref{corelemma}: 

\begin{proof} {(\it of Theorem \ref{constrthm}):}
We only focus on the occurrence of one quadruple coincidence, since the rest follows by symmetry.
We defined
$L_a=\lns[0](\pts[a](L)), 
L_b=\lns[0](\pts[b](L)) $
for an initial collection $L=(l_0,l_1,l_2,\ldots)$ of  lines of a Poncelet polygon.
The role of the two sequences in Lemma~\ref{corelemma} is played by the two 
subsequences
 \[
 (l_i,l_{i+a},l_{i+2a},l_{i+3a},\ldots)\quad
 \mathrm{ and }\quad
 (l_{i+b},l_{i+b+a},l_{i+b+2a},l_{i+b+3a},\ldots),\] 
 with $i$ being the index of an initial point. As usual indices are counted modulo~$m$.
 Both sequences are Poncelet chains with respect to the same conics:
   the circumscribed conic $\mathcal{X}$  and $\mathcal{B}=\mathcal{C}_a$, the conic on which the points $\pts[a](L)$ lie. 
 Similarly, the intersections of $l_{i+k\cdot a}$ and   $l_{i+b+ k\cdot a}$
 lie on the conic  $\mathcal{G}=\mathcal{C}_b$ for $k=0,1,2,\ldots$.   
 By a suitable index shift we may assume $i=0$. Applying Lemma~\ref{corelemma} after this shift we get
 \[
 l_i\wedge l_{i+a},\quad
 l_{i+b}\wedge l_{i+a+b},\quad
 l_i\wedge l_{i+b},\quad
 l_{i+a}\wedge l_{i+a+b},
 \]
 the tangents to the respective conics meet in a point. This is exactly what we want to prove.
\end{proof}
 \medskip

\subsubsection{Chasing indices}

The previous considerations yield a proof of Theorem~\ref{constrthm}, and ensure that from each Poncelet $m$-gon with $m\geq 7$ we can construct a configuration 
with $3m$ points and lines such that on each line there are (at least) 4 points and through each point there are (at least) 4 lines, i.e.\ we have a pre-$(n_4)$ configuration.

Nevertheless, it does not yet prove Theorem A which makes much more specific claims about the labels and indices of each of the constructed points and lines. 
To derive it at that point we have to create a careful bookkeeping of how we apply Lemma~\ref{corelemma}. Refer to Figure~\ref{fig:chase} for relations to the drawing.
Recall that the points were constructed by the procedure explained in Section~\ref{mainconstruction}. Each ring of points was related to one of the indices $a$, 
$b$, $c$.
We assume that the blue conic and lines were associated with the index $a$, and the green conic is associated with index $b$. 
Let $L=(l_1,l_2\upto l_m)$ be the lines of the initial central Poncelet polygon.

A careful analysis of the construction behind Lemma~\ref{corelemma} shows that we get the following statement that allows us to swap indices around an operator 
$\pts[0]$, or dually $\lns[0]$:
\begin{lemma}\label{lem:commute}
With the settings above
 we get
\[
\pts[b]\cdot\lns[0]\cdot\pts[a]
(L)\ = \ 
\pts[a]\cdot\lns[0]\cdot\pts[b]
(L).
\]
\end{lemma}

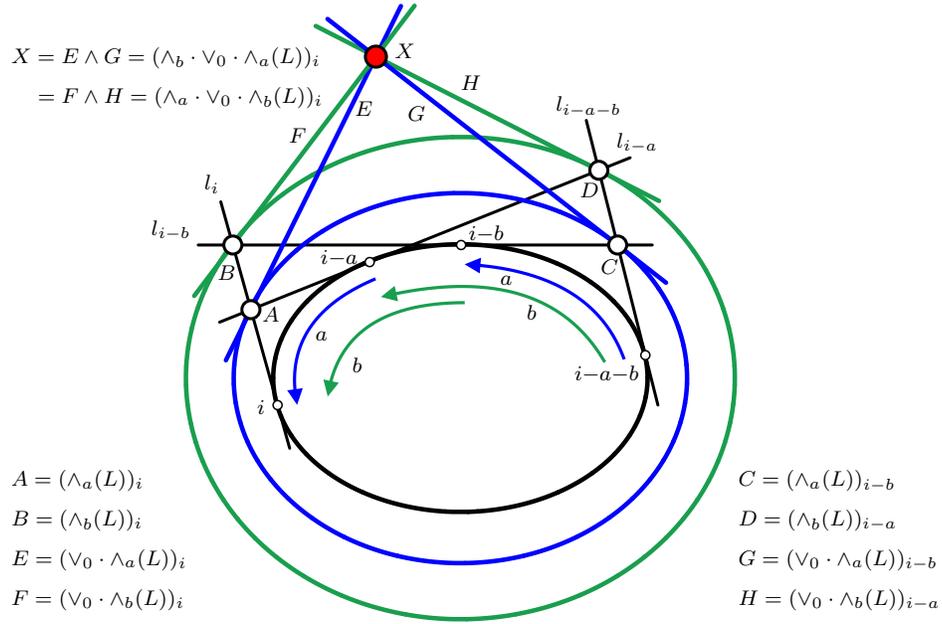
\begin{figure}[H]
\begin{center}
\includegraphics[width=.65\textwidth]{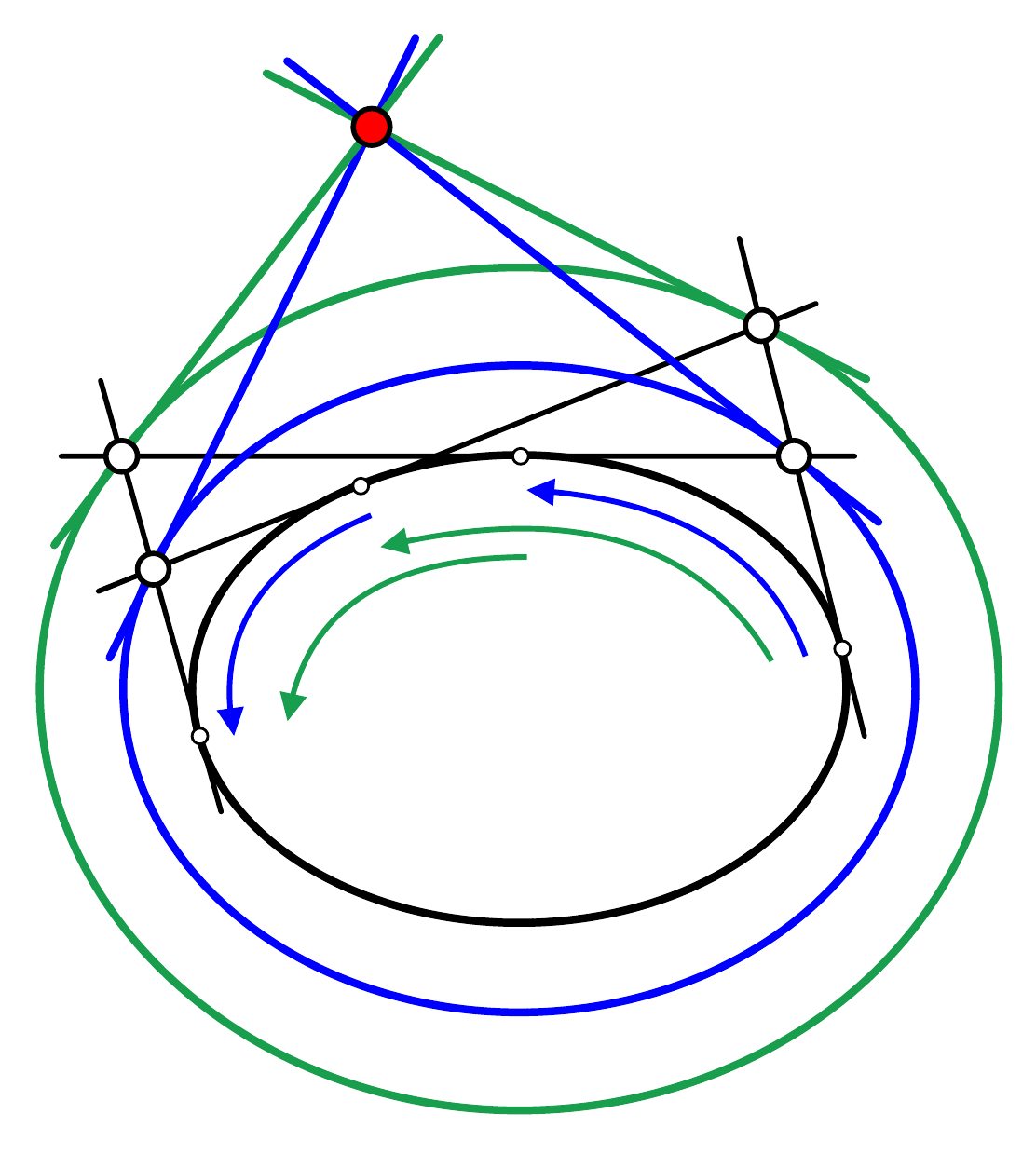}
\begin{picture}(0,0)
\put(-170,115){\footnotesize $a$}
\put(-100,136){\footnotesize $a$}
\put(-90,123){\footnotesize $b$}
\put(-156,103){\footnotesize $b$}
\put(-168,144){\footnotesize $i{-}a$}
\put(-112,153){\footnotesize $i{-}b$}
\put(-192,87){\footnotesize $i$}
\put(-72,100){\footnotesize $i{-}a{-}b$}
\put(-212,172){\footnotesize $l_i$}
\put(-56,188){\footnotesize $l_{i-a}$}
\put(-79,202){\footnotesize $l_{i-a-b}$}
\put(-232,155){\footnotesize $l_{i-b}$}

\put(-190,122){\footnotesize $A$}
\put(-207,138){\footnotesize $B$}
\put(-62,140){\footnotesize $C$}
\put(-70,169){\footnotesize $D$}

\put(-180,190){\footnotesize $F$}
\put(-155,200){\footnotesize $E$}
\put(-135,198){\footnotesize $G$}
\put(-115,210){\footnotesize $H$}
\put(-140,222){\footnotesize $X$}

\put(-285,60){\footnotesize $A=(\pts[a](L))_i$}
\put(-285,45){\footnotesize $B=(\pts[b](L))_i$}

\put(-285,30){\footnotesize $E=(\lns[0]\cdot\pts[{a}](L))_{i}$}
\put(-285,15){\footnotesize $F=(\lns[0]\cdot\pts[{b}](L))_{i}$}

\put(-285,220){\footnotesize $X=E\wedge G=(\pts[{b}]\cdot\lns[0]\cdot\pts[{a}](L))_i$}
\put(-275,205){\footnotesize $=F\wedge H=(\pts[{a}]\cdot\lns[0]\cdot\pts[{b}](L))_i$}

\put(-10,60){\footnotesize $C=(\pts[{a}](L))_{i-b}$}
\put(-10,45){\footnotesize $D=(\pts[{b}](L))_{i-a}$}
\put(-10,30){\footnotesize $G=(\lns[0]\cdot\pts[{a}](L))_{i-b}$}
\put(-10,15){\footnotesize $H=(\lns[0]\cdot\pts[{b}](L))_{i-a}$}

\end{picture}
\end{center}
\kern-3mm
\captionof{figure}{A detailed analysis of the line labels in the construction.} \label{fig:chase}
\end{figure}

\begin{proof}
Consider a concrete intersection of four lines in the construction of 
Section~\ref{mainconstruction}. Assume that the intersection comes from parameters $a$ (blue) and $b$ (green).
This means that from the lines $l_i$ tangent to the central conic exactly four are used in the partial construction that leads to the intersection.
Assume that these lines are $l_i,l_{i-a},l_{i-b}, l_{i-a-b}$.
Differences of indices between those lines that meet on the blue conic
must be $a$ and differences of indices between those lines that meet on the green conic must be $b$. Refer to Figure~\ref{fig:chase}  for the labelling.
The two points on the blue and on the green conic
are intersections of those lines. They are marked $A,B,C,D$ in the picture. From them we get:
\[
\begin{array}{lccl}
A=l_i\wedge l_{i-a}=(\pts[a](L))_i&&&
B=l_i\wedge l_{i-b}=(\pts[b](L))_i
\\[2mm]
C=l_{i-b}\wedge l_{i-a-b}=(\pts[a](L))_{i-b}&&&
D=l_{i-a}\wedge l_{i-a-b}=(\pts[b](L))_{i-a}
\\
\end{array}
\]
The corresponding tangent lines at those points are
\[
\begin{array}{lccl}
E=(\lns[0]\cdot \pts[a](L))_i&&&
F=(\lns[0]\cdot \pts[b](L))_i
\\[2mm]
G=(\lns[0]\cdot \pts[a](L))_{i-b}&&&
H=(\lns[0]\cdot \pts[b](L))_{i-a}
\\
\end{array}
\]
Finally, according to Lemma~\ref{corelemma} the red point can be derived in two different ways: Either by intersecting the blue lines $E$ and $G$ or by intersecting the green lines $F$ and $H$. 
The shift between $E$ and $G$ is $a$ while the shift between $F$ and $H$ is $b$.
We get:
\[
\begin{array}{rclcl}
X&=&E\wedge G&=&(\lns[0]\cdot \pts[a](L))_i \wedge (\lns[0]\cdot \pts[a](L))_{i-b}\\[2mm]
&=&F\wedge H&=&(\lns[0]\cdot \pts[b](L))_i\wedge 
(\lns[0]\cdot \pts[b](L))_{i-a}.\\
\end{array}
\]

\noindent
In other words,
\[
X\ = \ 
(\pts[b]\cdot\lns[0]\cdot \pts[a](L))_{i}\ = \ 
(\pts[a]\cdot\lns[0]\cdot \pts[b](L))_{i}.
\]
Since $i$ was generic, we get
\[
\pts[b]\cdot\lns[0]\cdot \pts[a](L)\ = \ 
\pts[a]\cdot\lns[0]\cdot \pts[b](L).\]
This proves the claim.
\end{proof}

\medskip

\noindent
We are finally in the position to prove Theorem A.

\begin{proof}(Of Theorem A)
For a Poncelet $m$-gon $P$ and indices $a,b,c$ less than $m/2$, we want to show
\[
\pts[c]\cdot\lns[b]\cdot
\pts[a]\cdot\lns[c]\cdot
\pts[b]\cdot\lns[a]
(P)=P.
\]

\noindent
We assume that $P$ is a Poncelet  $m$-gon $P$ and from it we first derive a  sequence of lines $L$ by the operation
\[L=
\lns[c]\cdot
\pts[0]\cdot\lns[a]
(P).
\] 
By the Poncelet grid theorems and its dual, while getting from $P$ to $L$ every intermediate construction step produces a Poncelet polygon of points or lines. Thus $L$ are the sides of a Poncelet Polygon as well and we can also get back by observing that
\[P=
\pts[a]\cdot
\lns[0]\cdot\pts[c]
(L).
\] 
By Lemma~\ref{lem:commute} we also have
\[P=
\pts[c]\cdot
\lns[0]\cdot\pts[a]
(L).
\] 

\noindent
Now we consider the following chain of reasoning:
\begin{align*} \pts[c]\cdot\lns[b]\cdot\pts[a]\cdot\lns[c]\cdot\pts[b]\cdot\lns[a](P) &= \pts[c]\cdot\lns[b]\cdot\pts[a]\cdot\lns[c]\cdot\pts[b]\cdot\underbrace{\lns[a]\cdot(\pts[a]}_{\text{id}}\cdot\,\lns[0]\cdot\pts[c](L)) \\
&= \pts[c]\cdot\lns[b]\cdot\pts[a]\cdot\lns[c]\cdot\underbrace{\pts[b]\cdot\lns[0]\cdot\pts[c]}_{\pts[c]\cdot\lns[0]\cdot\pts[b]}(L) \\
&=\pts[c]\cdot\lns[b]\cdot\pts[a]\cdot\underbrace{\lns[c]\cdot\pts[c]}_{\text{id}}\cdot\,\lns[0]\cdot\pts[b](L) \\
&= \pts[c]\cdot\lns[b]\cdot\underbrace{\pts[a]\cdot\lns[0]\cdot\pts[b]}_{\pts[b]\cdot\lns[0]\cdot\pts[a]}(L)\\
&= \pts[c]\cdot\underbrace{\lns[b]\cdot\pts[b]}_{\text{id}}\,\cdot\,\lns[0]\cdot\pts[a](L)\\
&= \pts[c]\cdot\lns[0]\cdot\pts[a](L)\\[2mm]
&= P
\end{align*}

\noindent
This finally finishes the proof of Theorem A.
\end{proof}

\noindent
If one compares the above proof to Figure~\ref{fig:star2},
one can literally recover our construction and the applications of Lemma~\ref{lem:commute}, which is in essence the ``four-tangents-meet-in-a-point'' statement of Lemma~\ref{corelemma}. Starting with one of the rings of points (say the green one) we first construct the inner set of (black) lines $L$. Every application of Lemma~\ref{lem:commute} in the above  sequence  of cancellations corresponds to jumping from one  ring of lines to another one that shares the same ring of points. 

\subsection{Closely related topics}

We presented the proof of our main Theorems A to C in a very constructive way. In this section we want to relate our construction to other concepts from the theory of discrete integrable systems. 
Each of these approaches is capable of deriving independent proofs of the our main theorem.

Poncelet's Porism is in a center of rich connections of various fields like elliptic functions, dynamical systems,
integrability, differential geometry, elementary geometry, the geometry of billiards and many more.
In this section we will describe different ways to attack some of our statement based on considerations from incircle nets and from elliptical billiards.

\subsubsection{Chasles--Graves Theorem}

There is an intimate relation between  Lemma~\ref{corelemma}
and a famous statement that was first discovered in 1843.  Since this connection provides lots of geometric insight we will elaborate on it here.

The Chasles--Graves Theorem (CGT) can be found in various sources in various
formulations~\cite{AlkZas07,BaFa91,Ber87,IzTa17} (and unfortunately with various degrees of correctness).
It goes back to Chasles and Graves (see also Darboux  \cite{ChaGra1841, Cha1865,Dar17}). Variants can also be found in Reye's work \cite{Rey1896}. 
The statement is about 4 lines tangent to a central conic.
In what follows we restrict ourselves to the case that the conics are ellipses if not explicitly stated otherwise. By this we avoid some of the intricacies related to
orientation, and the specific choice of intersections between a conic and a line.
Those intricacies are the source of several misinterpretations of this theorem that can be found in the literature.

We here literally quote a version of this statement 
that can be found in \cite{IzTa17}. This formulation is particularly useful in our context. The proof there is derived
via the geometry of billiards.

\begin{theorem}\label{graves} {\bf (CGT)}
Let $A$ and $B$ be two points on an ellipse. Consider the quadrilateral $ABCD$, made by the pairs of tangent lines from $A$ and $B$ to a confocal ellipse.

\noindent
{\it (1):} Its other vertices, $C$ and $D$, lie on a confocal hyperbola, and the quadrilateral is circumscribed about a circle.

\noindent
{\it (2):} Furthermore, if we intersect the lines $AC$, $AD$, $BC$ and $BD$, they have two additional intersections $E$ and $F$. Also, these two intersections lie on a conic 
(this time an ellipse) confocal to the other two.
\end{theorem}

Part 2 of this lemma slightly extends the original formulation from \cite{IzTa17}. However, its proof follows exactly the same pattern as the one for (1) and will be omitted here. The situation is illustrated in Figure~\ref{fig:incircle}.

\begin{figure}[t]
\begin{center}
\includegraphics[width=0.55\textwidth]{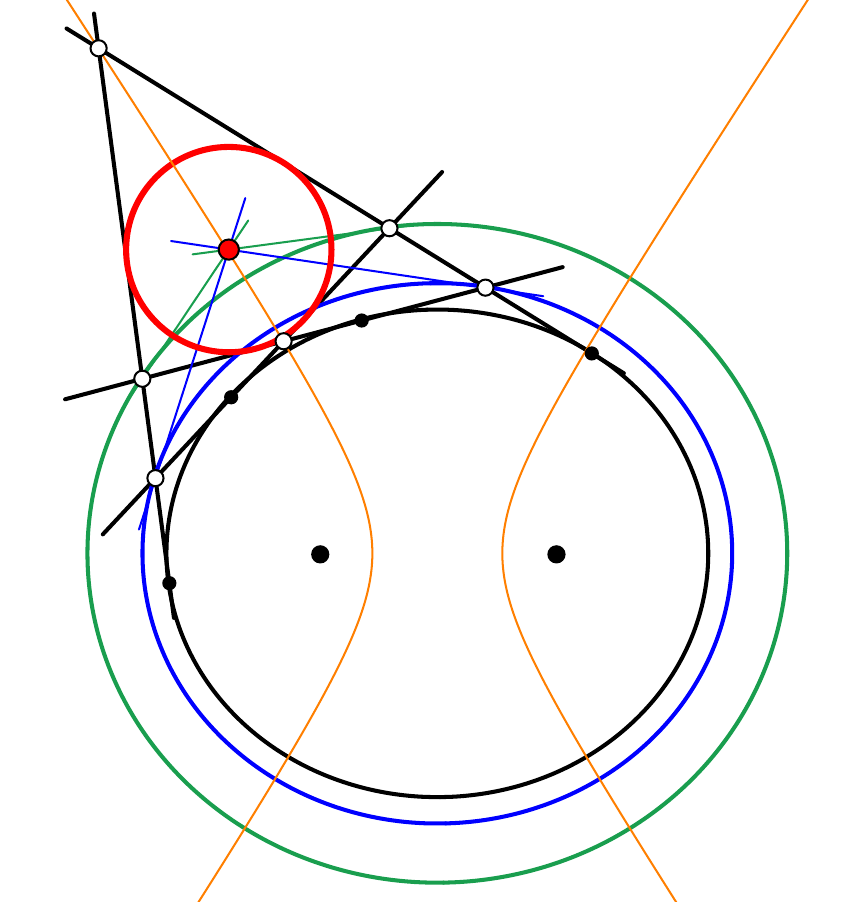}
\begin{picture}(0,0)
\put(-110,155){\footnotesize $B$}
\put(-170,118){\footnotesize $A$}
\put(-134,112){\footnotesize $C$}
\put(-167,189){\footnotesize $D$}
\put(-170,90){\footnotesize $E$}
\put(-90,140){\footnotesize $F$}
\end{picture}
\end{center}
\vskip-5mm
\captionof{figure}{The incircle property. The four tangents form the sides of a region that circumscribes a circle.} \label{fig:incircle}
\end{figure}
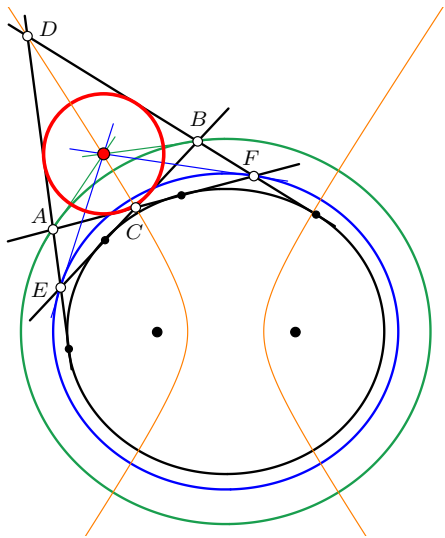

A pair of tangents to an inner ellipse from a point on an outer ellipse may be considered as a geometric reflection of a ray at the outer conic. 
This is the local situation around each of the points $A$, $C$, $E$ and $F$
(see \cite{Ta05}). 
Thus the angles between the two lines and the tangent are equal, or in other words the tangent at such a point is the angle bisector of the two lines meeting at the point. 
From the two possible angle bisectors it is the one pointing into the direction of the circle.

Comparing Figure~\ref{fig:incircle} with Figure~\ref{fig:chase} and Figure~\ref{fig:CGTchain} 
shows many structural commonalities:
The role of the conics $\mathcal{X}$, $\mathcal{G}$, $\mathcal{B}$ in the Poncelet grid
are now played by the three ellipses of the CGT. 
The conics being codependent in a Poncelet grid is (up to projective transformation) equivalent to the Euclidean statement that the ellipses are confocal. The CGT assumes tangency of the black lines to the inner conic and claims the existence of an incircle. Since the angle bisectors of two tangents to that circle pass through the center of the circle the tangents at the six points of the CGT meet in a point.
%
%
 In view of this close relation one might indeed be tempted to base the proof of 
Lemma~\ref{corelemma} on the Chasles--Graves Theorem (CGT). 
When we first encountered the similarity to our situation in Figure~\ref{fig:star3} to the CGT we were extremely optimistic that we could just quote the desired result from the literature about the CGT. It turned out that this is not the case. The references we found were either too weak for our situations, or they stated exactly what we wanted but turned out to be flawed when we checked the proofs and statements more closely. It comes as a surprise that a theorem of such a classical nature carries such subtleties that may easily lead to wrong formulations. Exactly these subtleties make it difficult to apply the CGT directly in the situation we need it for.
To demonstrate this we here explicitly give a flawed formulation that is similar to the ones we found in the literature.

\medskip
The subtle problem arises when one tries to combine the role of the tangents without giving an explicit way to relate the conic $\mathcal{B}$ 
 to the circle $\mathcal{C}$.
We here give a minimalistic version of a false statement (see \cite{RG24}) that can be found in the literature in similar ways.

\begin{notheorem}
{\bf (A flawed version of CGT):} Let $\mathcal{X}$ be an ellipse in the Euclidean plane and let $a,b,c,d$ be four distinct tangents to $\mathcal{X}$. Then the following statements are equivalent.
\begin{itemize}
\item[(i)] $a,b,c,d$ are tangent to a circle $\mathcal{C}$
\item[(ii)] The intersections  $P=a\meet b$ and $Q=c\meet d$ lie on a conic $\mathcal{R}$ confocal to $\mathcal{X}$.
\end{itemize}
Moreover, the tangents at $P$ and $Q$ to $\mathcal{R}$ meet in the center of the circle.
\end{notheorem}
\noindent
The problem here lies in the ``{\it Moreover$\ldots$}" part. The problem comes from the fact that the conic $\mathcal{R}$ may not be unambiguously defined from its properties stated in the statement. As a matter of fact this is  not the case for most of the drawings of the Chasles--Graves Theorem you will see. But there are some (not too degenerate) situations where this problem might arise.
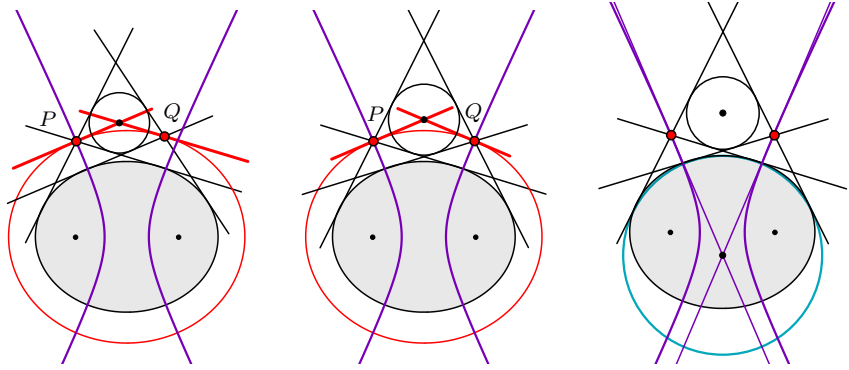
\begin{figure}[ht]
\centering
\!\!\includegraphics[width=.3\textwidth]{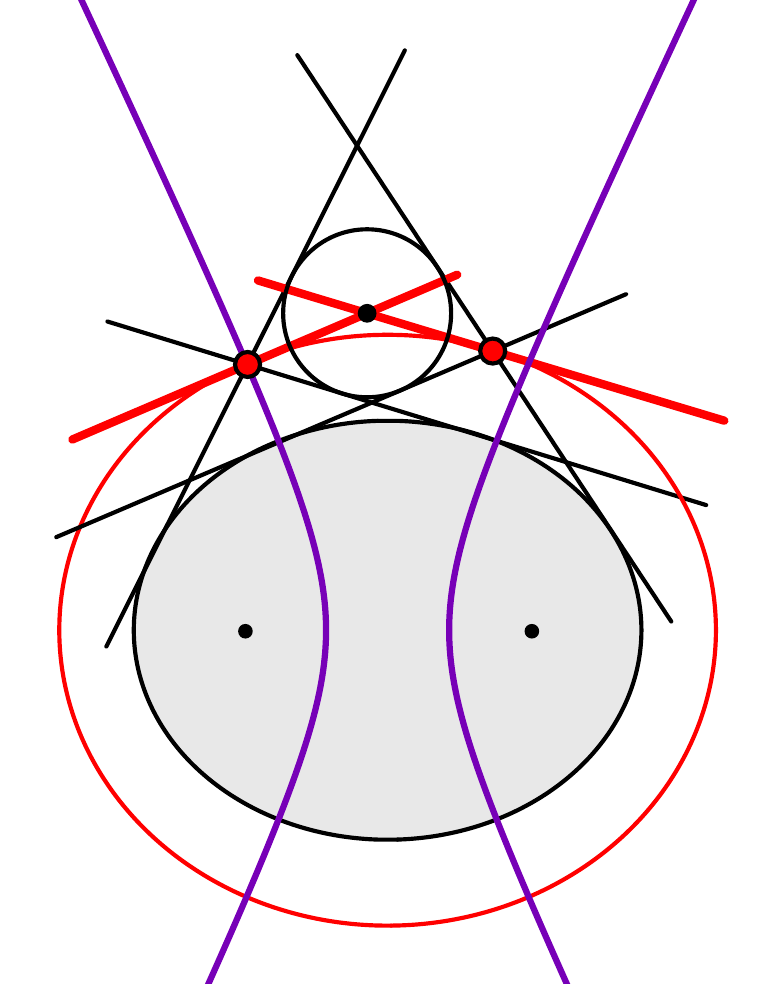}\;\;  
\includegraphics[width=.3\textwidth]{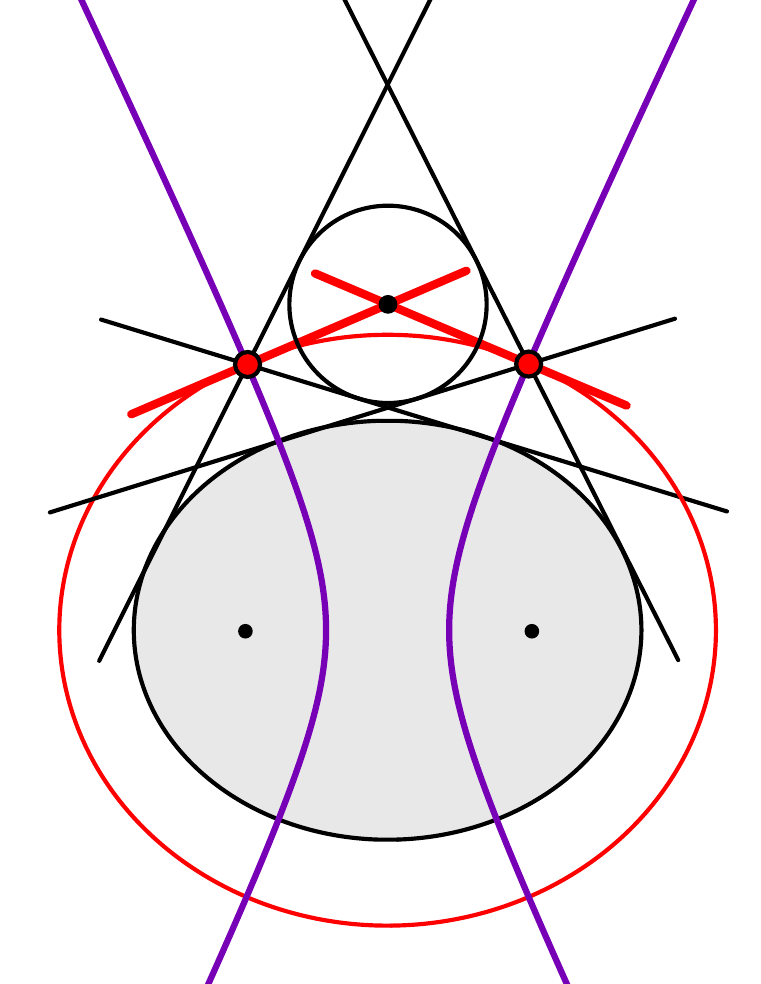}\;\; 
\includegraphics[width=.3\textwidth]{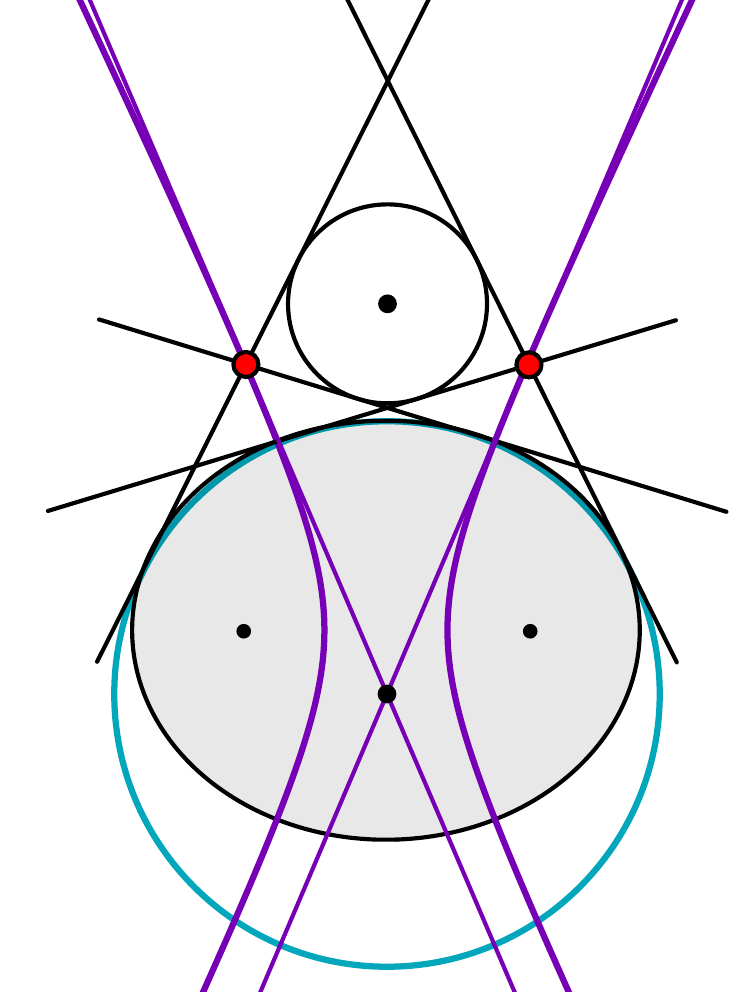}
\begin{picture}(0,0)
\put(-312,90){\footnotesize {$P$}}
\put(-265,93){\footnotesize {$Q$}}
\put(-188,92){\footnotesize {$P$}}
\put(-151,93){\footnotesize {$Q$}}
\end{picture}
\caption{The second conic through Point $P$ (purple conic in leftmost picture) may or may not pass as well through point $Q$ (in the middle image it passes through $Q$) and may generate a counterexample to the flawed version of CGT (right).}	
\label{fig:flaw}
\end{figure}

To see this consider the drawings in Figure \ref{fig:flaw}. The leftmost picture highlights the second confocal conic that passes through point $P$ (in purple). Usually, this conic does not pass through $Q$. However, in the particular situation
in which the points $P$ and $Q$ are symmetric with respect to the perpendicular bisector of the (real) foci of the conics then the role of the conic $\mathcal{R}$  may as well   be played by the purple conic (middle picture).
The right-hand picture shows the situation in which one only considers this conic (ignoring the red one). 
It is a conic {\it confocal to $\mathcal{X}$} and {\it passing through $P$ and $Q$} but its tangents do not pass through the center of the  (black) circle shown in the image tangent to $a,b,c,d$. 

However, the tangents pass through the center of {\it another} circle that arises in this situation. It is indicated in cyan in 
the rightmost picture of  Figure~\ref{fig:flaw}. This circle arises only in this particular geometric situation.
For this symmetric situation we may think of the situation as follows. The {\it suitable} choice for the red conic (the ellipse) is the black circle and the suitable choice for the purple conic (the hyperbola) is the  cyan  circle.
\medskip

In our situation this choice has to be explicitly created from the situation of the Poncelet grid and is not part of the hypotheses of the CGT. Specifically, we can derive the following information from the Poncelet grid situation:
\begin{itemize}
\item[(i)] Three confocal (i.e. codependent) conics 
$\mathcal{X}$, $\mathcal{B}$, $\mathcal{G}$.
\item[(ii)] a quadrilateral with two points in $\mathcal{B}$ and two points in $\mathcal{G}$ whose sides are tangent to $\mathcal{X}$.
 \item[(iii)] Tangents at those four points to the respective conics.
\end{itemize}
We want the corresponding tangents to meet in a point.
As the last example shows, this cannot be shown by the above hypotheses alone.
We need additional information that comes from the specific situation of being a Poncelet polygon.
In fact, it is possible to derive a proof of our main result based on the CGT but this would require additional homotopy or limit case arguments.

\subsubsection{Incircle nets}

\newcommand{\mybox}{%
           \mathrel{\raisebox{.1em}{%
           \reflectbox{\rotatebox[origin=c]{45}{$\square$}}}}}

\def\dwedge{\blacktriangleright\!\blacktriangleleft}

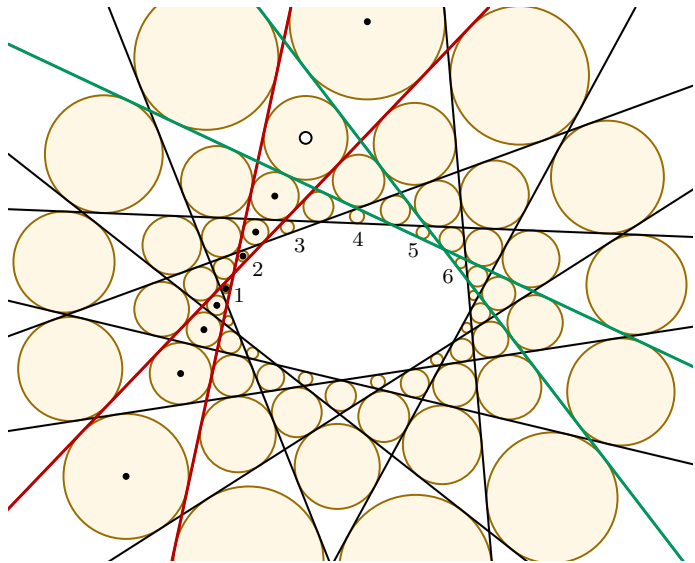
\begin{figure}[t]
\begin{center}
\includegraphics[width=0.75\textwidth]{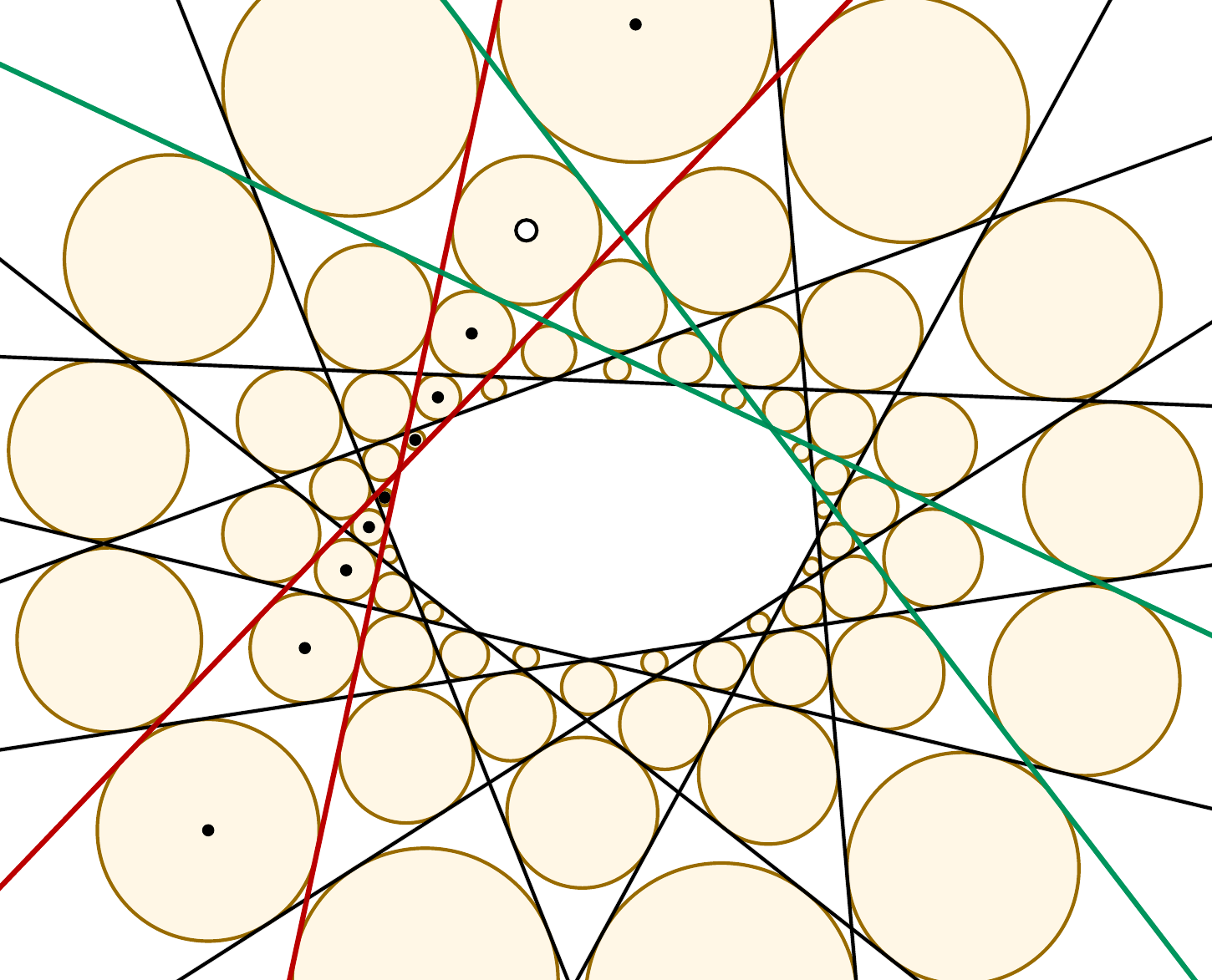}
\begin{picture}(0,0)
\put(-177,98){\footnotesize $1$}
\put(-170,108){\footnotesize $2$}
\put(-154,117){\footnotesize $3$}
\put(-132,119){\footnotesize $4$}
\put(-111,115){\footnotesize $5$}
\put(-98,105){\footnotesize $6$}
\end{picture}
\end{center}
\vskip-5mm
\captionof{figure}{The incircle property. The four lines forming the sides of a region that circumscribes a circle.}\label{fig:gridc1}
\end{figure}

Let us again consider the case that our Poncelet polygon is, without loss of generality, supported by a pair of confocal ellipses. We now take a holistic point of view. 
The fact that the local situation from Theorem~\ref{graves} occurs for many choices of supporting tangents in a Poncelet grid allows us to create many circles in the grid. 
The centers of such circles are potential candidates for points of an $(n_4)$ configuration.

Figure~\ref{fig:gridc1} illustrates how, in the cell complex created by the lines of a Poncelet grid,
each cell is circumscribing a circle. Such configurations of lines and circles were extensively studied in 
\cite{AkBo18}. The image also shows how circles that share the same two tangents have collinear centers.
In the picture, the centers for all circles simultaneously tangent to line $1$ and line $2$ are shown.
We will exploit exactly these collinearities for creating  lines in $(n_4)$ configurations.

To do so we need a precise system to label cells in such an arrangement. Here it is easiest to implicitly orient the lines to be able to talk about the relative position of a circle with respect to a line
(to do so in a specific way we rely on the fact that the situation is moved to the situation of confocal ellipses centered at the origin).
It turns out that the language most appropriate to describe the situation is the one of oriented halfspaces
in oriented projective geometry, or equivalently, topes in line arrangements in oriented matroid theory.
A beautiful treatment of oriented projective geometry can be found in the book of Stolfi \cite{Sto91}.
We want to avoid a full introduction into this topic here, since it is only a tool for bookkeeping in our case.
Instead, we will describe directly what happens. 

We consider our (projective) plane in which everything
takes place represented by pairs of antipodal points on the unit sphere.
A point $(x,y,z)$ on the unit sphere $S^2$ represents the corresponding projective point with these homogeneous coordinates. Thus the points $(x,y,z)$ and $(-x,-y,-z)$ represent the same geometric point. 
Now we deliberately distinguish between these two points and consider $S^2$ instead of $\mathbb{RP}^2$. Each point in the plane now has a positive and a negative representative. 
The positive halfspace associated to a line $l$ with homogeneous coordinates $(a,b,c)$ now is the set of all points $(x,y,z)$ with $ax+bx+cx>0$. 
Thus it contains all positive points on one side of the line and all negative points on the other side.

In  our Poncelet setup we now orient all lines such that  the center $(0,0,1)$  of the inner ellipse becomes a positive point with respect to these lines.
By $H_i^+$ and $H_i^-$  we denote the halfspaces corresponding to the positive and the negative side of line $i$.

\begin{figure}[t]
\begin{center}
\includegraphics[width=0.75\textwidth]{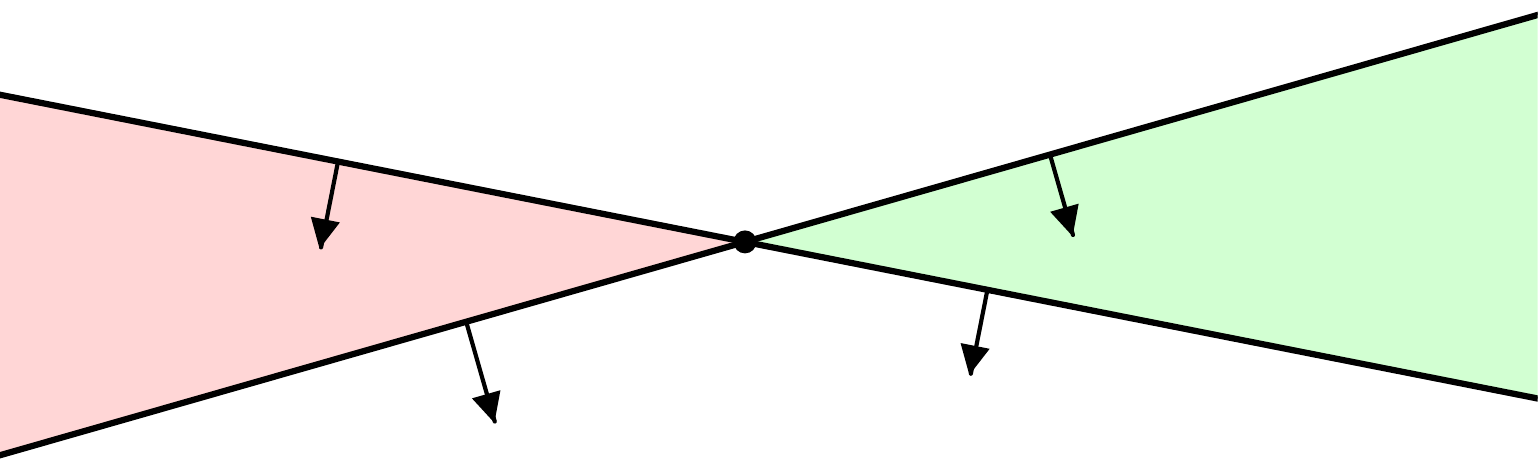}
\begin{picture}(0,0)
\put(-46,69){ $i$}
\put(-226,60){ $j$}
\put(-60,38){ {\Large $+$}}
\put(-240,30){ {\Large $-$}}
\end{picture}
\end{center}
\vskip-5mm
\captionof{figure}{The oriented double wedge $\dwedge_j^i$.} \label{fig:wedge}
\end{figure}

We denote the points at intersection of lines $i$ and $j$ by $\bullet_{i,j}$. 
The {\it double wedge} with its origin at point $\bullet_{i,j}$ can be characterised by

\[\dwedge^i_j:=H_i^+\cap H_j^-.\]

The situation is illustrated in Figure~\ref{fig:wedge}. There the points belong to the positive 
part of $\dwedge_j^i$ are marked green. Points in the red region geometrically correspond to
points in $\dwedge_j^i$ as well, but they should be considered equipped with a negative sign.
Circles whose centers are shown in  Figure~\ref{fig:gridc1} are in the double wedge 
$\dwedge^1_2$. Some are in the positive side, some in the negative side.

Notice that we have $\dwedge^i_j=-\dwedge^j_i$.
The absolute value of the difference of $i$ and $j$ indicates the ring in the Poncelet grid 
$\pts[{|i-j|}](L)$
to which the intersection $\bullet_{i,j}$ of the lines belongs. Observe that if we consider the two parts of a double wedge in the projective plane, they represent just one region.

A single circle can be precisely addressed by describing its relative position with respect to the four lines to which it is tangent. 
For instance, the circle marked by a white dot  in Figure~\ref{fig:gridc1} is characterised by being on the positive side of $1$ and on the negative side of $2$, and simultaneously 
being on the positive side of $6$ and on the negative side of $5$. 
We define a cell by intersecting two double wedges
\[
\raisebox{-.15em}{$\mybox^{a,b}_{c,d}$}
\ 
:=\ \dwedge^a_b \cap \dwedge^d_c 
\ =\ 
H_a^+\cap H_b^- \cap H_d^+\cap H_c^-.
\]
In particular, we get
\[
\mybox^{a,b}_{c,d}
\ =\ 
\mybox^{a,c}_{b,d}
\ =\ 
\mybox^{d,b}_{c,a}
\ =\ 
\mybox^{d,c}_{b,a}.
\]

The cell with the white circle center now becomes $\mybox^{1,2}_{5,6}$.
Notice that not all cells in our labelling system correspond to cells with circles that are shown in Figure~\ref{fig:gridc1}.
The cells shown with circles are those that have smallest gridwidth, and they are of the form $\mybox^{a,a+1}_{c,c+1}$. The $1$ occurring in this expression
encodes that the {\it combinatorial sidelength} is one unit.
Our system also allows us to address cells that are composed from more than one
 such elementary regions. In general, $\mybox^{a,b}_{c,d}$ describes a generalised rectangle.
A generalised {\it square}   with combinatorial sidelength $k$ has the form  
$\mybox^{a,a+k}_{c,c+k}=\mybox^{a,c}_{a+k,c+k}$.
Figure~\ref{fig:grid2} illustrates some combinatorial squares in a Poncelet grid.  The image also indicates that cells can extend via infinity---like the blue cell in the picture. 
The part that lies beyond infinity and comes back from the other side of the picture has to be considered negative. 

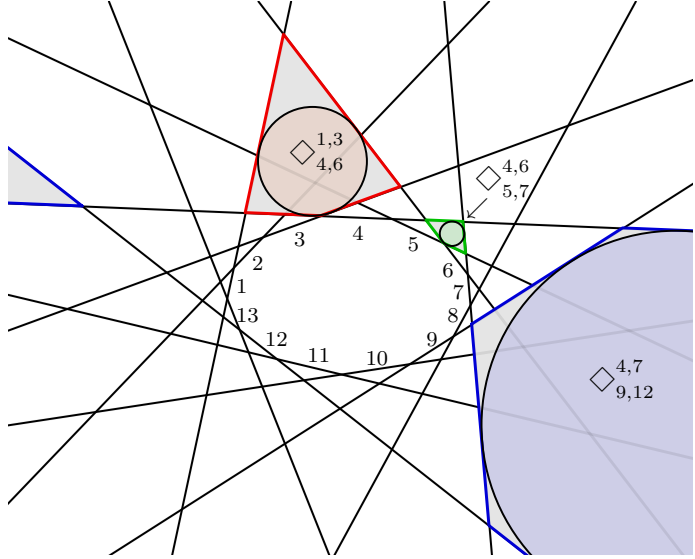
\begin{figure}[t]
\begin{center}
\includegraphics[width=0.75\textwidth]{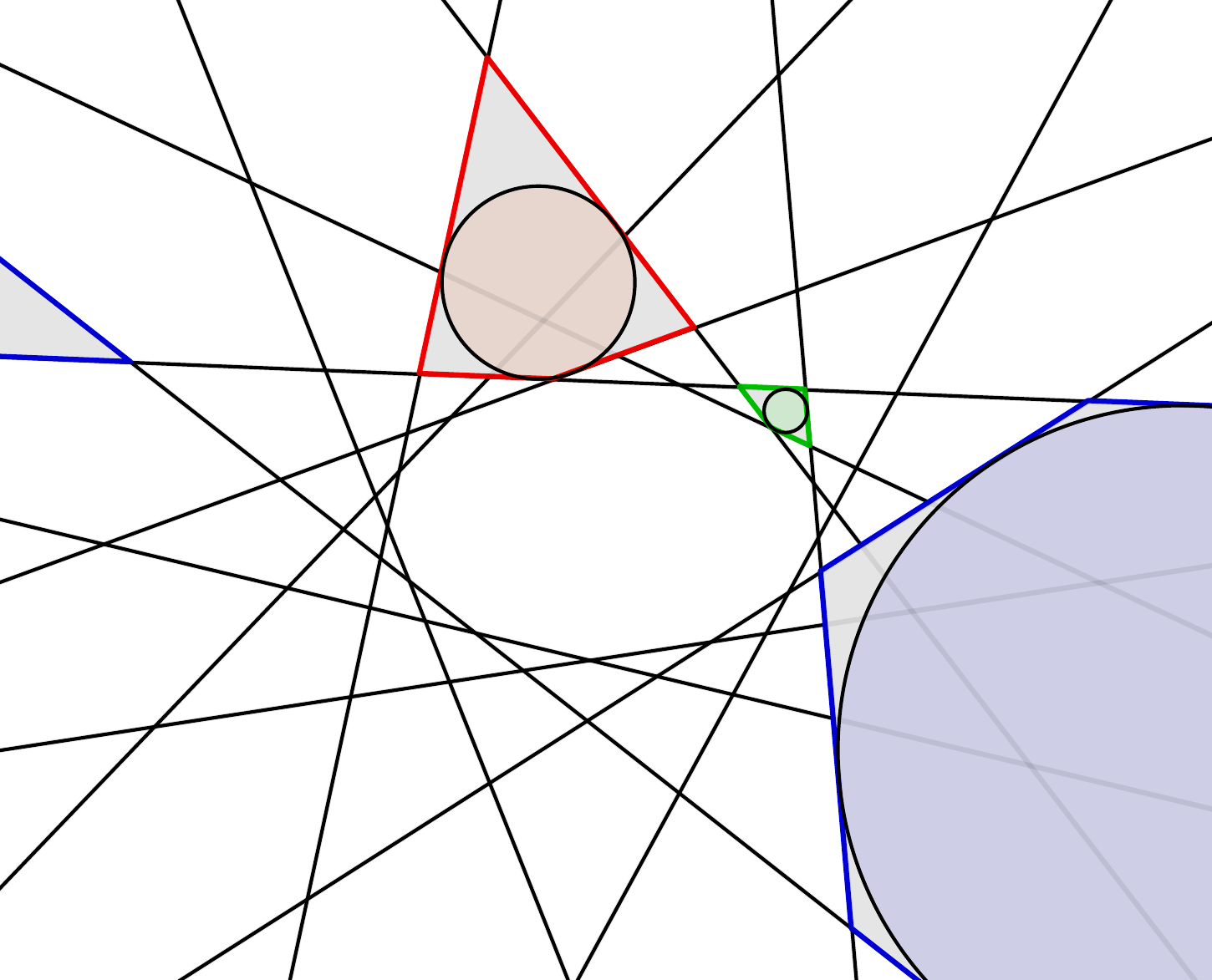}
\begin{picture}(0,0)
\put(-176,99){\footnotesize $1$}
\put(-170,108){\footnotesize $2$}
\put(-154,117){\footnotesize $3$}
\put(-132,119){\footnotesize $4$}
\put(-111,115){\footnotesize $5$}
\put(-98,105){\footnotesize $6$}
\put(-94,97){\footnotesize $7$}
\put(-96,88){\footnotesize $8$}
\put(-104,79){\footnotesize $9$}
\put(-127,72){\footnotesize $10$}
\put(-149,73){\footnotesize $11$}
\put(-165,79){\footnotesize $12$}
\put(-176,88){\footnotesize $13$}
\put(-156,148) {$\mybox^{1,3}_{4,6}$}
\put(-89,138){ $\mybox^{4,6}_{5,7}$}
\put(-89,129){\footnotesize $\swarrow$}
\put(-46,62){ $\mybox^{4,7}_{9,12}$}
\end{picture}
\end{center}
\vskip-5mm
\captionof{figure}{Various combinatorial squares in a Poncelet grid.} \label{fig:grid2}
\end{figure}

\noindent
A cell $\mybox^{a,b}_{c,d}$ has the corners $\bullet_{a,c}$, $\bullet_{b,c}$, $\bullet_{a,d}$ and $\bullet_{b,d}$. If the cell is a combinatorial square $\mybox^{a,b}_{a+k,b+k}$, then
the two corners $\bullet_{a,a+k}$ $\bullet_{b,b+k}$ come from the same ring of points in a Poncelet grid, namely the ring $\pts[k](L)$. This means that they lie on a common confocal conic. 
Hence, we can apply Theorem \ref{graves} and get an incircle for every square of any sidelength in the grid. The circles are also shown in Figure~\ref{fig:grid2}.
A little care is appropriate here. Four lines in the projective plane decompose the projective plane into 4 triangles and 3 quadrangles. 
The property of one of these quadrangles circumscribing a circle does not automatically imply that the other two quadrangles are circumscribing as well.
However, with our specification of the cell (that also takes the relative oriented position with respect to the lines into account) we exactly specify the cell in which an incircle exists. 
Taking all that together we get

\begin{lemma}
Each combinatorial square  $\mybox^{a,b}_{a+k,b+k}$ has an incircle.
\end{lemma}
 
If we without loss of generality assume $a<b$ and set  $l=b-a$, we may also represent our square by two shifts $l$ and $k$:
  \[
 \mybox^{a,b}_{a+k,b+k}
 \ = \ 
 \mybox^{a,a+l}_{a+k,a+k+l}.
\]
The two shifts $k$ and $l$ characterise the {\it type} $(k,l)$ of
a square cell. This type is intimately related to the rings in the Poncelet grid. The two pairs of opposite  sides of the cell meet in Poncelet grid points from the rings
$\pts[k](L)$ and $\pts[l](L)$. The index $a$ characterises the rotational position.
We have two ways to think about the cell of type  $(k,l)$. We may think of it as generated by intersecting two double wedges $\dwedge^a_{a+k}$ and  $\dwedge^{a+l+k}_{a+l}$
which have their apex in the ring $\pts[k](L)$. Or we can interchange the role of $k$ and $l$.  
The cell is also the intersection of double wedges $\dwedge^a_{a+l}$ and  $\dwedge^{a+l+k}_{a+k}$ which have their apex in the ring $\pts[l](L)$: therefore, a cell is associated with two rings. 
This essential fact is more or less a reformulation of Theorem~\ref{graves} and is reflected by the fact that $\mybox^{a,b}_{c,d}
\ =\ 
\mybox^{a,c}_{b,d}$.

We extend our notation   and use the symbols
 $\bigcirc^{a,b}_{c,d}$ for the incircle and
 $\odot^{a,b}_{c,d}$ for the incircle center of a square cell
 $\mybox^{a,b}_{c,d}$. We now generate configurations 
 from centers of incircles. For ease of notation
 we set ${}_{i}\odot_{k,l}:=\odot^{i,i+l}_{i+k,i+k+l}$.

\begin{theorem}
Let $L$ be the lines supporting the sides of a Poncelet $m$-gon and let 
$a,b,c$ be three positive, distinct indices smaller than $m/2$.
Then the collection of centers
\[\mathcal{P}:=
\left\{
{}_{i}\!\odot_{a,b}\ \big\vert \ 1\leq i\leq m
\right\}\cup
\left\{
{}_{i}\!\odot_{b,c}\ \big\vert \ 1\leq i\leq m
\right\}\cup
\left\{
{}_{i}\!\odot_{c,a}\ \big\vert \ 1\leq i\leq m
\right\}
\]
are the points of a pre-$(n_4)$ configuration.
\end{theorem}

\begin{figure}[h]
\begin{center}
\includegraphics[width=0.45\textwidth]{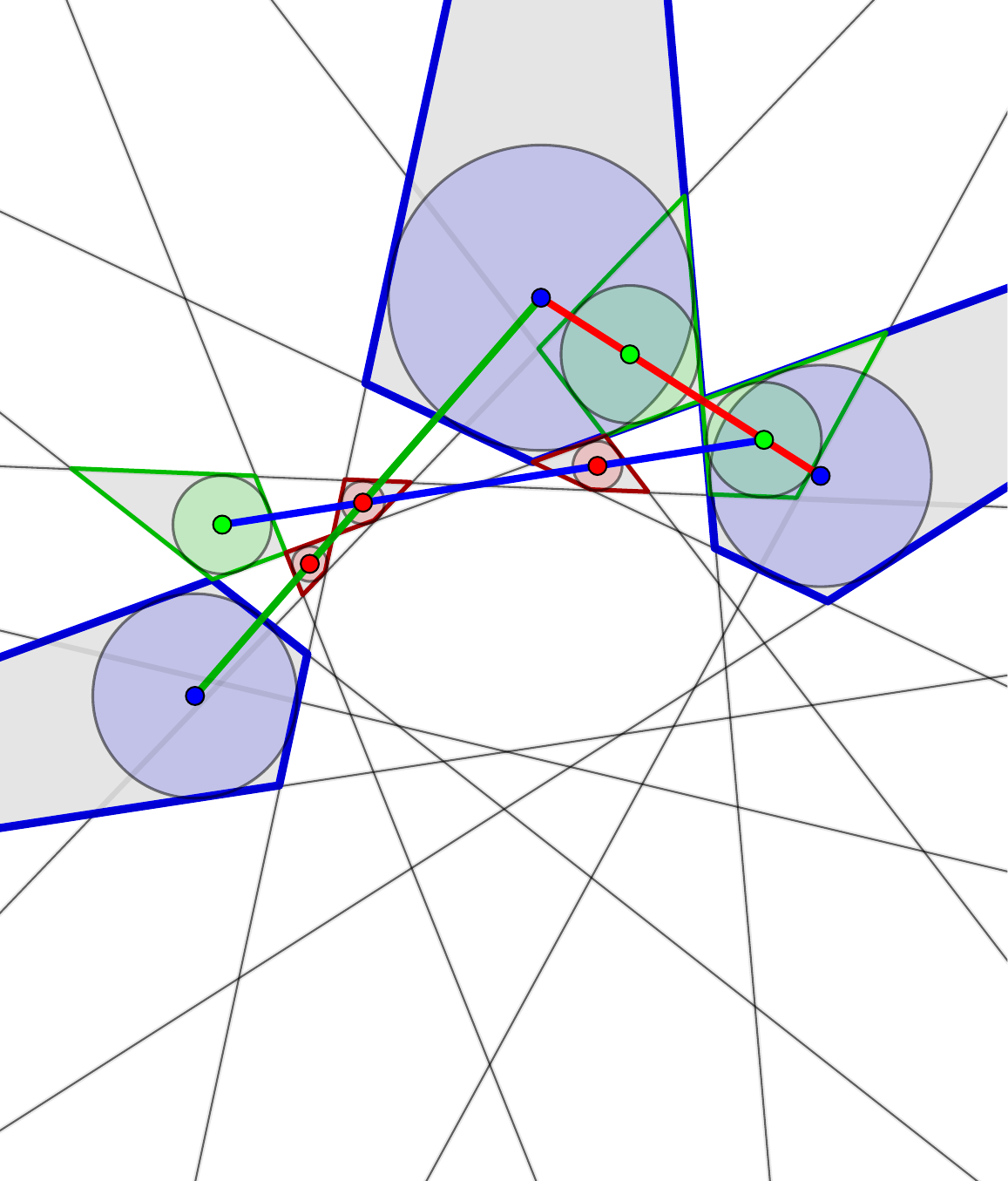}\hfill
\includegraphics[width=0.45\textwidth]{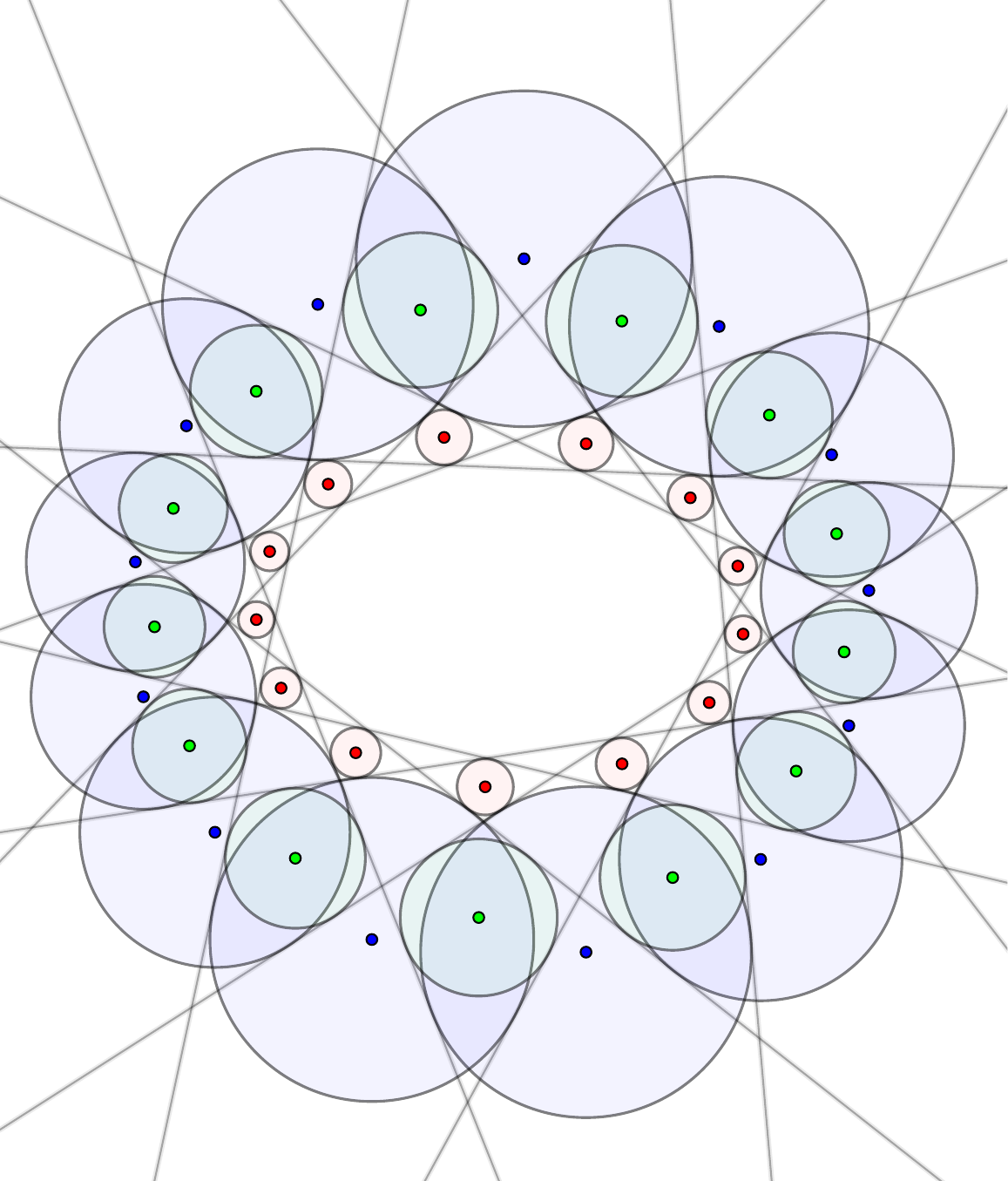}
\begin{picture}(0,0)
\end{picture}
\end{center}
\vskip-5mm
\captionof{figure}{Left: some collinearities between centers of points. Right: all centers belonging to an $(n_4)$ configuration.} \label{fig:grid}
\end{figure}

\begin{figure}[h]
\begin{center}
\includegraphics[width=.85\textwidth]{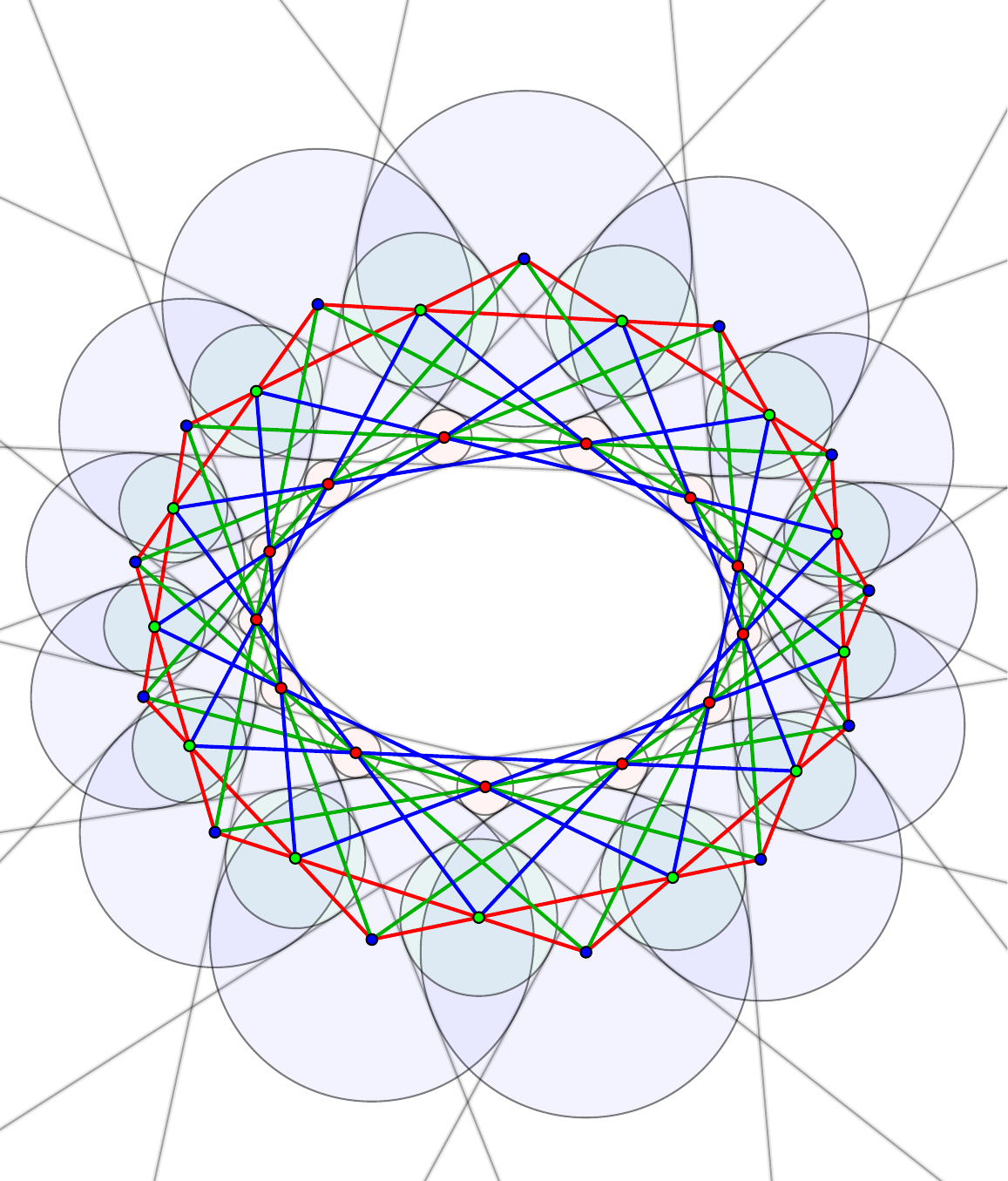}
\begin{picture}(0,0)
\end{picture}
\end{center}
\vskip-5mm
\captionof{figure}{The full relation between incircle nets and the construction for $(n_4)$ configurations.}\label{fig:13incircle}
\end{figure}

\begin{proof}
With everything we have already said, the proof is fairly easy. We have to show that each point is contained in 4 lines that contain 4 points each.
Fix $a,b,c$ as well as $i$ as in the Theorem.
Since all points in $\mathcal{P}$ are defined in the same way, it suffices to show the existence of the four lines for one of the points
$p=:{}_{i}\odot_{a,b}$.  It corresponds to the cell
$\mybox^{{\color{red}i},i+b}_{{\color{red}i+a},i+a+b}$.
Consider the points 
\[
{}_{i-b}\odot_{a,b}=\odot^{{\color{black}i-b},{\color{red}i}}_{{\color{black}i-b+a},{\color{red}i+a}};\qquad
{}_{i}\odot_{c,a}=\odot^{{\color{red}i},{\color{red}i+a}}_{{\color{black}i+c},{\color{black}i+a+c}};\qquad
{}_{i-c}\odot_{c,a}=\odot^{{\color{black}i-c},{\color{black}i-c+a}}_{{\color{red}i},{\color{red}i+a}}.
\]
Those indices marked in red show that  each of the corresponding circles is
tangent to the lines $i$ and $i+a$. Thus their centers and point $p$  are four collinear points.
All the points we considered are in the collection $\mathcal{P}$ given in the theorem.
In exactly the same way it can be shown that there are three more four-point lines:
one defined by the pair of lines $\{i,i+b\}$, one defined by $\{i+a,i+a+b\}$
and one defined by  $\{i+b,i+a+b\}$.
\end{proof}

\bigskip

\newpage
\subsubsection{Billiard geometry and local coordinate systems}

\rightline{\small \it Mathematics is the art of giving the same name to different things.}
\vskip1mm
\rightline{\small \it Poincaré}

\medskip

In this section,  we take a different perspective and relate billiard trajectories in ellipses and their relationship to Poncelet polygons to our $\wedge$ and $\vee$ constructions.
We first recall a few basic concepts. For a detailed exposition, see \cite{Ta05}. Consider an elliptic billiard table and draw the trajectory of a ball. 
The ball is reflected whenever it hits the cushion by the optical reflection law (so we do not do trickshots). Figure~\ref{fig:billiard} shows an infinite and a closed trajectory. 
It turns out that all segments of a trajectory are tangent to a single conic, a caustic, that is confocal to the billiard table boundary.
So if a trajectory returns to its initial starting conditions (position and direction), we actually create a Poncelet polygon. 
This relation between Poncelet chains and billiard trajectories can be exploited in various ways.

\begin{figure}[h]
\begin{center}
\includegraphics[width=0.45\textwidth]{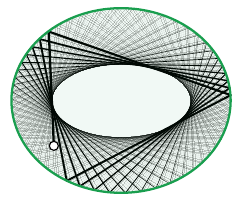}\hfill
\includegraphics[width=0.45\textwidth]{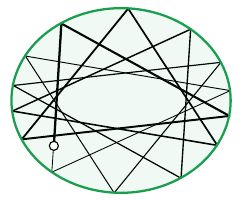}
\begin{picture}(0,0)
\end{picture}
\end{center}
\vskip-5mm
\captionof{figure}{Billiard on an elliptic table. An open and a closed trajectory.} \label{fig:billiard}
\end{figure}

The concept of billiards on elliptic tables is tightly interwoven with our
elaborations on Poncelet's Porism and $(n_4)$ configurations. 
Our Theorem A states that if we start with an arbitrary  Poncelet polygon $P$ and 
construct
$\pts[c]\cdot\lns[b]\cdot
\pts[a]\cdot\lns[c]\cdot
\pts[b]\cdot\lns[a](P)=P'$ 
we end up with $P$ again.
In a sense, this statement can be decomposed in two parts.
\begin{itemize}
\item[1.] The  points $P'$ are on the same conic as the one that supports $P$, and
\item[2.] on that conic they end up at the same position as the original points.
\end{itemize}
The billiard approach can be used to derive the second part relatively easily by introducing a certain distorted coordinate system that comes from the specific initial Poncelet polygon. In what follows we will sketch this line of reasoning.

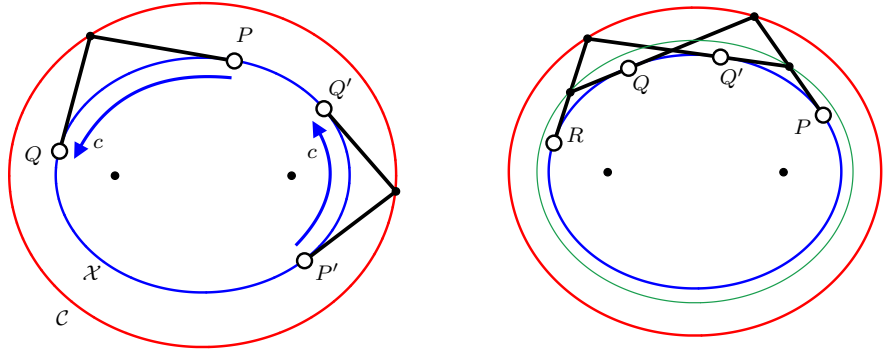
\begin{figure}[t]
\begin{center}
\includegraphics[width=0.45\textwidth]{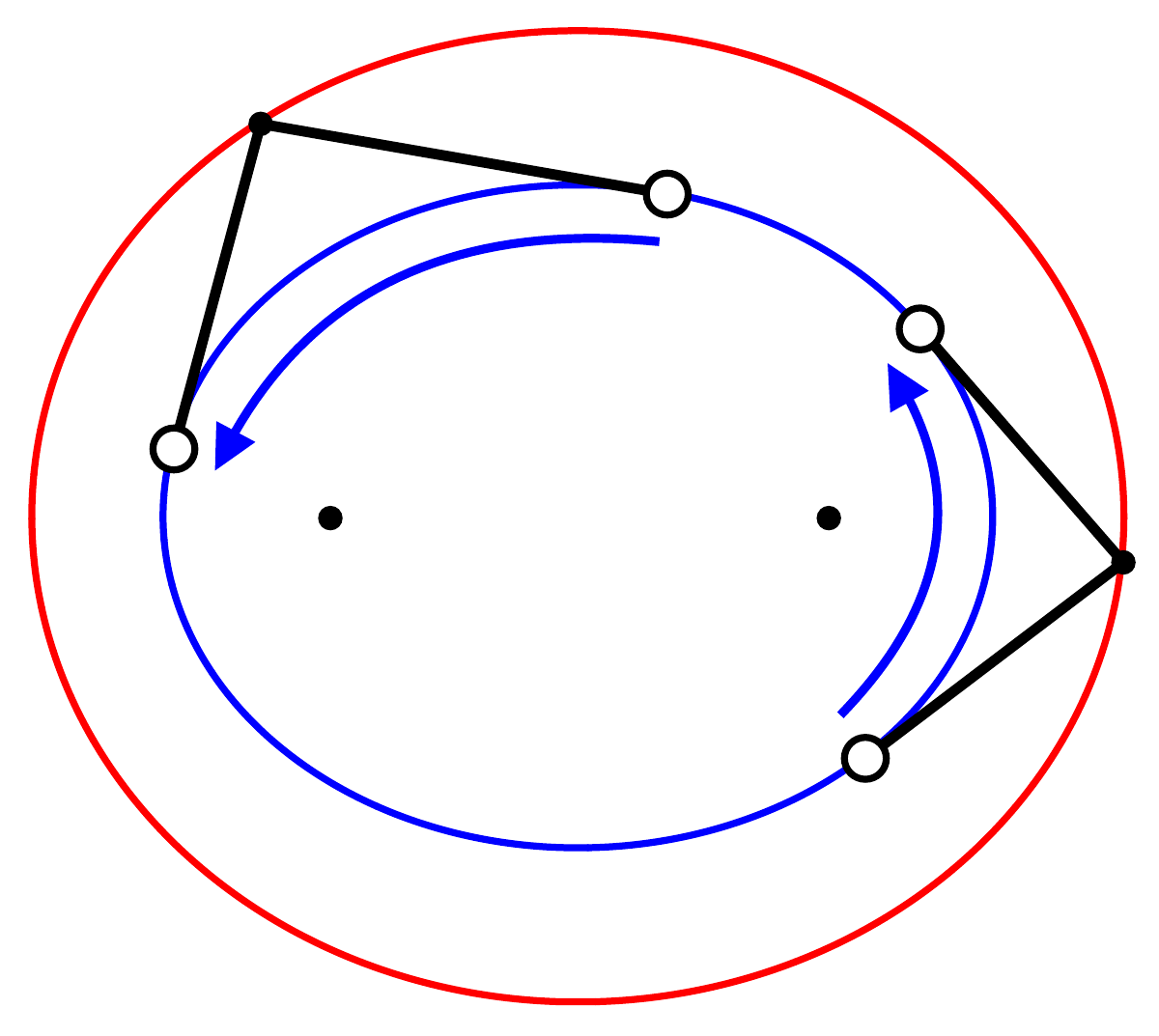}\hfill
\includegraphics[width=0.45\textwidth]{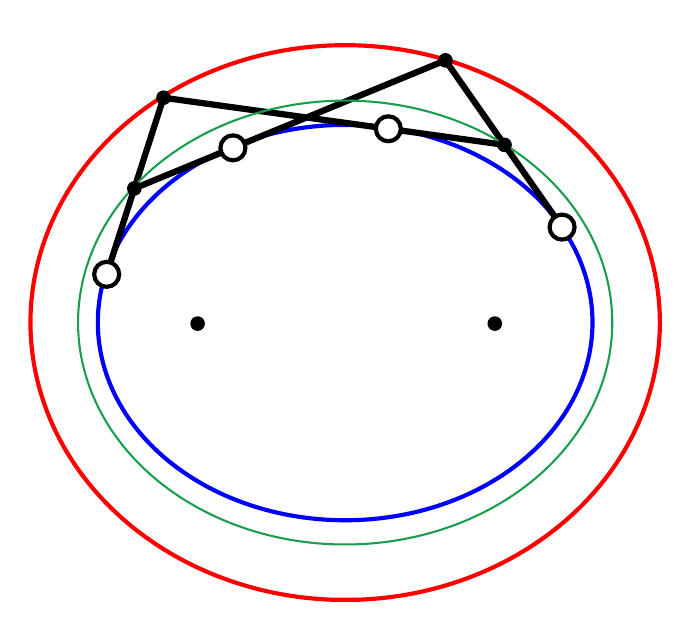}
\begin{picture}(0,0)
\put(-255,120){\footnotesize $P$}
\put(-335,75){\footnotesize $Q$}
\put(-225,30){\footnotesize $P'$}
\put(-220,100){\footnotesize $Q'$}
\put(-228,76){\footnotesize $c$}
\put(-309,80){\footnotesize $c$}
\put(-313,30){\footnotesize $\mathcal{X}$}
\put(-323,13){\footnotesize $\mathcal{C}$}

\put(-44,85){\footnotesize $P$}
\put(-105,102){\footnotesize $Q$}
\put(-72,104){\footnotesize $Q'$}
\put(-130,82){\footnotesize $R$}

\end{picture}
\end{center}
\vskip-5mm
\captionof{figure}{Shifts on an elliptic boundary. The left picture shows the same shift represented by two pairs of points. 
                          The picture on the right illustrates commutativity. 
                          We have two shifts $c=[P,Q]=[Q',R]$ and
                          $d=[P,Q']=[Q,R]$. It holds $c+d=[P,Q]+[Q,R]=[P,R]=[P,Q']+[Q',R]=d+c$.} \label{fig:shift}
\end{figure}

\medskip
We restrict ourselves to classes of confocal ellipses. 
We assume that the ellipses are equipped with a counterclockwise orientation. For the moment, we fix the two foci $f_1$ and $f_2$,
and only consider Poncelet $m$-gons whose inscribing and circumscribing ellipses have these focal points. 
Assume that the conic $\mathcal{X}$  to which the segments of the Poncelet polygon are tangent is given and fixed,
 and that the conic~$\mathcal{C}$ on which the points lie  is variable (but still with foci $f_1$ and $f_2$).
The possible ellipses  $\mathcal{C}$ form a one-parameter family.

Assume that an ordered pair of points $(P,Q)$ on $\mathcal{X}$ is given.
The tangents at these points do have an intersection that uniquely defines a corresponding ellipse~$\mathcal{C}$. 
We equip this conic with a sign depending on the angle between $P$ and $Q$ as
seen from the center of the conic. The sign is positive if the angle is less than $\pi$, zero if the angle is $0$ or $\pi$ and negative otherwise.
We consider another pair of points $(P',Q')$ to be equivalent to $(P,Q)$  if it defines the same signed conic $\mathcal{C}$. 
We call an equivalence class  $[P,Q]$ that arises that way a {\it shift}. 
We can define an {\it addition} on shifts by juxtaposition: $[P,Q]+[Q,R]=[P,R]$.
We can consider shifts as a kind of rotation along the conic by a certain amount (similar to rotations along a circle). 
However, the measurement is distorted by the specific geometry of the conic. Addition of shifts turns out to be commutative. 
This is a consequence of the incidence Theorem sketched in Figure~\ref{fig:shift} on the right. 
All in all, the shifts around an ellipse form a commutative group similar to rotations on the boundary of a circle. 
The neutral element is shift $[P,P]$ and the inverse of $[P,Q]$ is $[Q,P]$.

Shifts operate on the points of $\mathcal{X}$. If $c$ is a shift and $P\in\mathcal{X}$, 
then there is a unique point $Q\in\mathcal{X}$ such that $[P,Q]=c$. The operation of the shifts on $P$ is defined by 
$[P,Q]\circ P=Q$.
If we fix a starting point $P_0$ on the conic $\mathcal{X}$ and a shift $c$,
the sequence $P_{i+1}=c\circ P_i$ defines a Poncelet chain. It is completely determined by $P$ and $c$.
Since by Poncelet's Porism the closing of the chain is not dependent on the starting point, the property of closing after $m$ steps is entirely determined by $c$.
Thus for a given $m$ and $\mathcal{X}$, a Poncelet $m$-gon with all vertices in cyclic order occurs only for one specific choice of $c$.

\begin{figure}[t]
\begin{center}
\includegraphics[width=.9\textwidth]{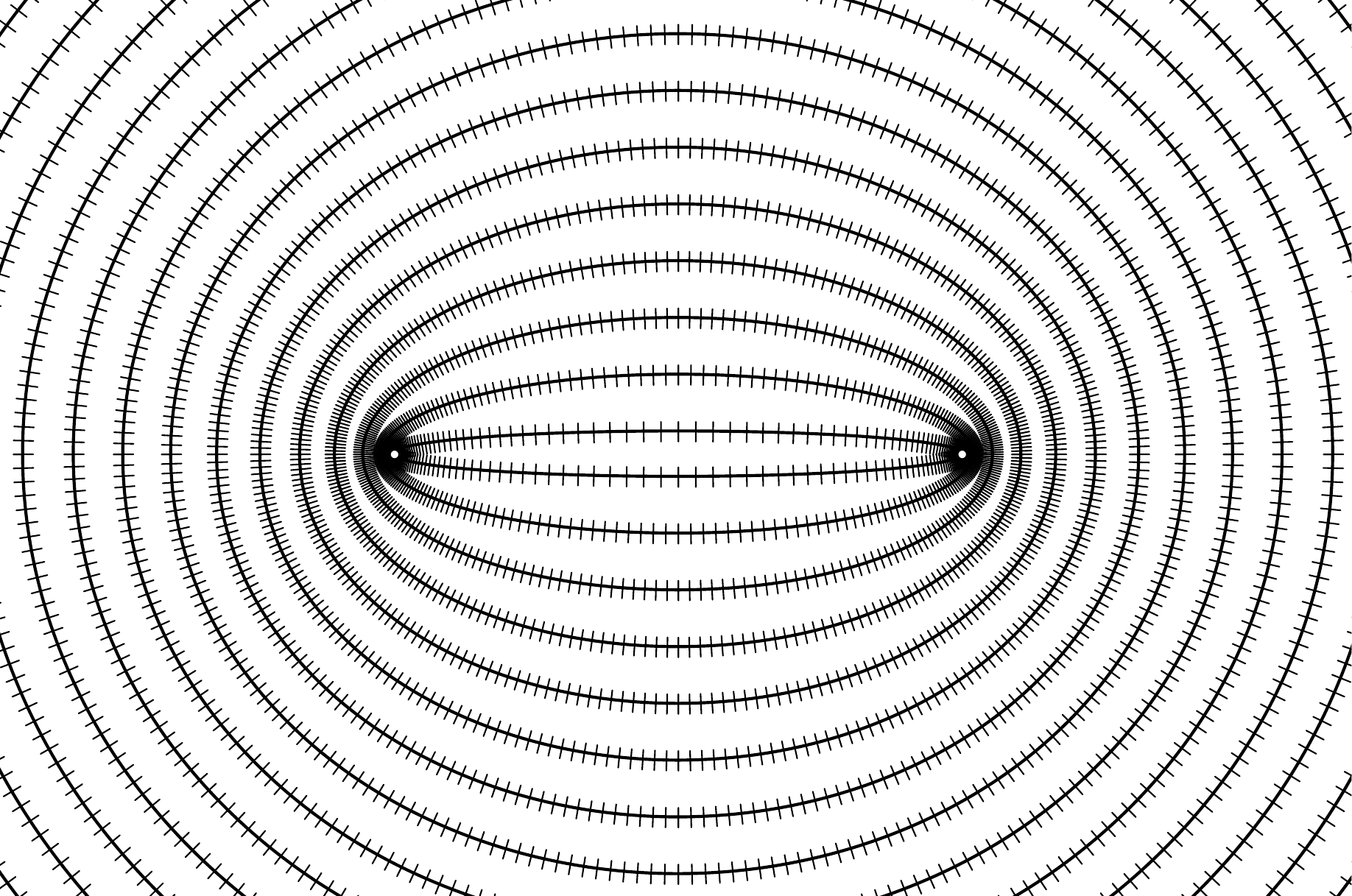}
\begin{picture}(0,0)
\end{picture}
\end{center}
\vskip-5mm
\captionof{figure}{The density of the Poncelet measurements along various confocal conics.}\label{fig:measure}
\end{figure}

In \cite{Iz15,LaTa07}
it is proved that there is an isomorphism $\phi$ 
between the shifts on $\mathcal{X}$ and the half open interval $[0,1)$ (with 
$\phi(c_1+c_2)=\phi(c_1)+\phi(c_2)\  \mathrm{mod} \ 1$.) 
We can consider this isomorphism as a normalised measurement of distances along the conic. 
Figure~\ref{fig:measure} shows a collection of confocal ellipses subdivided into 128 equal steps with respect to this measure. 
Observe that the curvature causes compression of this measure.

\noindent 
Taking everything together shifts can be represented in four different ways:
\begin{itemize}
\item by ordered pairs $(P,Q)$,\\[-5mm]
\item as the corresponding signed conic $\mathcal{C}$,\\[-5mm]
\item as a number $0\leq c< 1$,\\[-5mm]
\item by a Poncelet chain, modulo a starting point.
\end{itemize}

Poncelet $m$-gons with points cyclically ordered arise if a shift subdivides $[0,1)$ into $m$ equal parts. 
This happens for $\phi(c)=1/m$. If appropriate, we may in the case of Poncelet polygons renormalise our measure again by multiplying it with $m$ and calculate modulo $m$. 
Then the points of a Poncelet polygon appear at places $x+i$ for some integer $i$ and some real number $x$.

\medskip
Now we relate this point of view to our considerations about constructions and configurations from the previous sections. For that we have to relate shift measures of different confocal ellipses.
Consider the following incidence Lemma whose geometric situation is depicted in Figure  \ref{fig:doublinginci}.

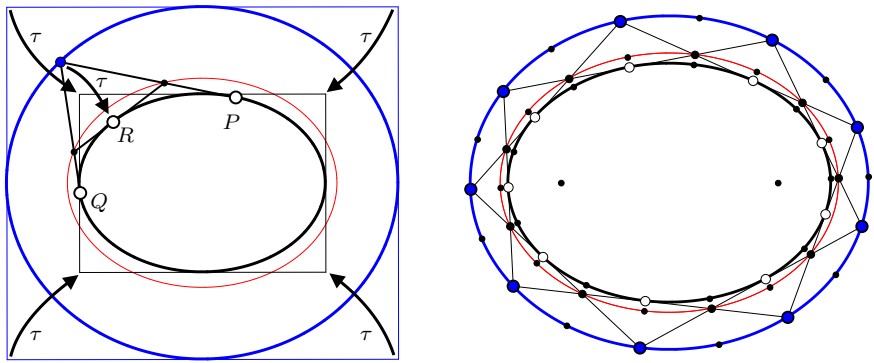
\begin{figure}[t]
\begin{center}
\includegraphics[width=.45\textwidth]{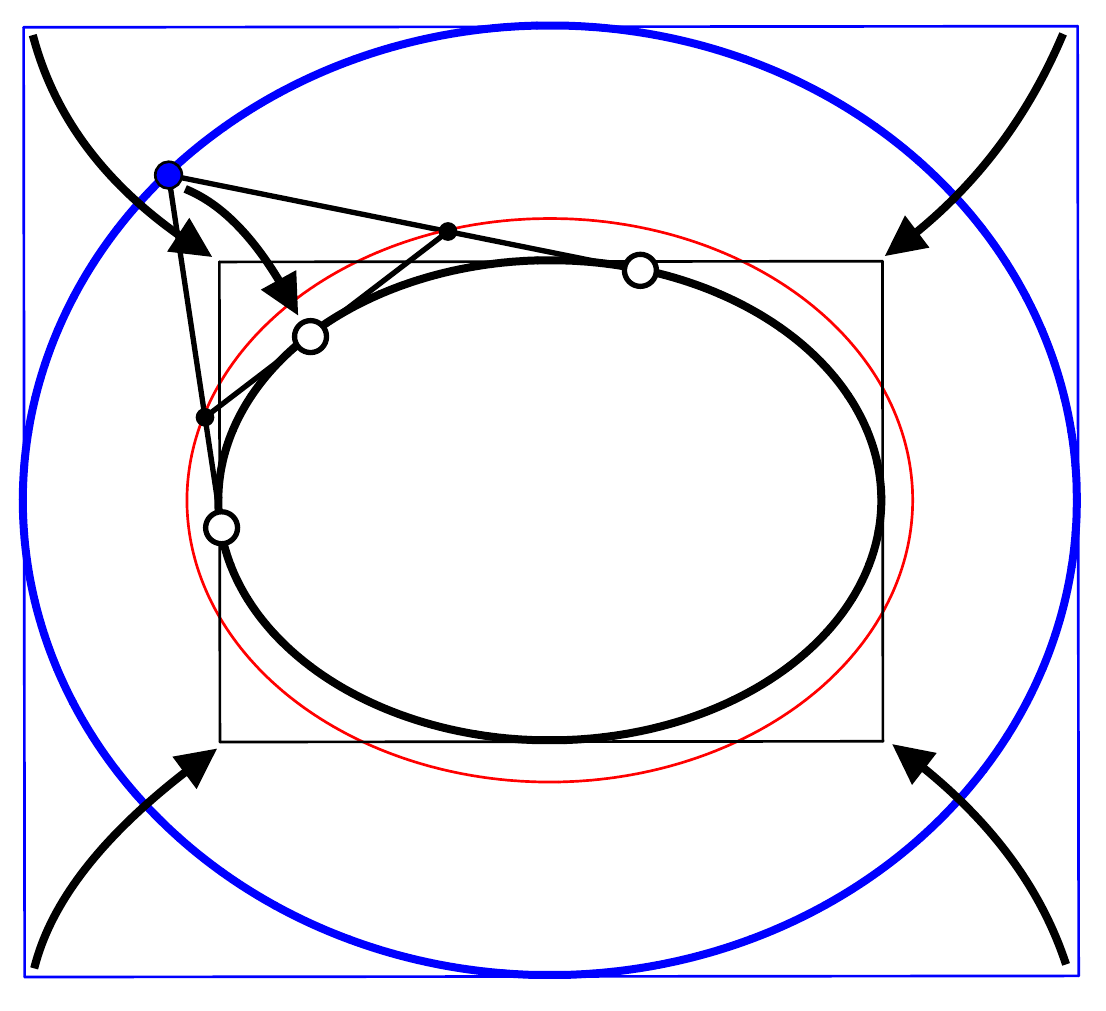}\hfill
\includegraphics[width=.5\textwidth]{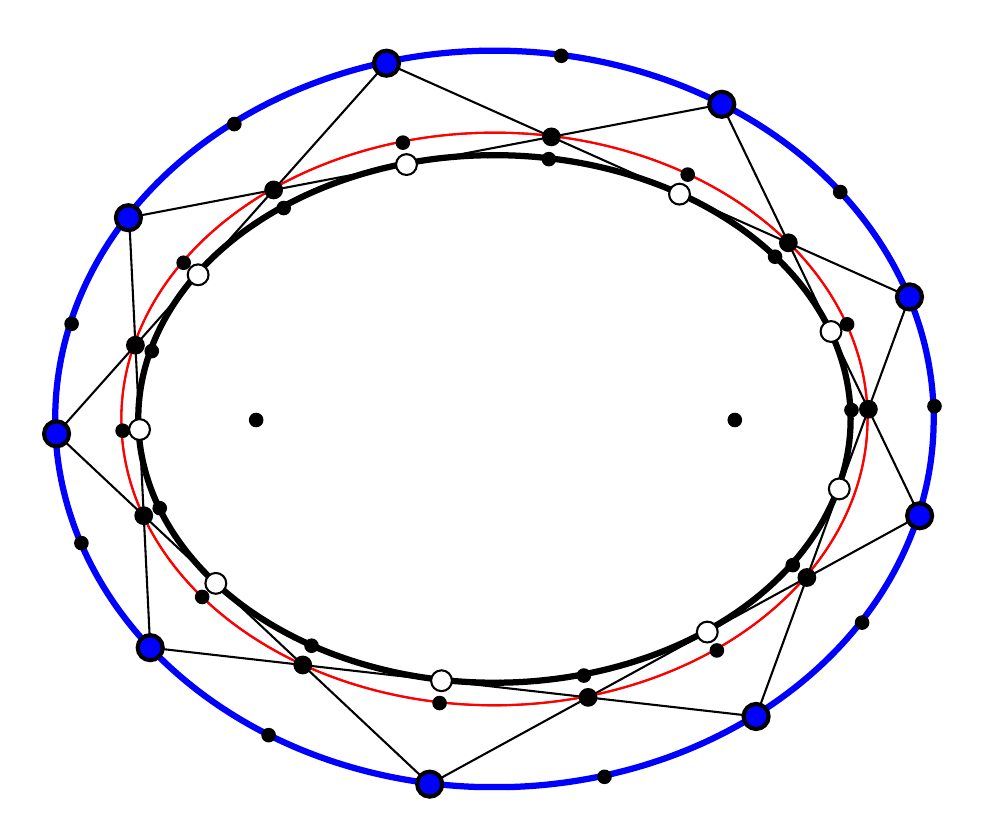}
\begin{picture}(0,0)
\put(-333,125){\footnotesize $\tau$}
\put(-208,125){\footnotesize $\tau$}
\put(-333,12){\footnotesize $\tau$}
\put(-208,12){\footnotesize $\tau$}
\put(-308,108){\footnotesize $\tau$}
\put(-260,92){\footnotesize $P$}
\put(-310,62){\footnotesize $Q$}
\put(-300,87){\footnotesize $R$}
\end{picture}
\end{center}
\vskip-5mm
\captionof{figure}{
Left: Halving a shift can be represented by a transformation that maps one conic to the other. The intersection of the tangents is mapped to the halving point. Right: repeated occurrence of this situation in the $\pts[a](L)$ operation.}\label{fig:doublinginci}
\end{figure}

\begin{lemma}\label{doublinglem}
Let $\mathcal{X}$ be a conic, let $c$ be a shift and let $\mathcal{C}$ and $\mathcal{B}$ be the confocal conics associated with the shifts $c$ and $c/2$, respectively. 
Assume that the conics are located symmetric to the $xy$-axes. Let  the shifts be represented by $[P,Q]=c$ and  $[P,R]=c/2$. Let $l_P$ and $l_Q$ be the tangents
to $\mathcal{X}$ at $P$ and $Q$. Let $\tau$ be a transformation that only scales in the $x$ and $y$ direction and maps $\mathcal{C}$  to  $\mathcal{X}$. 
Then $\tau$ maps the intersection of the tangents $l_P$ and $l_Q$ to $R$.
\end{lemma}

We will not prove this lemma here, since we only sketch the main lines of thought. A proof can be found implicitly in  \cite{Iz15,LaTa07}. 
The situation is relevant for our construction of $(n_4)$ configurations, and helps to relate the positions of the rings of lines and rings of points to each other. 
Consider Figure~\ref{fig:doublinginci} on the right. It shows two rings of points in our construction of the main theorem together with the corresponding lines that connect them. 
The two rings lie on two conics $\mathcal{B}$ and $\mathcal{C}$, and the lines are tangent to a conic $\mathcal{X}$.
We represent the lines by the respective touching points on  $\mathcal{X}$. 
By the Poncelet Grid Theorem, each ring of points forms a Poncelet $m$-gon. Without loss of generality, we may consider all three conics to be confocal. The points on $\mathcal{C}$ form a Poncelet $m$-gon 
that are generated by a shift $c$. Although the other rings of points lie on confocal conics, they are not generated by a uniform shift along these conics. However, by mapping them to the inner conic, 
we can relate them to a shift, as we will see next.

The situation repeatedly  contains the geometric situation of Lemma \ref{doublinglem}. By the projective transformation induced by Lemma \ref{doublinglem} we may represent 
the points of one of the outer rings by the mapped points on the inner ring on $\mathcal{X}$. All potential points on $\mathcal{X}$ are generated by the iterated shift $c/2$. 
Each ring of points is either mapped to the original points $x+i\cdot c$, or to the points shifted by $1/2$ related to the shifts $x+i\cdot c+{1\over 2}$. 

\medskip

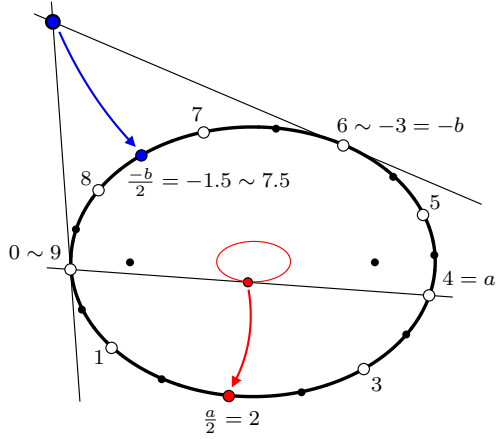
\begin{figure}[t]
\begin{center}
\includegraphics[width=.5\textwidth]{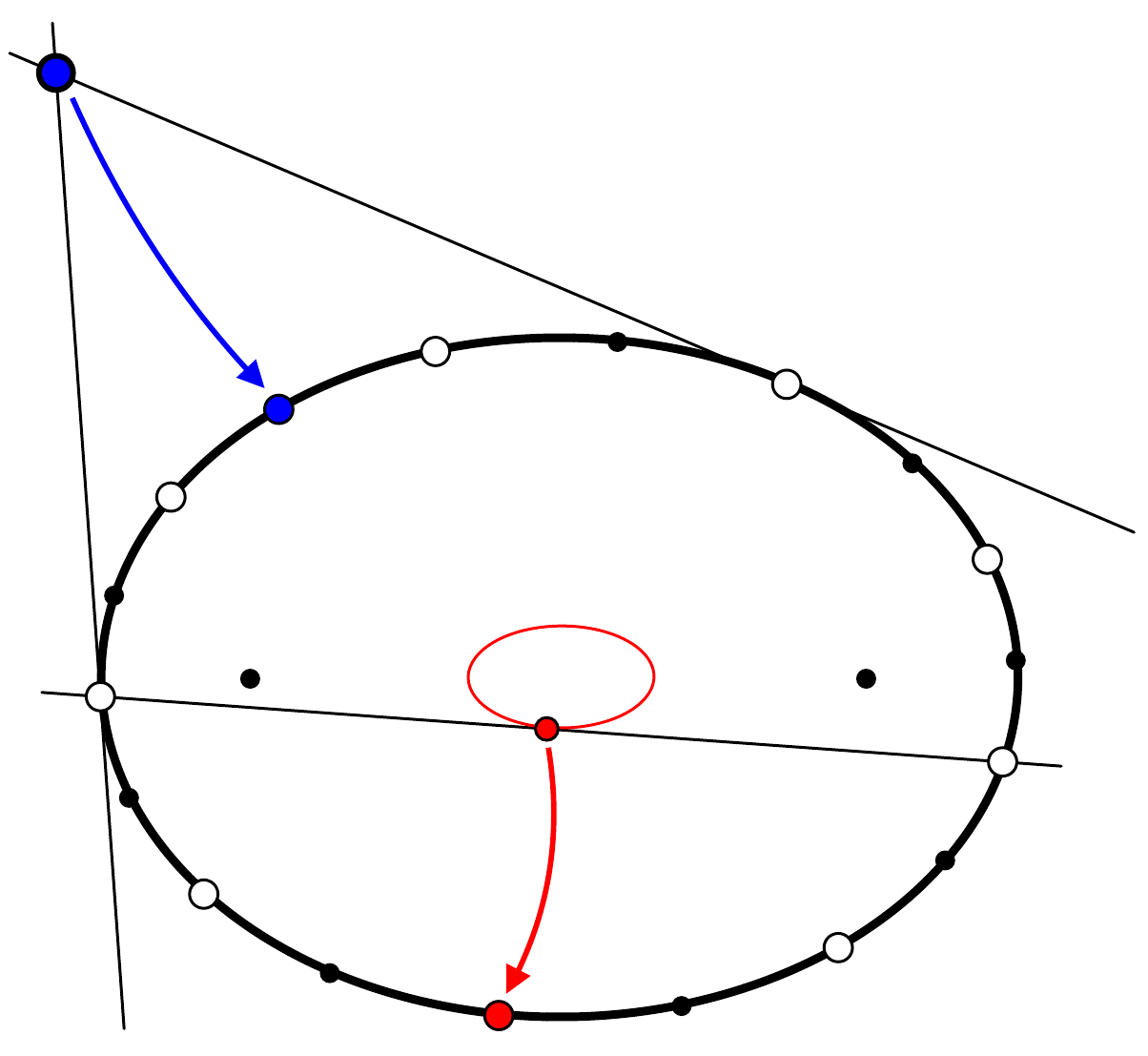}
\begin{picture}(0,0)
\put(-184,58){\footnotesize $0\sim 9$}
\put(-152,18){\footnotesize $1$}
\put(-111,-5){\footnotesize ${a\over2}=2$}
\put(-47,8){\footnotesize $3$}
\put(-20,48){\footnotesize $4=a$}
\put(-25,77){\footnotesize $5$}
\put(-60,106){\footnotesize $6\sim -3=-b$}
\put(-115,110){\footnotesize $7$}
\put(-157,86){\footnotesize $8$}
\put(-140,85){\footnotesize ${-b\over 2}=-1.5\sim 7.5$}
\end{picture}
\end{center}
\vskip-0mm
\captionof{figure}{ Two examples of representing the operations  $\lns[a]$ and $\pts[b]$
by index shifts. The operation $\lns[a]$ connects two points $i$ and $i+a$ by a line that is tangent to the red conic in the dual Poncelet grid. Its touching point is mapped to a point $a\over 2$ on the central conic when the inner red conic is linearly blown up to match the outer one. Similarly $\pts[b]$ results in the point of intersection of two tangents. It lies on a certain conic of the Poncelet grid (not shown in the picture). Shrinking this conic down to the black conic maps this intersection to $-b\over 2$.} \label{fig:indextrafo}
\end{figure}

Let us step back for a moment. The above considerations give us the possibility to represent all points and lines of all rings in an 
$m\#(a_1,b_1;a_2,b_2;\ldots;a_k,b_k)$ construction by points on {\it one} Poncelet $2m$-gon on a single conic. Let us assume that these points are labelled 
by $0.5, 1, 1.5, 2, 2.5, 3, 3.5\upto m$ that represent their shifts referred  to an initial point; we consider indices modulo $m$. 
Then, our operations $\lns[a]$ and $\pts[b]$  for $1 \leq a,b<m/2$ may be expressed by a cyclic shift by $a/2$, resp.\ ${(-b)/2}$ of these indices. 
The minus sign for the operation $\pts[b]$ occurs because we defined this operation by intersection of lines that are shifted clockwise (mathematically negative) by $b$ index steps. 
The situation is depicted in Figure~\ref{fig:indextrafo}.
 After working out all the details in this approach one can derive a conceptually very  simple proof of the fact that the points in
  \[
\pts[c]\cdot\lns[b]\cdot
\pts[a]\cdot\lns[c]\cdot
\pts[b]\cdot\lns[a](P)
\]
are shifted identically to those of $P$.

By our above considerations we represent each ring of points and lines
by the points of one Poncelet $2m$-gon generated by a shift $1\over 2m$.
Operations of the form $\pts[a]$ result in an index shift of $a\over 2$ on these points
and operations of the form $\lns[b]$ result in an index shift of $-b\over 2$ on these points.
Thus, showing that the above sequence of operations results in the identity simply translates---in the distorted metric of shifts---to the simple equation 
 \[
{a\over 2}
-{b\over 2}
+{c\over 2}
-{a\over 2}
+{b\over 2}
-{c\over 2}=0.
\]

\subsubsection{Pentagram map}

We do not want to close this section without mentioning the relations of our results to the topic of {\it pentagram maps}. 
The pentagram map was introduced by R.\ Schwartz about 30 years ago, and it has been thoroughly studied since then. 
See \cite{GSTV16,Sch92,OST10} for some initial literature. 
The pentagram map studies maps that arise when mapping the vertices of a polygon $P$ to the intersections of diagonals of the polygon (similar to our $\pts[b]\cdot \lns[a](P)$ operations).
Originally the map was only considered on the short diagonals of a polygon and it was later extended to more general cases.

Let us denote the short pentagram map by $T_2$.
It is constructed by connecting the vertices  $p_i$ and $p_{i+2}$ of an $m$-gon and then intersecting consecutive diagonals. 
Similarly one defines the deeper-diagonal pentagram maps $T_k,\ k<m/2$  by connecting the vertices $p_i$ and $p_{i+k}$.
In our notation the map $T_k$ is represented by the operator  $\pts[1]\cdot\lns[k]$.

The pentagram map commutes with projective transformations, and it descends to the moduli space of projective equivalence classes of polygons. 
Thus the polygons that arise are considered modulo projective transformations. The first (and slightly surprising) fact in this theory is that the map $T_2$ for the pentagon is the identity. 
In other words, the intersections of the diagonals of a pentagon are projectively equivalent to the original pentagon.
This was already known to Clebsch \cite{Cle1871}.
In general, the resulting map is completely integrable in the sense of Liouville and, by now, it is one of the best known and most studied example of a discrete integrable system. 
In particular, the pentagram map is intimately related with the recently emerged theory of cluster algebras.

 \begin{figure}[ht]
\centering
\includegraphics[width=.76\textwidth]{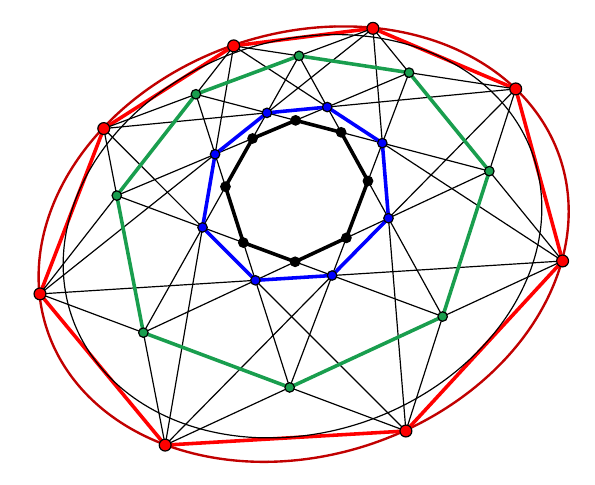}
\caption{$T_2 \circ T_3=T_3 \circ T_2$: both compositions send the red outer Poncelet octagon to the black inner one.}	
\label{fig:commute}
\end{figure}

 With this notation, our Theorem B implies that, if applied to Poncelet polygons, these pentagram maps commute: $T_s \circ T_t = T_t \circ T_s$, see Figure \ref{fig:commute}. 

Let us mention that the pentagram map nicely interacts with Poncelet polygons. 
For example, recall that the pentagram map is completely integrable: a  number of functions on the space of polygons are its integrals. It is proved in \cite{Sch15} that these integrals remain constant on the 1-parameter family of polygons, inscribed into one and circumscribed about another conic.  As another example, it is proved in \cite{Izo22} that  if a convex polygon is projectively equivalent to its pentagram image, then it is Poncelet.

\subsubsection{The easier case of odd-gons}

For the proof of our core Theorem A several essential ingredients had to be taken together: Poncelet grids, a continuous incidence theorem similar to the Chasles--Graves Theorem, an elaborate geometric construction, and careful bookkeeping of the labels and indices involved.
Amazingly, if the number of points on the Poncelet polygons is odd, then a significantly simpler proof can be applied
(actually, this is how we started). The deep reason for that is that for a Poncelet $m$-polygon $P$ with odd $m$ the result of the operation
$(\pts[2]\cdot\lns[1])(P)$ leads to a polygon that is projectively equivalent to $P$.
This is not the case for even $m$. 
 This projective equivalence provides a kind of shortcut in the argument. 
For that case we need 
the following strengthening of Theorem \ref{ponceletGrid1}, a proof of which can be found in~\cite{LaTa07,Sch07} (using different notation). 

\begin{theorem} \label{thm:grid2}
Let $m$ be odd, and let $P$ be a Poncelet $m$-gon with lines $L = \lns[1](P)$. Then, for every $a$, there exists a projective transformation $\tau_a$ that takes the points (taken as a set) of the polygon $P$ to points of the polygon 
$\pts[a](L)$, and these projective transformations commute.
\end{theorem} 

For the case in which the Poncelet polygon is supported by an ellipse  $\mathcal{A}$ that  is symmetric with respect to the coordinate axses we can express these projective transformations in a very simple way.
The points  $Q=(\pts[a]\cdot\lns[1])(P)$ are the points of another Poncelet polygon supported by an ellipse $\mathcal{B}$ confocal to  $\mathcal{A}$. 
There are four natural projective  transformations that map the ellipse  $\mathcal{B}$  to  $\mathcal{A}$, by scaling the $x$ and $y$ coordinates. The four possibilities come from the different signs of the scaling in $x$ and $y$ direction.
The projective transformation  $\tau_a$, whose existence is stated in the above theorem, that in addition maps the points set $P$ to the point set $Q$, 
can be expressed as a diagonal $3\times 3$ matrix. For odd $a$
this matrix has signature $(+,+,+)$, and for even $a$ the signature is $(-,-,+)$.

If we want to overcome the fact that in this mapping the assignment is only setwise then we have to take a cyclic shift of indices into account. Let $\pi\in S_m$ be the permutation that cyclically shifts the points in $(1,\ldots , m)$ by one step.
Then the we get the following more precise formulation of Theorem \ref{thm:grid2}.

\begin{theorem} \label{thm:grid3}
Let $m$ be odd, and let $P$ be a Poncelet $m$-gon with lines $L = \lns[1](P)$. Then, for every $a$, there exists a projective transformation $\tau_a$ and a cyclic shift $\pi^{k}$ such that $(\pts[a]\cdot\lns[1])(P)_i=(\tau_a(P))_{\pi^{k}(i)}$.
\end{theorem} 

From Theorem~\ref{thm:grid2} this statement can be proved directly by careful bookkeeping of the indices.
Let us denote the combined action of transformation and index shift as $F_a$, where $(\pts[a]\cdot\lns[1])(P)=F_a(P)$.
Notice that for shifts $a$ and $b$, the operations $F_a$ and $F_b$ still commute.

Applying Theorem \ref{thm:grid3}, to two different shifts $a$ and $b$ and using the fact that $F_a$ and $F_b$  commute we get:
\[
(\pts[a]\cdot\lns[1])\cdot (\pts[b]\cdot\lns[1])(P)=F_a(P)\cdot F_b(P)=
F_b(P)\cdot F_a(P)=(\pts[b]\cdot\lns[1])\cdot (\pts[a]\cdot\lns[1])(P).
\]
and hence
\[
(\pts[a]\cdot\lns[1]\cdot \pts[b])(L)=(\pts[b]\cdot\lns[1]\cdot \pts[a])(L).
\]

\noindent
Setting $P'= \pts[a](L)$ (which implies $L= \lns[a](P')$)  we get

\[
(\pts[a]\cdot\lns[1]\cdot \pts[b]\cdot\lns[a])(P')=(\pts[b]\cdot\lns[1]\cdot \pts[a]\cdot\lns[a])(P').
\]

\noindent
Multiplying both sides with $\lns[a]$ and cancelling leaves us with:
\[
(\lns[1]\cdot \pts[b]\cdot\lns[a])(P')=(\lns[a]\cdot\pts[b]\cdot\lns[1]\cdot)(P').
\]

\noindent
Since we can bijectively go back and forth between Poncelet polygons  by our $\wedge$ and $\vee$
operators, we can assume that $P'$ is an arbitrary Poncelet odd-gon.
Also, a corresponding dual statement holds for lines $L$ of a Poncelet odd-gon:
\[
(\pts[1]\cdot \lns[b]\cdot\pts[a])(L)=(\pts[a]\cdot\lns[b]\cdot\pts[1]\cdot)(L).
\]
Taking these statements together we can derive Theorem A for odd $m$, as follows:

\begin{align*}
\lns[a] \cdot \pts[b] \cdot \lns[c]\ (P) &= \lns[a] \cdot \pts[b] \cdot \lns[c] \cdot \overbrace{\pts[1] \cdot \lns[1]}^{\text{id}} \ (P)\\[1mm]
&=\lns[a] \cdot \underbrace{\pts[b] \cdot \lns[c]\cdot \pts[1]}_{\pts[1] \cdot \lns[c]\cdot \pts[b]} \cdot \lns[1] (P)\\
&= \underbrace{\lns[a] \cdot\pts[1] \cdot \lns[c]}_{\lns[c] \cdot\pts[1] \cdot \lns[a]}\cdot \pts[b]\cdot \lns[1] (P)\\
&=\lns[c] \cdot\underbrace{\pts[1] \cdot \lns[a]\cdot \pts[b]}_{\pts[b] \cdot \lns[a]\cdot \pts[1]}\cdot  \lns[1] (P)\\
&=\lns[c] \cdot\pts[b] \cdot \lns[a]\cdot \underbrace{\pts[1] \cdot \lns[1]}_{\text{id}}\  (P)\\
&=\lns[c] \cdot \pts[b] \cdot \lns[a]\ (P)
\end{align*}

\noindent
Which is exactly the statement of Theorem A.

\medskip

One might wonder why this approach cannot be applied in the case of Poncelet even-gons.
The main obstacle is that the point sets $\pts[a]\cdot\lns[1](P)$ are not necessarily projectively equivalent to the points of $P$,
depending on the parity of $a$. Hence, in that case, it is not easy to find the equivalent of the maps $F_a$ and $F_b$.
In a sense, our Lemma \ref{lem:commute} presents an alternative way to obtain a commutative behaviour from which we could derive Theorem A.

\section{Examples}\label{sect:examples}

Our exposition so far already contained a multitude of interesting $(n_4)$ configurations
for various choices of parameters:
Figure~\ref{fig:GrRi} shows a $7\#(3,1;2,3;1,2)$ configuration,
Figure~\ref{fig:celest1} presents a $8\#(3,1;2,3;1,2)$ configuration,
Figure~\ref{fig:star2} a $10\#(2,3;4,2;3,4)$ configuration and Figure~\ref{fig:13incircle} a $13\#(5,2;4,5;2,4)$ configuration.
In this section, we present a few of the more complex and slightly exotic configurations that are covered by our constructions.

\subsection{More than three rings}
Our Theorem C allows for the creation of $(n_4)$ configurations with arbitrarily many rings from a Poncelet polygon. 
Keeping in mind the condition that in the description 
$m\#(a_1,b_1;a_2,b_2;a_3,b_3;\ldots;a_k,b_k)$ the multisets 
$[a_1\upto a_k]$ and $[b_1\upto b_k]$ have to be identical,
and that no two adjacent letters can be the same, we get (up to isomorphisms) a unique pattern for such trivial celestial $(n_4)$ configurations with $k=4$. 
They follow one of the two patterns
\[m\#(a,b;c,d;b,a;d,c)\quad\text{or}\quad m\#(a,b;c,a;d,c;b,d).\]
For that to happen we need $m\geq 9$ and we may get a 4-ring configuration like $9\#(4,1;2,3;1,4;3,2)$.
Here the letters $1,2,3,4$ can be arbitrarily interchanged with each other. We show a configuration
$10\#(4,1;2,3;1,4;3,2)$ in Figure~\ref{fig:4rings}.
\begin{figure}[H]
\begin{center}
\includegraphics[width=.9\textwidth]{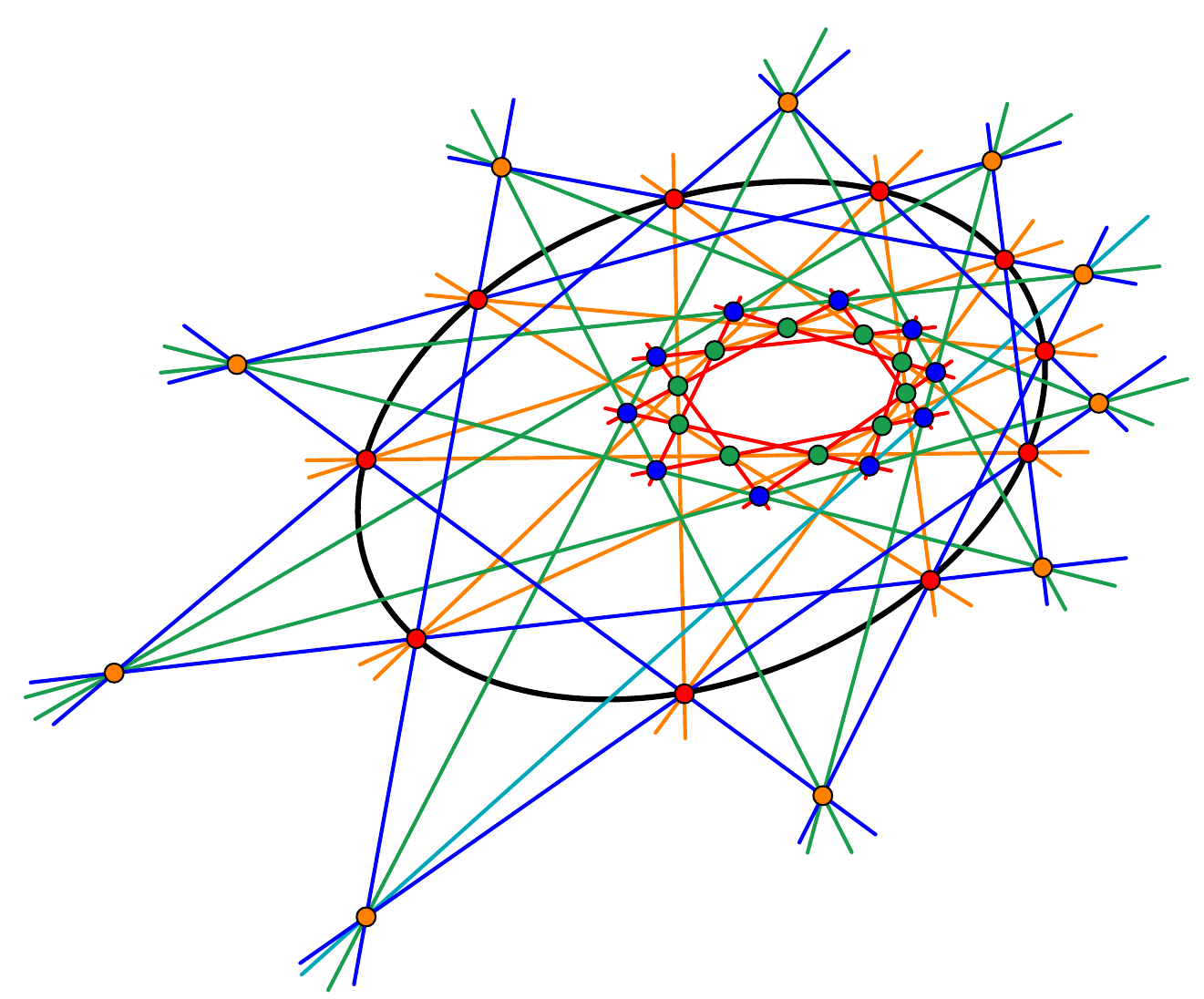}
\begin{picture}(0,0)
\end{picture}
\end{center}
\kern-3mm
\captionof{figure}{A $10\#(4,1;2,3;1,4;3,2)$ configuration. The initial ring is the ring of red points on the black conic.}\label{fig:4rings}
\end{figure}
This configuration has the additional interesting property that the green and the orange lines also meet in sets of four, as do the blue and red lines, if extended and thus may be extended to a $(4,6)$-configuration in the sense of \cite{BerBu10}.

\subsection{Breaking up additional incidences}

Our construction has not only the potential to create incidences. It also has the potential to
{\it destroy} incidences in a meaningful way.
Consider the configuration in Figure~\ref{fig:destroy} on the left.
It is a $12\#(5,4;1,4)$ configuration that was presented, for instance, in \cite[Figure 3.6.2]{Gr09}. 
It is a $(24_4)$ configuration that consists only of 2 rings of 12 points each.
This configuration exists because there is a non-trivial solution to the cosine condition, due to
special trigonometric relations between the angles of the type $i\cdot{\pi\over 6}$.
This can be achieved if the initial ring consists of the points of a regular 12-gon.
These special incidences immediately break if we replace the 12-gon by an arbitrary Poncelet 12-gon.

\begin{figure}[H]
\begin{center}
\includegraphics[width=.45\textwidth]{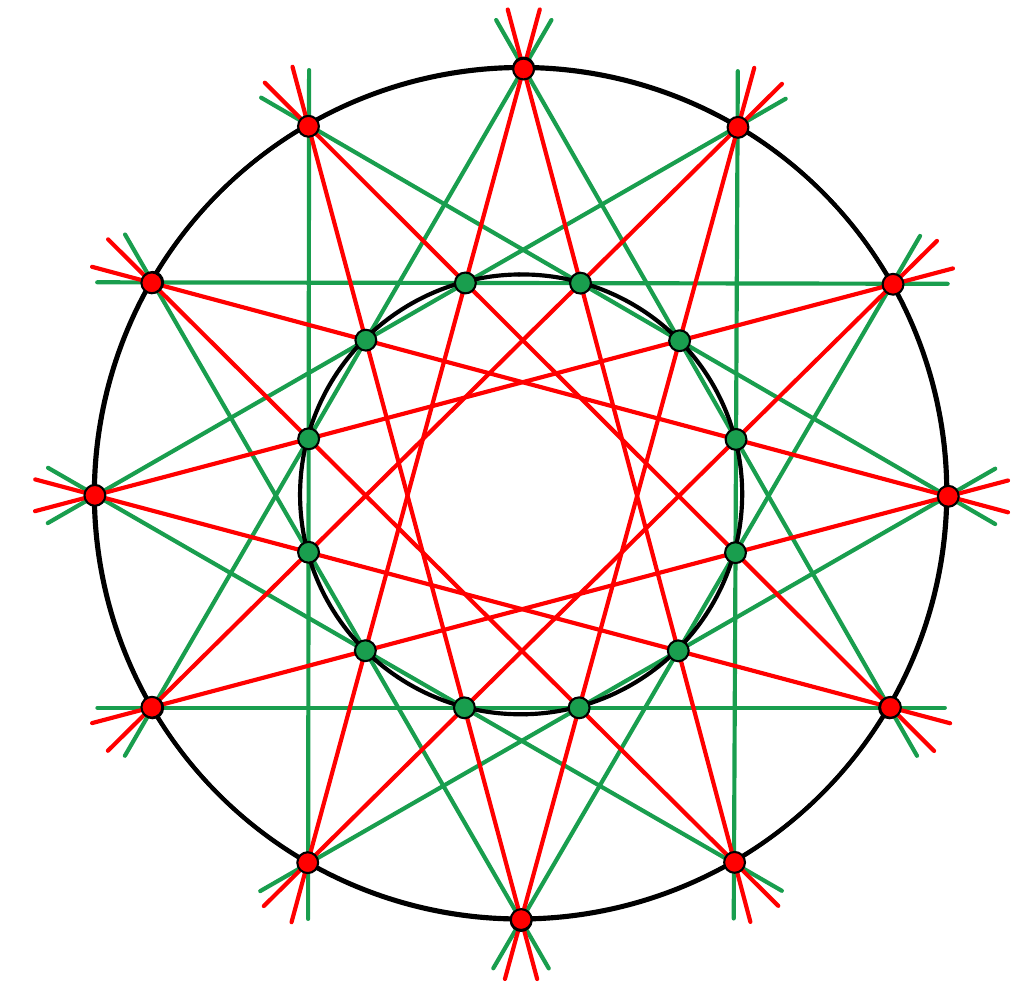}\hfill
\includegraphics[width=.49\textwidth]{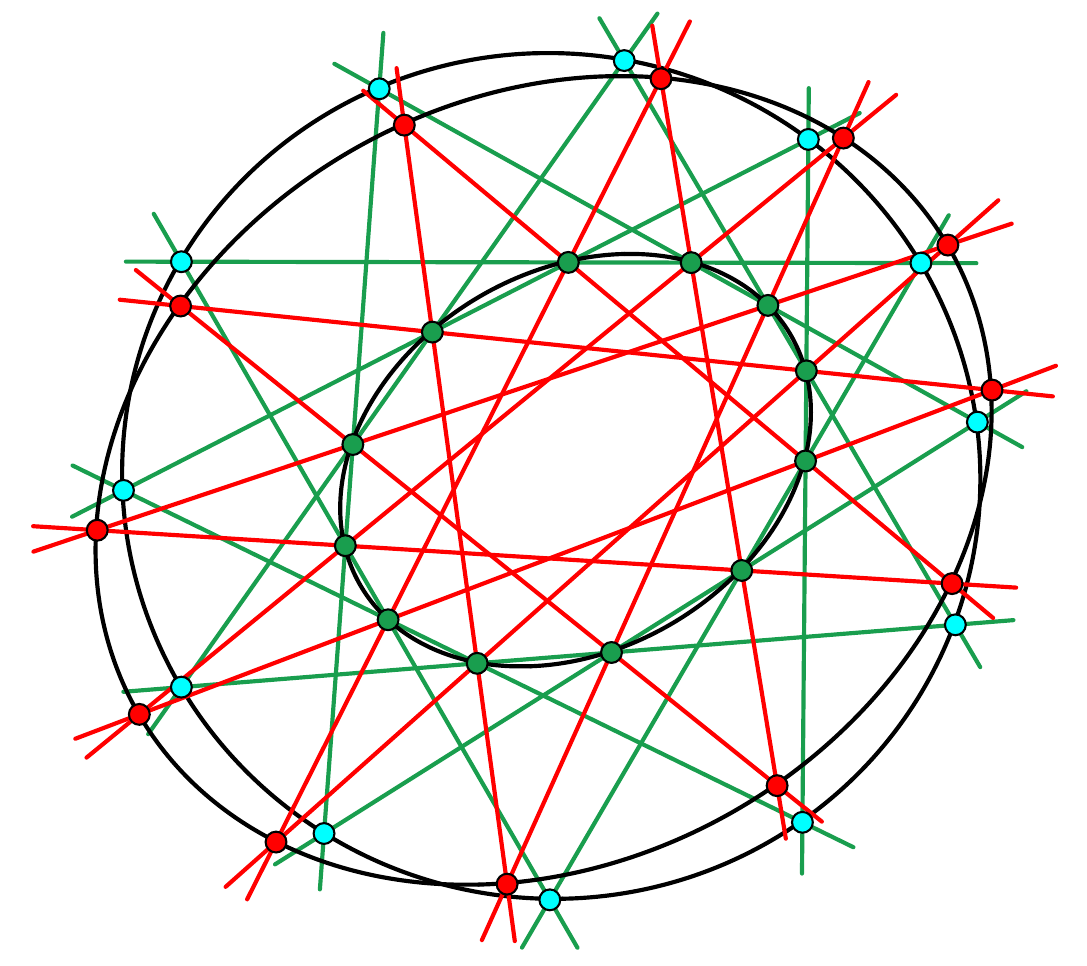}
\begin{picture}(0,0)
\end{picture}
\end{center}
\kern-3mm
\captionof{figure}{A $12\#(5,4;1,4)$ construction in symmetric position (left) and broken  by starting with a Poncelet $12$-gon (right)}\label{fig:destroy}
\end{figure}

One can use this effect to create $(n_4)$ configurations that otherwise (in the rotationally symmetric case) would only be realisable with additional unwanted incidences.
Figure~\ref{fig:destroy2} shows a trivial
$12\#(5,4;1,4;3,5;4,3;4,1)$ configuration 
that starts with the initial sequence of $12\#(5,4;1,4)$.
It is only realisable properly since we did not start with a fully symmetric dodecagon. If we had done so, we would have had additional incidences. 
Observe that  in the Poncelet-perturbed situation, the red and cyan points always lie nearby each other, but are not identical.

\begin{figure}[H]
\begin{center}
\includegraphics[width=.75\textwidth]{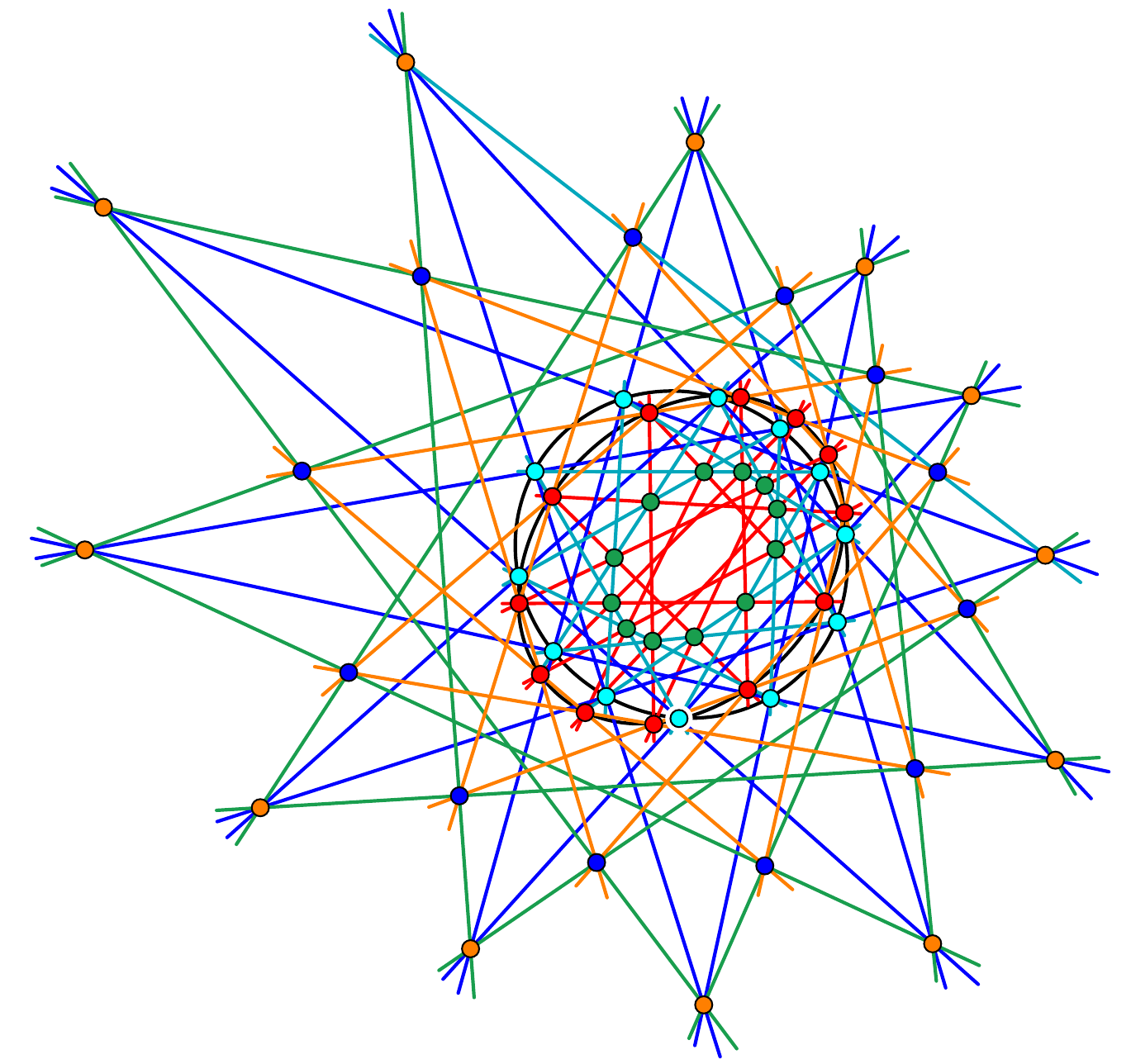}
\begin{picture}(0,0)
\end{picture}
\end{center}
\kern-3mm
\captionof{figure}{A construction of a configuraton $12\#(5,4;1,4;3,5;4,3;4,1)$ whose existence depends on starting with a non-regular Poncelet polygon.} \label{fig:destroy2}
\end{figure}

\subsection{Movable $(n_6)$-configurations}

Let us conclude our journey by an example that goes beyond $(n_4)$ configurations.
In \cite{Ber11} one of us created a method of nesting trivial celestial $(n_4)$ configurations to form higher structures with a huge degree of incidences. 
The idea is to use the same ring of points simultaneously in several $m\#(a,b;c,a;b,c)$ configurations for varying parameters. 
By this one can combine several rings of points in a way such that  each point is incident to multiple lines. 
The tricky part is to get a combinatorially consistent combination of such arrangements.  
For details on the process we refer to \cite{Ber11}. Here we only want to explain how this relates to the methods in the present article.

There is a specific construction that works on 10 rings  of points, with each ring taking part in ten $m\#(a,b;c,a;b,c)$  configurations. 
For that, from a selection of 5 different shift parameters $a,b,c,d,e$, all possible three element subsets are taken, and with those shifts 10 $(n_4)$ configurations are created simultaneously. 
This leads to an $(n_6)$ configuration with 10 rings of points and 10 rings of lines. 
The Poncelet movability is inherited from the partial configurations to the big configurations.

Since one needs at least 5 different shift parameters, such a construction needs at least 11 points in each ring. 
Thus a $(110_6)$ is the smallest possible instance of such a configuration. 
Figure~\ref{fig:120_6} shows a Poncelet distorted example of a $(120_6)$ configuration based on 10 rings with 12 points each.
One of the underlying configurations, a $12\#(3,2;4,3;2,4)$, is emphasised in the picture. 
We use a beginning Poncelet 12-gon, here, because it is easier to construct geometrically exactly by doubling; we do not have a purely geometric construction to produce a starting Poncelet 11-gon. Explicit constructions are given in \cite{BGRGT24b}.

\begin{figure}[H]
\begin{center}
\includegraphics[width=.85\textwidth]{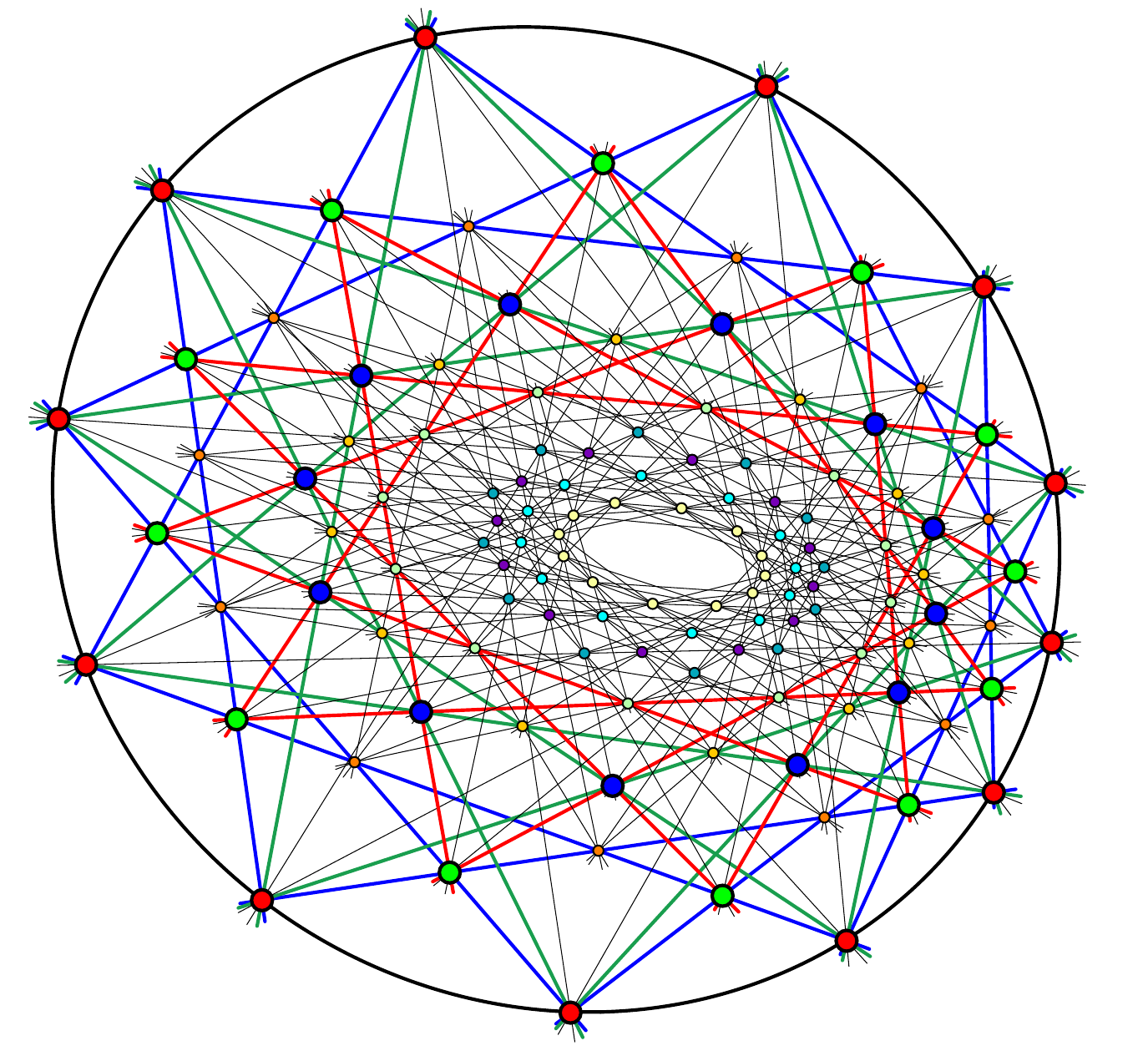}
\begin{picture}(0,0)
\end{picture}
\end{center}
\kern-3mm
\captionof{figure}{A movable $(120_6)$ configuration consisting of 10 nested configurations with 3 rings.
A subconfiguration of type $12\#(3,2;4,3;2,4)$ is highlighted by the bigger points and lines in red/green/blue.} \label{fig:120_6}
\end{figure}

Although the picture is kind of overcrowded, it has stunning combinatorial properties.
First of all, there are 10 types of points and 10 types of lines. They are organised in a way such that on each line there are six points of three types and dually, 
through each point there are six lines of three different types. Thus hidden in the background, a $(10_3)$ configuration is responsible for the combinatorial structure. 
In this case it is a Desargues configuration. In addition, in the symmetric case, the Desargues graph is the reduced Levi graph of the configuration, see \cite[Figure 6]{BerBer14}.)
Figure~\ref{fig:120_6_G}
highlights this structure by emphasising exactly one element from each orbit (of lines and of points).
The highlighted region is indeed a Desargues configurations that is rotated by Poncelet shifts to form the entire configuration.

\begin{figure}[H]
\begin{center}
\includegraphics[width=.85\textwidth]{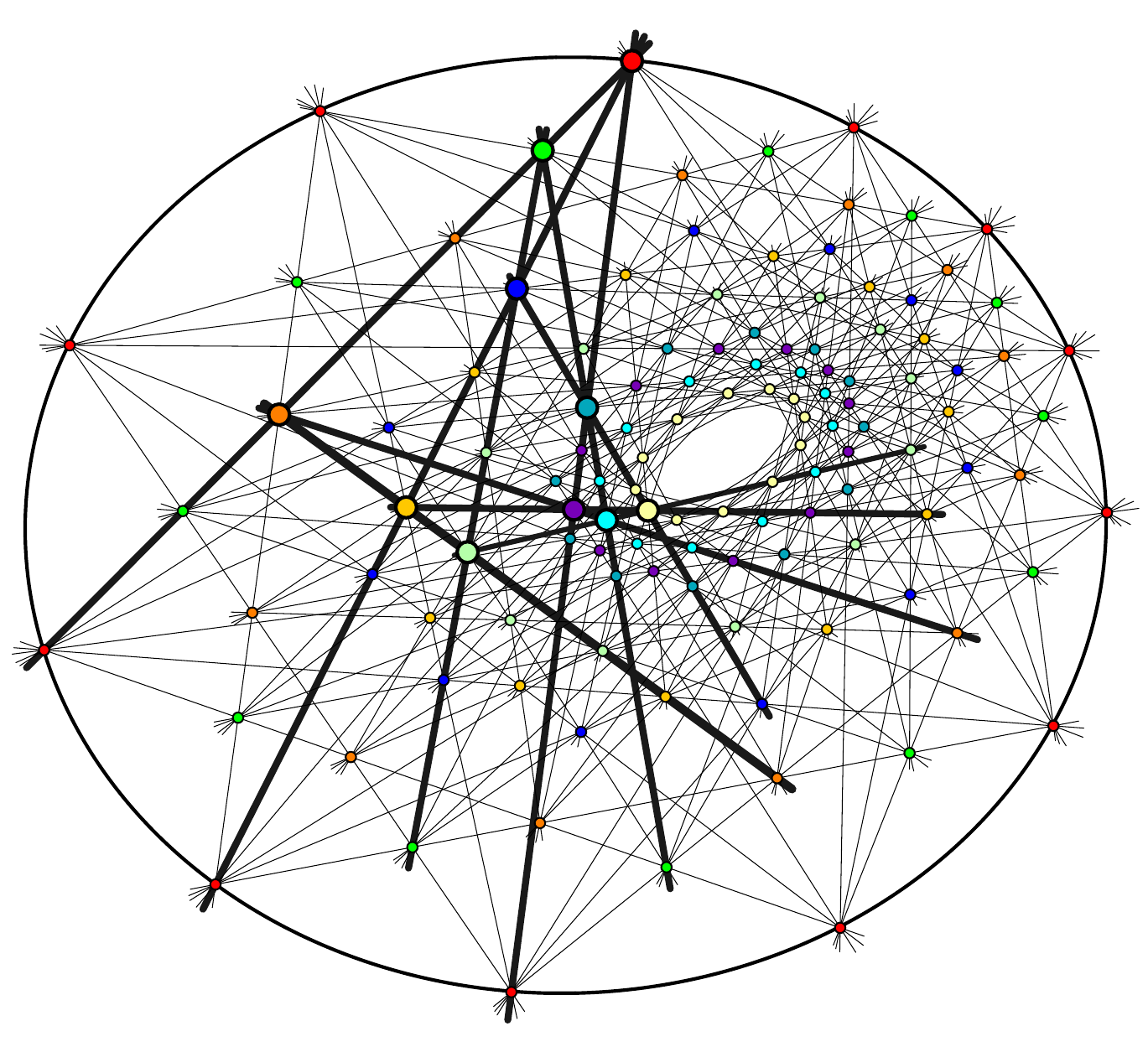}
\begin{picture}(0,0)
\end{picture}
\end{center}
\kern-3mm
\captionof{figure}{A $(120_6)$ configuration is created by Poncelet-rotating a Desargues configuration.} \label{fig:120_6_G}
\end{figure}

\subsection{When Grünbaum and Poncelet meet Pascal}

Our last example carries us over from the realm of point-line configurations to that of point-conic configurations.  Luis Montejano  observed that certain 7-tuples 
of points in the Grünbaum--Rigby $(21_4)$ configuration can be inscribed in conics. By applying the converse of Pascal's Theorem, one of us verified this property~\cite{Gev19},
and derived from the Grünbaum--Rigby configuration two different (i.e.\ non-isomorphic) $(21_7)$ configurations of points and conics. Precisely one of them has the property 
that it inherits movability from the underlying Grünbaum--Rigby configuration (see Figure~\ref{fig:GGconics} below). It turns out that these 21 conics admit 14 additional triple points that all lie on a common conic (black). This conic has an additional interesting property: polarising the 21 points at this conic creates the 21 lines of the original Grünbaum--Rigby  $(21_4)$-configuration.

We know other celestial point-line configurations which 
admit circumscribed conics such that point-conic configurations can be derived from them; searching for movable examples among these is a subject of future work. 

\begin{figure}[H]
\begin{center}
\includegraphics[width=.85\textwidth]{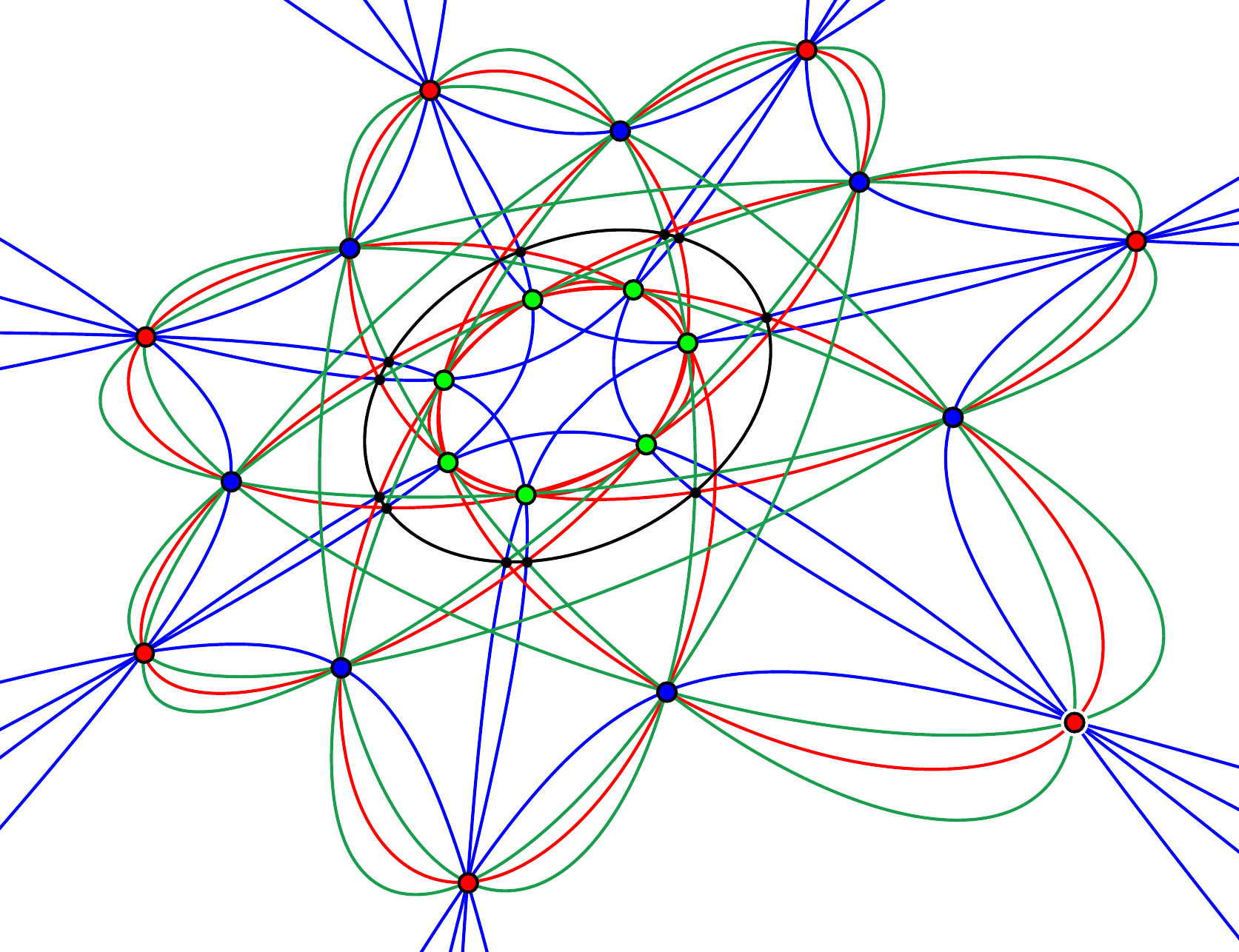}
\begin{picture}(0,0)
\end{picture}
\end{center}
\kern-3mm
\captionof{figure}{A movable $(21_7)$ configuration of points and conics.} \label{fig:GGconics}
\end{figure}

\paragraph{\bf Acknowledgements.}

 We are grateful to A. Akopyan, Tim Reinhardt and Lena Polke for a useful discussion.
ST was supported by NSF grants DMS-2005444 and DMS-2404535, and by a Mercator fellowship within the CRC/TRR 191. 
GG
was supported by the
Hungarian National Research,
Development and Innovation Office,
OTKA Grant No. SNN 132625.

\bibliographystyle{plain} 
{\small
\bibliography{refsd}

\begin{thebibliography}{10}

\bibitem{AkBo18}
Arseniy~V. Akopyan and Alexander~I. Bobenko.
\newblock Incircular nets and confocal conics.
\newblock {\em Trans. Amer. Math. Soc.}, 370(2825-2854), 2018.

\bibitem{AlkZas07}
Arseniy~V. Akopyan and Aleksej~A. Zaslavsky.
\newblock {\em Geometry of Conics}, volume~26 of {\em Mathematical World}.
\newblock American Mathematical Society, Providence, 2007.

\bibitem{BerBer14}
Angela Berardinelli and Leah~Wrenn Berman.
\newblock Systematic celestial 4-configurations.
\newblock {\em Ars Math. Contemp.}, 7:361--377, 2014.

\bibitem{Ber87}
Marcel Berger.
\newblock {\em Geometry. I. II.}
\newblock Springer-Verlag, Berlin, 1987.

\bibitem{Ber01}
Leah~Wrenn Berman.
\newblock A characterization of astral ($n_4$) configurations.
\newblock {\em Discrete Comput. Geom.}, 26(4):603--612, 2001.

\bibitem{Ber06}
Leah~Wrenn Berman.
\newblock Movable ($n_4$) configurations.
\newblock {\em Electron.J.Combin.}, $\#$R104, 2006.

\bibitem{Ber08}
Leah~Wrenn Berman.
\newblock A new class of movable ($n_4$) configurations.
\newblock {\em Ars Math. Contemp.}, 1:44--50, 2008.

\bibitem{Ber11}
Leah~Wrenn Berman.
\newblock Constructing highly incident configurations.
\newblock {\em Discrete Comput. Geom.}, 46:447--470, 2011.

\bibitem{BGRGT24b}
Leah~Wrenn Berman, G{\'a}bor G{\'e}vay, Serge Tabachnikov, and J{\"u}rgen
  Richter-Gebert.
\newblock When {G}r{\"u}nbaum meets {P}oncelet -- geometric constructions.
\newblock in preparation, 2024.

\bibitem{BoGePi15}
Marko Boben, G{\'a}bor G{\'e}vay, and Toma{\v z} Pisanski.
\newblock Danzer's configuration revisited.
\newblock {\em Adv. Geom.}, 15(4):393--408, 2015.

\bibitem{BaFa91}
Alexander~I. Bobenko and Alexander~Y. Fairley.
\newblock Nets of lines with the combinatorics of the square grid and with
  touching inscribed conics.
\newblock {\em Discrete \& Computational Geometry}, 66(4):1--19, 2021.

\bibitem{BerBu10}
Nadine~Alise Burtt and Leah~Wrenn Berman.
\newblock A new construction for symmetric (4, 6)-configuration.
\newblock {\em Ars Math. Contemp.}, 3:165--175, 2010.

\bibitem{Cha1865}
Michel Chasles.
\newblock {\em Trait{\'e} des sections coniques: faisant suite au trait{\'e} de
  g{\'e}om{\'e}trie sup{\'e}rieure}.
\newblock Gauthier-Villars, Paris, 1865.

\bibitem{ChaGra1841}
Michel Chasles and Charles Graves.
\newblock {\em Two Geometrical Memoirs On the General Properties of Cones of
  the Second Degree, and On the Spherical Conics}.
\newblock Dublin: For {G}rant and {B}olton, 1841.

\bibitem{Cle1871}
Alfred Clebsch.
\newblock Ueber das ebene {F}\"unfeck.
\newblock {\em Mathematische Annalen.}, 4(3):476--489, 1871.

\bibitem{Cox92}
Harold Scott~Macdonald Coxeter.
\newblock {\em The Real Projective Plane}.
\newblock Springer, New York, 1992.

\bibitem{Cox94}
Harold Scott~Macdonald Coxeter.
\newblock {\em Projective Geometry}.
\newblock Springer, New York, Berlin, 1994.

\bibitem{DGSV18}
Ulrich Daepp, Pamela Gorkin, Andrew Shaffer, and Karl Voss.
\newblock {\em Finding Ellipses}, volume~34 of {\em The Carus Mathematical
  Monographs}.
\newblock AMS/MAA, 2018.

\bibitem{Dar17}
Gaspard Darboux.
\newblock {\em Principes de Geometrie Analytique}.
\newblock Paris, Gauther-Villars, 1917.

\bibitem{Cen16a}
Andrea Del~Centina.
\newblock Poncelet's porism: a long story of renewed discoveries, i.
\newblock {\em Arch. Hist. Exact Sci.}, 70:1--122, 2016.

\bibitem{Cen16b}
Andrea Del~Centina.
\newblock Poncelet's porism: a long story of renewed discoveries, ii.
\newblock {\em Arch. Hist. Exact Sci.}, 70:123--173, 2016.

\bibitem{DraRa11}
Vladimir Dragovi\'c and Milena Radnovi\'c.
\newblock {\em Poncelet porisms and beyond. Integrable billiards, hyperelliptic
  Jacobians and pencils of quadrics}.
\newblock Birkh{\"a}user/Springer, Basel, 2011.

\bibitem{Fl09}
Leopold Flatto.
\newblock {\em Poncelet's theorem}.
\newblock Amer. Math. Soc., Providence, RI, 2009.

\bibitem{GSTV16}
Michael Gekhtman, Michael Shapiro, Serge Tabachnikov, and Alek Vainshtein.
\newblock Integrable cluster dynamics of directed networks and pentagram maps.
\newblock {\em Adv. Math.}, 300:390--450, 2016.

\bibitem{Gev19}
G\'abor G\'evay.
\newblock Resolvable configurations.
\newblock {\em Discrete Appl. Math.}, 266:319--330, 2019.

\bibitem{GePo23}
G{\'a}bor G{\'e}vay and Piotr Pokora.
\newblock Klein's arrangements of lines and conics.
\newblock {\em Beitr{\"a}ge zur Algebra und Geometrie}, 29, 2023.

\bibitem{GriHa78}
Phillip Griffiths and Joseph Harris.
\newblock On {C}ayley's explicit solution to {P}oncelet's porism.
\newblock {\em L'Enseignement Math{\'e}matique}, 1(2):31--40, 1978.

\bibitem{Gr00a}
Branko Gr{\"u}nbaum.
\newblock Astral ($n_4$) configurations.
\newblock {\em Geombinatorics}, 9:127--134, 2000.

\bibitem{GR00b}
Branko Gr{\"u}nbaum.
\newblock Which ($n_4$) configurations exist?
\newblock {\em Geombinatorics}, 9:164--169, 2000.

\bibitem{Gr09}
Branko Gr{\"u}nbaum.
\newblock {\em Configurations of {P}oints and {L}ines}.
\newblock Amer. Math. Soc., Providence, RI, 2009.

\bibitem{GrRi90}
Branko Gr{\"u}nbaum and John~F. Rigby.
\newblock The real configuration ($21_4$).
\newblock {\em J. Lond. Math. Soc.}, 41:336--346, 1990.

\bibitem{HaHu15}
Lorenz Halbeisen and Norbert Hungerb{\"u}hler.
\newblock A simple proof of {P}oncelet's theorem (on the occasion of its
  bicentennial).
\newblock {\em Amer. Math. Monthly}, 122(6):537--551, 2015.

\bibitem{Iz15}
Ivan Izmestiev.
\newblock A porism for cyclic quadrilaterals, butterfly theorems, and
  hyperbolic geometry.
\newblock {\em Amer. Math. Monthly,}, 122(5):467--475, 2015.

\bibitem{IzTa17}
Ivan Izmestiev and Serge Tabachnikov.
\newblock Ivory's theorem revisited.
\newblock {\em Journal of Integrable Systems}, 2:1--36, 2017.

\bibitem{Izo22}
Anton Izosimov.
\newblock The pentagram map, {P}oncelet polygons, and commuting difference
  operators.
\newblock {\em Compos. Math.}, 158:1084--1124, 2022.

\bibitem{Kl1878}
Felix Klein.
\newblock Ueber die {T}ransformationen siebenter {O}rdnung der elliptischen
  {F}unktionen.
\newblock {\em Math. Ann.}, 14(3):428--471, 1878.

\bibitem{Klu1898}
Leopold Klug.
\newblock {\em Die {C}onfiguration des {P}ascal'schen {S}echseckes im
  {A}llgemeinen und in 4 speciellen {F}{\"a}llen}.
\newblock Ajtai K. Albert, Kolozsv{\'a}r, 1898.

\bibitem{LaTa07}
Mark Levi and Serge Tabachnikov.
\newblock The {P}oncelet grid and billiards in ellipses.
\newblock {\em Amer. Math. Monthly}, 114:895--908, 2007.

\bibitem{OST10}
Valentin Ovsienko, Richard~Evan Schwartz, and Serge Tabachnikov.
\newblock The pentagram map: a discrete integrable system.
\newblock {\em Comm. Math. Phys.}, 299:409--446, 2010.

\bibitem{Pon1865}
Jean-Victor Poncelet.
\newblock {\em Trait\'e sur les propri\'et\'es projectives des figures}.
\newblock Bachelier; 2nd edition. Paris, Gauthier-Villars., 1865-66.

\bibitem{Rey1896}
Theodor Reye.
\newblock Beweis einiger {S}{\"a}tze von {C}hasles {\"u}ber konfokale
  {K}egelschnitte.
\newblock {\em Naturforschenden Gesellschaft in Z{\"u}rich}, 14(2):65--75,
  1896.

\bibitem{RG95}
J{\"u}rgen Richter-Gebert.
\newblock Mechanical theorem proving in projective geometry.
\newblock {\em Ann. Math. Artif. Intell.}, 13(1-2):139--172, 1995.

\bibitem{RG06}
J{\"u}rgen Richter-Gebert.
\newblock Meditations on {C}eva's theorem.
\newblock In {\em The Coxeter legacy}, pages 227--254. Amer. Math. Soc., 2006.

\bibitem{RG11}
J{\"u}rgen Richter-Gebert.
\newblock {\em Perspectives on {P}rojective {G}eometry}.
\newblock Springer, 2011.

\bibitem{RG24}
J{\"u}rgen Richter-Gebert.
\newblock Comments on the {C}hasles--{G}raves theorem.
\newblock in preparation, 2024.

\bibitem{RiKo99}
J{\"u}rgen Richter-Gebert and Ulrich Kortenkamp.
\newblock The {I}nteractive {G}eometry {S}oftware {C}inderella.
\newblock {\em Springer-Verlag, Berlin}, 1999.

\bibitem{Sch92}
Richard~Evan Schwartz.
\newblock The pentagram map.
\newblock {\em Experiment. Math}, 1:71--81, 1992.

\bibitem{Sch07}
Richard~Evan Schwartz.
\newblock The {P}oncelet grid.
\newblock {\em Adv. Geom}, 7(2):157--175, 2007.

\bibitem{SchTa10}
Richard~Evan Schwartz and Serge Tabachnikov.
\newblock Elementary surprises in projective geometry.
\newblock {\em Math. Intelligencer}, 32(3):31--34, 2010.

\bibitem{Sch15}
Richard~Evan Schwarz.
\newblock The pentagram integrals for poncelet families.
\newblock {\em J. Geom. Phys}, 87:432--449, 2015.

\bibitem{Sto91}
Jorge Stolfi.
\newblock {\em Oriented Projective Geometry: A Framework for Geometric
  Computations}.
\newblock Academic Press, Inc. San Diego., 1991.

\bibitem{Ta05}
Serge Tabachnikov.
\newblock {\em Geometry and {B}illiards}, volume~30 of {\em Student
  Mathematical Library}.
\newblock American Mathematical Society, Providence, 2005.

\bibitem{Ta16}
Serge Tabachnikov.
\newblock Skewers.
\newblock {\em Arnold Math. J.}, 2:171--193, 2016.

\bibitem{Ta19}
Serge Tabachnikov.
\newblock Kasner meets {P}oncelet.
\newblock {\em The Mathematical Intelligencer}, 41(4):56--59, 2019.

\end{thebibliography}
}
\newpage

\appendix 
\section{Proof of the core lemma}
This appendix is dedicated to the proof of the core technical statement{, Lemma~1}.
Recall that if we have a Poncelet chain with lines $l_0,l_1,l_2,\ldots$ tangent to a conic $\mathcal{X}$ then the intersections $l_i\wedge l_{i+k}$ all lie on a 
conic $\mathcal{C}_k$. As usual, by abuse of notation, we denote  the matrix defining a conic $\mathcal{C}$ also by the letter $\mathcal{C}$. Points and lines will be  identified with their homogeneneous coordinates.
We call two conics {\it in generic position} if they have four distinct points of intersection. In what follows we make the assumption that Poncelet chains are defined by pairs of conics that are in generic position. 

Poncelet chains that cyclically close up will always require inscribed and circumscribed conics in generic position.

\begin{figure}[h]
\begin{center}
\includegraphics[width=0.57\textwidth]{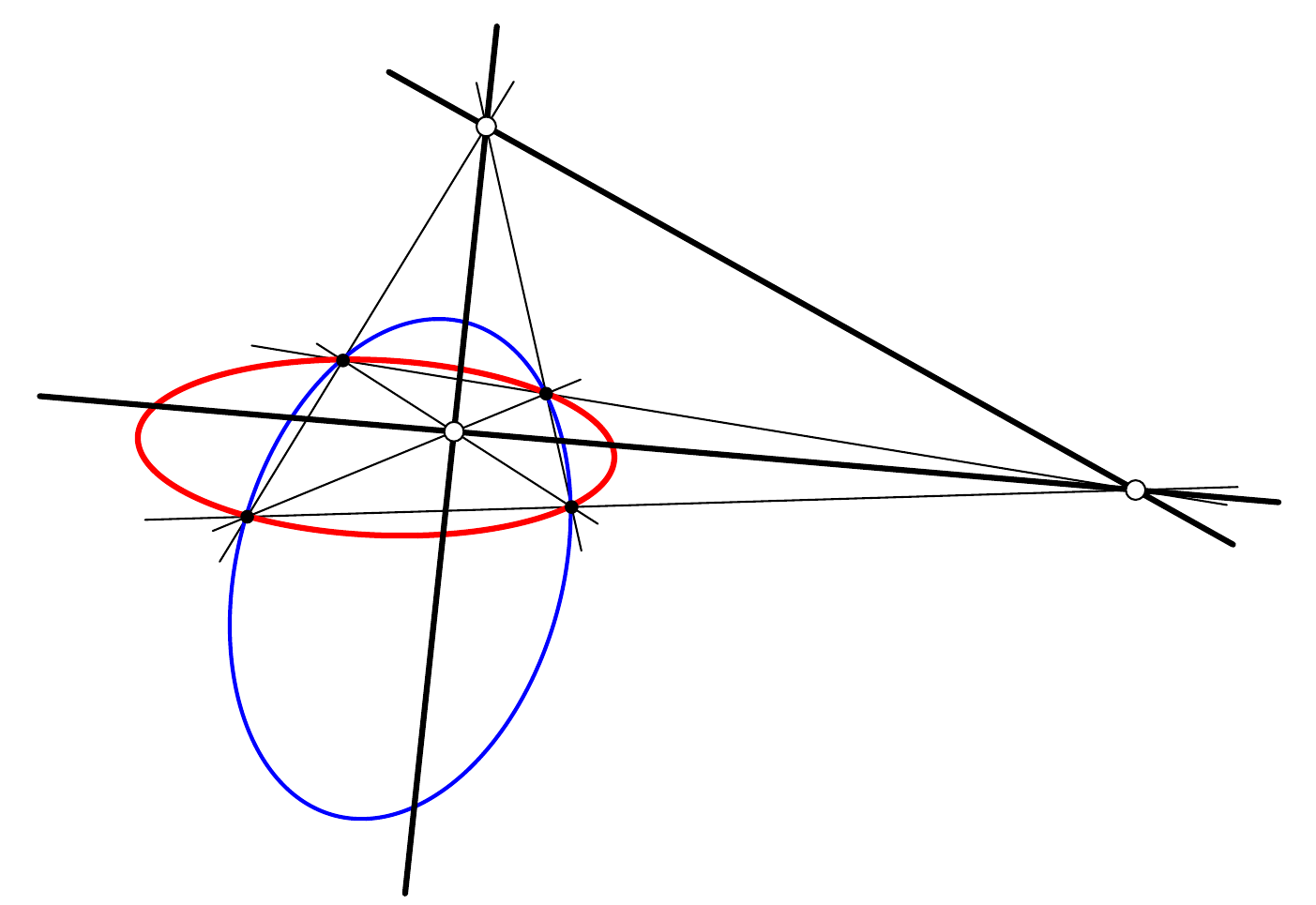}\quad
\includegraphics[width=0.4\textwidth]{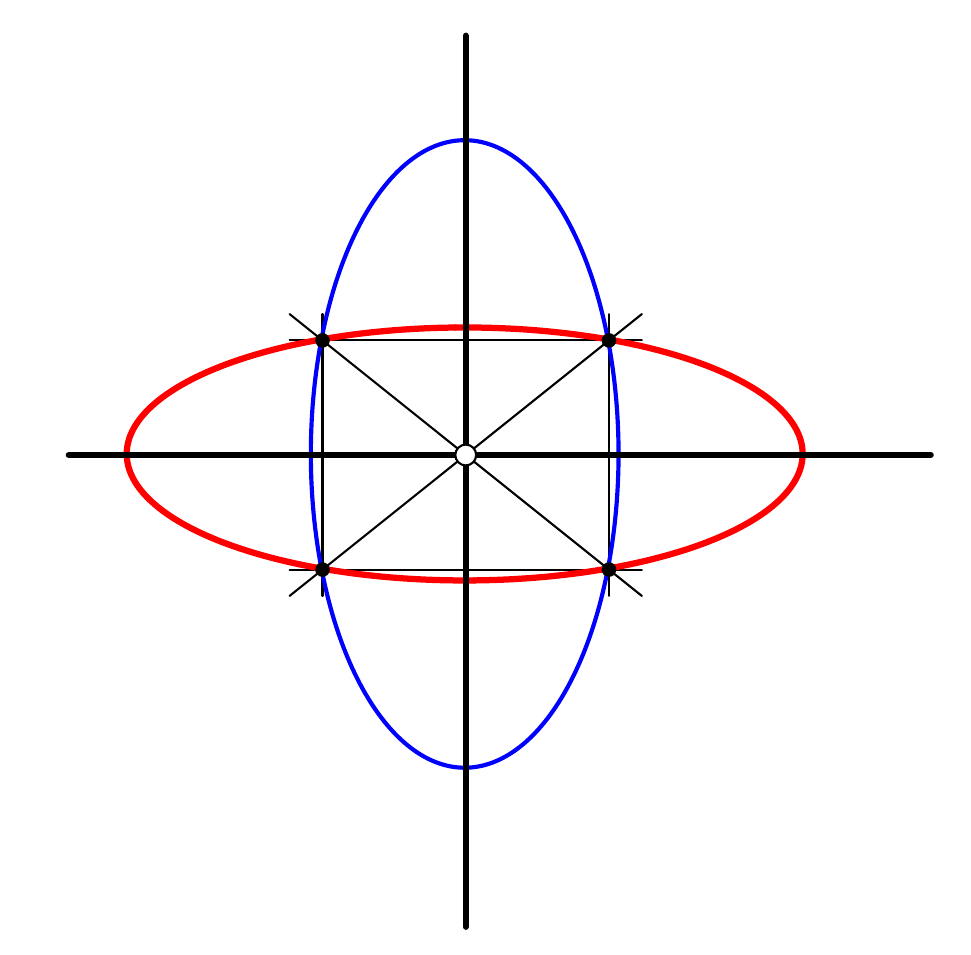}
\begin{picture}(0,0)

\end{picture}
\end{center}
\vskip-5mm
\captionof{figure}{Symmetrising a pair of conics with respect to the coordinate axes.} \label{fig:transform}
\end{figure}

If two conics are in generic position it is possible to apply  a (possibly complex) projective transformation that simultaneously diagonalises the two matrices. To see this consider the four points of intersection and map them to the vertices of the square $(\pm 1, \pm 1,1)$ (given in homogeneous coordinates). Both matrices can then be expressed as  linear combinations
of two arbitrary conics containing these four points (for instance $x^2+y^2-z^2=0$ and $x^2-y^2-z^2=0$) and therefore are in diagonal form.
Since invertible  diagonal matrices are closed under inversion and linear combination, all conics in a Poncelet grid can be diagonalised simultaneously.
We are heading for the following lemma.

\noindent
{\bf Lemma 1:\ }
{\it
Let $l_0,l_1,l_2,\ldots$ be the lines of a Poncelet chain
  tangent to a  
conic~$\mathcal{X}$, and let $a$, $b$ be such that
$l_0,\ l_a,\ l_b, \ l_{a+b}$ are pairwise distinct lines.
Let $\mathcal{B} = \mathcal{A}_a$ and 
$\mathcal{G} = \mathcal{A}_b$. Consider four points 
$
P=l_0\wedge l_a,\ 
P'=l_b\wedge l_{a+b},\ 
Q=l_0\wedge l_b,\ 
Q'=l_a\wedge l_{a+b}.\ 
$
Then the tangents 
$
\mathcal{B}\cdot P,\ 
\mathcal{B}\cdot P',\ 
\mathcal{G}\cdot Q,\ 
\mathcal{G}\cdot Q'
$ 
meet in a point.
}

It is adequate to speak of tangents here since $P$ and $P'$ lie on $\mathcal{B}$ and  $Q$ and $Q'$ lie on $\mathcal{G}$. Furthermore for our purposes it is no restriction to consider only situations in which the $l_0,\ l_a,\ l_b, \ l_{a+b}$ are all different (compare the proof of Theorem A).

\medskip

\subsection{The generic case}

In the situation of Lemma 1 the three conics $\mathcal{X}, \mathcal{B}, \mathcal{G}$ are codependent, since they are part of the same Poncelet grid. Hence we may apply a (potentially complex) projective transformation that maps all of them to diagonal matrices. Applying another (potentially complex) projective transformation we can scale the $x$ and $y$ axes and 
assume that $\mathcal{X}$ is the unit circle. By stereographic projection
we parametrise the points on $\mathcal{X}$ by  $(t^2-1,2t,t^2+1)$. 
Each parameter $t$ yields a point on the unit circle in homogeneous $\mathbb{RP}^2$ coordinates. We are missing one point $(1,0,1)$ on $\mathcal{X}$  
that may be associated with $t=\infty$. Replacing the coordinate $t$ by its negative $-t$ corresponds to a reflection in the $x$ axis. Replacing $t$ with its reciprocal $1/t$ corresponds to a reflection in the $y$ axis.

\noindent
{\it Remark:} An alternative approach would be to introduce homogeneous  $\mathbb{RP}^1$ coordinates and replace $t$ by a point $(t_1,t_2)$ on $\mathbb{RP}^1$ the stereographic projection would then  yield
coordinates  $(t_1^2-t_2^2,2t_1,t_1^2+t_2^2)$ in $\mathbb{RP}^2$. We intentionally do not take this approach to keep the amount of necessary variables low.
In that framework reflection on the $y$-axis correspond to swapping the roles of $t_1$ and~$t_2$.

The tangent of such a point to $\mathcal{X}$  has coordinates 
$\varphi(t):=(t^2-1,2t,-t^2-1)$.
Intersecting two such tangents by performing a vector product gives
$\varphi(s)\wedge \varphi(t)=2 (s - t) \left(-1 + s t, s + t, 1 + s t\right)$.
Dividing by the factor $2 (s - t)$ does not change the position of the points, but resolves the removable singularity when $s$ and $t$ coincide. We define an operation $\wedge(s,t):=\left(-1 + s t, s + t, 1 + s t\right)$ that creates the intersection of the two tangents associated with $s$ and $t$. This operator is obviously commutative in $s$ and $t$.

We now can assume that 
$
\varphi(s),
\varphi(t),
\varphi(u),
\varphi(v)
$ are four distinct tangents to $\mathcal{X}$ (they will play the roles of
$l_0,l_a,l_b,l_{a+b}$).
The points $P,P',Q,Q'$ in our lemma correspond to pairwise intersections of such points 
\[
P=\wedge(s,t),\ 
P'=\wedge(u,v),\ 
Q=\wedge(s,u),\ 
Q'=\wedge(t,v).
\]
Now we will calculate a conic $\mathcal{B}$ through $P$ and $P'$
and a conic $\mathcal{G}$ through $Q$ and $Q'$. Since our projective 
transformation allowed us to have all conics in diagonal form we can calculate the conics in the following way. If $P=(x,y,z)$ is represented by
homogeneous coordinates  we define its squared coordinates by
$P^2:=(x^2,y^2,z^2)$. We define $P'^2, Q^2, Q'^2$ analogously. $P$ lies on a conic $\mathcal{B}$ with diagonal entries $D=(a,b,c)$ if
and only if $\langle D,P^2\rangle=0$ with 
$\langle \ldots,\ldots\rangle$ representing the canonical inner product.
Thus we can calculate the diagonal entries of $\mathcal{B}$ by
$P^2 \times P'^2$. Similarly, we get the diagonal entries of 
$\mathcal{G}$ by
$Q^2 \times Q'^2$.

\medskip
At this point a little care is necessary to avoid degenerate calculations.
If $P^2 =\lambda P'^2$, then the exterior product results in the zero-vector and the 
corresponding conic is not defined. We call such a situation {\it special}.
We will discuss the exact geometric situation of {\it special} positions later.
For now we assume that we do not have a special situation and both conics 
$\mathcal{B}$ and $\mathcal{G}$  are properly defined.
Elementary calculations show that diagonal entries of the two matrices are 
\[
\begin{array}{l}
4 (s t - u v) (-1 +  s t u v),\\
 (1 + s t)^2 (u + v)^2 - (s + t)^2 (1 +  u v)^2,\\
  -(-1 + s t)^2 (u + v)^2 + (s + t)^2 (-1 + u v)^2\end{array}
\]
for $\mathcal{B}$ and
\[
\begin{array}{l}
4 (s u - t v) (-1 +  s t u v),\\ 
(1 + s u)^2 (t + v)^2 - (s + u)^2 (1 + t v)^2,\\
     -(-1 + s u)^2 (t + v)^2 + (s + u)^2 (-1 + t v)^2
\end{array}
\]
for  $\mathcal{G}$. Since the formulas in coordinate representation are not very insightful and all calculations are very elementary we from now on argue by using computer algebra.
We can compute the tangents at $P,P',Q,Q'$ by multiplying the points with the  matrix of the respective conic.
In a sense, the only important part here is that 
 those values can be calculated in a straightforward way from the corresponding values of $s,t,u,v$. 
Given the values of the conics and of the points we may easily 
compute the tangents and check if they turn out to be coincident.
It suffices to show that 
\[
\det(\mathcal{B}\cdot P,\mathcal{B}\cdot P',\mathcal{G}\cdot Q)=0
\quad \text{and}\quad
\det(\mathcal{B}\cdot P,\mathcal{B}\cdot P',\mathcal{G}\cdot Q')=0.
\]
This can be easily done by a computer algebra system. The following screenshot from a  Mathematica session serves as a witness:
\smallskip

\noindent
\includegraphics[width=.7\textwidth]{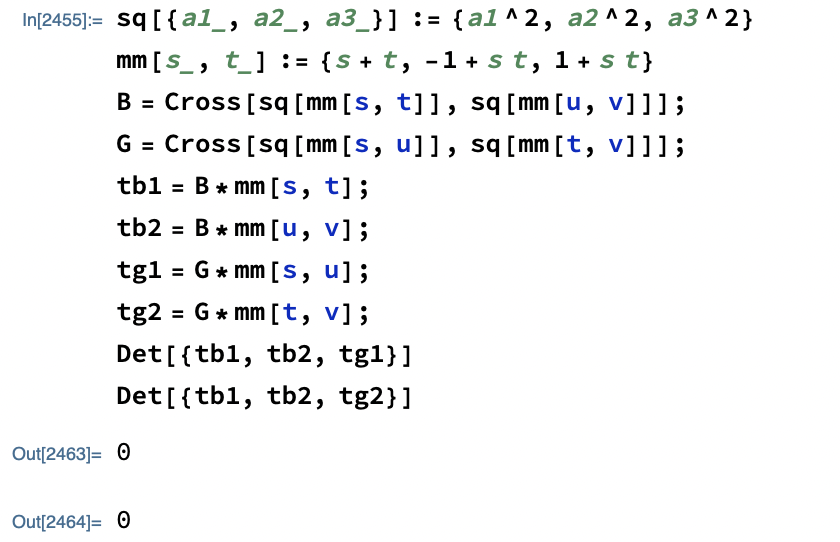}

\medskip
\subsection{The special cases}
Let us now come to the treatment of situations where the conics cannot be calculated by the above procedure since the points $P$ and $P'$ or
$Q$ and $Q'$ are are in special position to each other. In fact this is the situation for which the fact that the conics are taken from a Poncelet grid 
really becomes essential.

If $P=(x,y,1)$ and $P'=(x',y',1)$ are finite points given by homogeneous coordinates then they are in special position with resepct to each other of
$x=\pm x'$ and $y=\pm y'$ (exactly in that case we have $P^2=P'^2$. In other words, a finite point $P'$ is in special position to $P$ if it forms one of the corners of an axis-symmetric rectangle with initial vertex $P$.

Considering the fact that replacing $t$ by $-t$ (resp $1/t$) corresponds to reflection of $\varphi(t)$ in the $y$ or $x$ axis we see that we get 
points $\wedge(u,v)$ in special position to $\wedge(s,t)$ for the eight possible choices  in which  $\{u,v\}$ taken as a set equals one of the sets
\[
 \{s,t\}, \ 
 \{-s,-t\},\  
 \{1/s,1/t\},\  
 \{-1/s,-1/t\} 
 \]
The first case can be excluded since this leads to a pair of identical lines in our lemma which was forbidden by  our initial assumptions.
We will only discuss the second case, which leads to the two possibilities
$(u,v)=(-s,-t)$ and
$(u,v)=(-t,-s)$. 
The remaining cases can be treated by similar calculations.
In fact they could alternatively be achieved by suitable coordinate transformations.

\noindent\underline{\bf The case $\mathbf{u=-s}$ and $\mathbf{v=-t}$:}
In this case the tangents $\varphi(s)$ and $\varphi(u)$ are symmetric with respect to the $x$-axis. Their intersection $\wedge(s,u)$ lies on the $x$-axis. Similarly, $\wedge(t,v)$ lies on the $x$-axis. Thus the conic $\mathcal{G}$ is degenerate and forms a double line coinciding with the $x$-axis.
The tangents $\mathcal{G}\cdot Q$ and $\mathcal{G}\cdot Q'$ are the $x$-axis itself. 
Due to the symmetry of the entire configuration  the tangents
 $\mathcal{B}\cdot P$ and  $\mathcal{B}\cdot P'$
 intersect on the $x$-axis as well. This proves the lemma for that case.

\begin{figure}[ht]
\centering
\includegraphics[width=.48\textwidth]{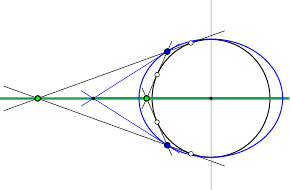}\quad\ \
\includegraphics[width=.45\textwidth]{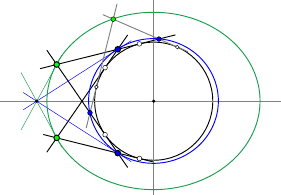}

\begin{picture}(0,0)
\put(-60,35){\footnotesize {${s}$}}
\put(-75,48){\footnotesize {${t}$}}
\put(-75,78){\footnotesize {${v{=}{-}t}$}}
\put(-60,89){\footnotesize {${u{=}-s}$}}
\put(-80,29){\footnotesize {$P$}}
\put(-80,94){\footnotesize {$P'$}}
\put(-150,70){\footnotesize {$Q$}}
\put(-97,70){\footnotesize {$Q'$}}
\put(90,35){\footnotesize {${s}$}}
\put(75,46){\footnotesize {${t}$}}
\put(68,66){\footnotesize {${x}$}}
\put(75,78){\footnotesize {${u{=}{-}t}$}}
\put(87,87){\footnotesize {${v{=}-s}$}}
\put(115,97){\footnotesize {$y$}}
\put(75,97){\footnotesize {$P'$}}
\put(73,25){\footnotesize {$P$}}
\put(40,30){\footnotesize {$Q$}}
\put(42,90){\footnotesize {$Q'$}}
\end{picture}
\caption{The situation in the two special cases.}	
\label{fig:CGT}
\end{figure}
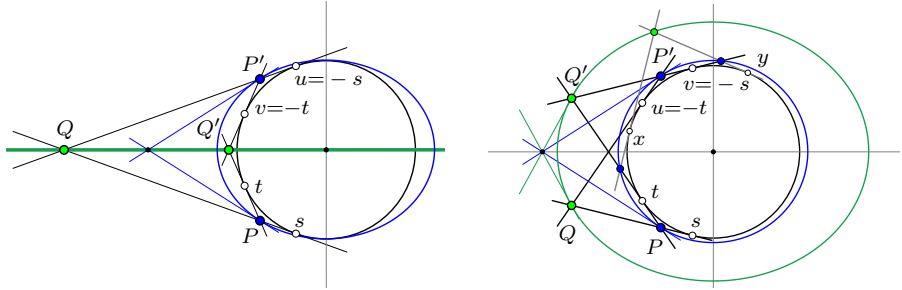

\noindent\underline{\bf The case $\mathbf{u=-t}$ and $\mathbf{v=-s}$:}
Now we come to the most interesting case that requires the fact that we come from a Poncelet polygon.
In this case we  have
$P=\wedge(s,t),\ 
P'=\wedge(-t,-s),\ 
Q=\wedge(s,-t),\ 
Q'=\wedge(t,-s).$
By the symmetry of the situation the tangents at $\mathcal{B}\cdot P$ and $\mathcal{B}\cdot P'$ will intersect in the $x$ axis. 
Let us assume that the point of intersection has coordinates $A=(0,a,1)$.
The tangents
 $\mathcal{G}\cdot Q$ and $\mathcal{G}\cdot Q'$ will also intersect on the $x$-axis. However, it is not a priori clear that they will intersect in the same point.
The point $A$ can be freely chosen on the $x$-axis and from this point and 
$s$ and $t$ we can reconstruct the entirity of the (diagonal) matrix $\mathcal{B}$ in a unique way.
The diagonal entries are
\[(s + t)^2 (1 + s t), (1 + a + (-1 + a) s t) (-1 + s^2 t^2), -a (s + 
   t)^2 (-1 + s t).\]

\noindent
It is straightforward to check that this conic $\mathcal{B}$ passes through
$P$ and that the tangent $\mathcal{B}\cdot P$ intersects the $y$-axis at point $(1,0,a)$. It is demonstrated by the following Mathematica session.

\noindent
\includegraphics[width=.9\textwidth]{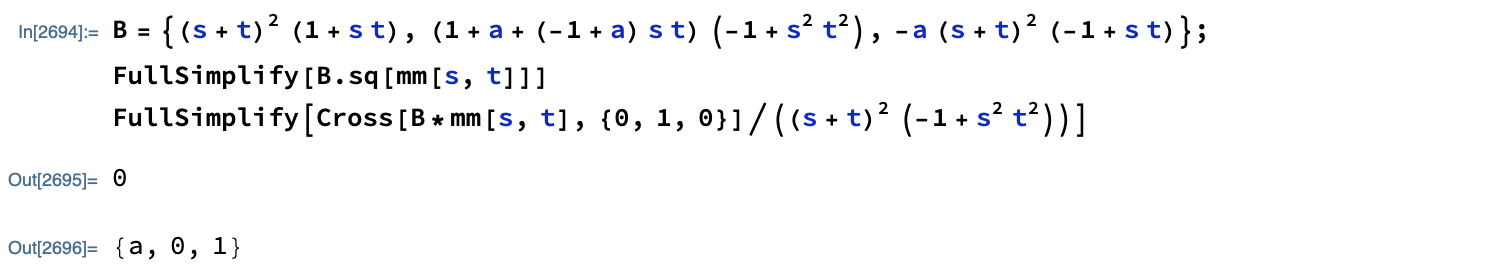}

For every choice of $a$ we get a corresponding choice of the conic $\mathcal{B}$. We now extend the Poncelet chains $(s,t)$ and $(-t,-s)$
by one further point using the coordinates of $\mathcal{B}$.
We call the resulting coordinates for the two points $x$ and $y$ respectively. 
Again these points can be calculated by elementary arithmetic operations from the previous data.

In terms of our original lines of the Poncelet chains 
the sequences $s,t,x$ and $-t,-s,y$ correspond to the lines:
\[
l_0,l_a,l_{2a}\quad\text{and}\quad l_b,l_{a+b},l_{2a+b},
\]
respectively. The only thing left to show is to see that these values are compatible with the conic $\mathcal{G}$. In other words we must show.
that $\wedge(x,y)$ is on $\mathcal{G}$. The following computer algebra session witnesses that.

\noindent
\includegraphics[width=.9\textwidth]{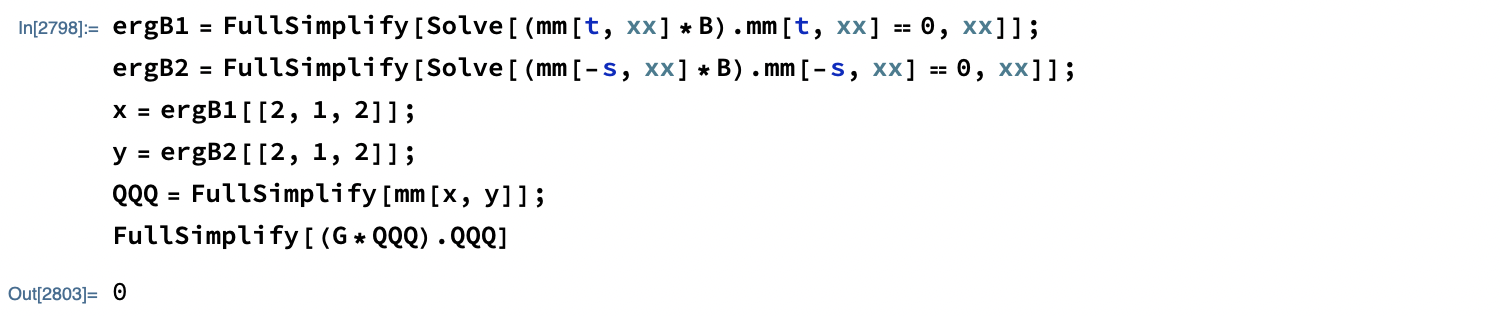}

\noindent
Although, in the first two lines a Solve operator is applied, under our setup
it results in a rational expression. The  results are extracted in lines 3 and 4.
Obtaining a $0$ at the end finishes the Proof of Lemma 1.

\medskip

\noindent
{\footnotesize
{\sc Department of Mathematics \& Statistics, University of Alaska Fairbanks, USA}\\
Email address: {\tt lwberman@alaska.edu}

\noindent
{\sc Bolyai Institute, University of Szeged, Hungary}\\
Email address: {\tt gevay@math.u-szeged.hu}

\noindent
{\sc Department of Mathematics, Technical University of Munich, Germany}\\
Email address: {\tt richter@tum.de}

\noindent
{\sc Department of Mathematics, Penn State University, USA}\\
Email address: {\tt tabachni@math.psu.edu}

}
\end{document}